\documentclass[10pt]{article}

\usepackage{url}
\usepackage[numbers]{natbib}
\usepackage{mathtools}
\usepackage{amsmath}
\usepackage{amssymb}
\usepackage{mathabx}
\usepackage{amsthm}
\usepackage{empheq}
\usepackage{latexsym}
\usepackage{enumitem}
\usepackage{eurosym}
\usepackage{dsfont}
\usepackage{appendix}
\usepackage{color} 
\usepackage[unicode]{hyperref}
\usepackage{frcursive}
\usepackage[utf8]{inputenc}
\usepackage[T1]{fontenc}
\usepackage{geometry}
\usepackage{multirow}
\usepackage{todonotes}
\usepackage{lmodern}
\usepackage{anyfontsize}
\usepackage{stmaryrd}
\usepackage{cases}
\usepackage{calc}
\usepackage{pdfpages}

\definecolor{red}{rgb}{0.7,0.15,0.15}
\definecolor{green}{rgb}{0,0.5,0}
\definecolor{blue}{rgb}{0,0,0.7}
\hypersetup{colorlinks, linkcolor={red},citecolor={green}, urlcolor={blue}}
			
\makeatletter \@addtoreset{equation}{section}

\usepackage[justification=centering]{caption}
\newtheorem{theorem}{Theorem}[section]

\newtheorem{lemma}[theorem]{Lemma}
\newtheorem{proposition}[theorem]{Proposition}

\newtheorem{definition}[theorem]{Definition}
\newtheorem{remark}[theorem]{Remark}

\DeclareUnicodeCharacter{014D}{\=o}
\def \E{\mathbb{E}}
\def \F{\mathbb{F}}

\def \H{\mathbb{H}}

\def \M{\mathbb{M}}

\def \P{\mathbb{P}}

\def \R{\mathbb{R}}
\def \S{\mathbb{S}}

\def \U{\mathbb{U}}
\def \V{\mathbb{V}}

\def\Ac{{\cal A}}
\def\Bc{{\cal B}}
\def\Cc{{\cal C}}

\def\Fc{{\cal F}}

\def\Hc{{\cal H}}

\def\Lc{{\cal L}}
\def\Mc{{\cal M}}

\def\Pc{{\cal P}}
\def\Qc{{\cal Q}}

\def\Uc{{\cal U}}
\def\Vc{{\cal V}}

\def\eps{\varepsilon}
\def\erm{\mathrm{e}}
\def\drm{\mathrm{d}}

\title{Mean--field moral hazard for optimal energy demand response management\footnote{The authors gratefully acknowledge the support of the ANR project PACMAN ANR-16-CE05-0027.} }

\author{Romuald {\sc \'Elie} \footnote{LAMA, Universit\'e Gustave Eiffel, France, romuald.elie@univ-mlv.fr} \and Emma {\sc Hubert} \footnote{LAMA, Universit\'e Gustave Eiffel, France, emma.hubert@univ-paris-est.fr} \and Thibaut {\sc Mastrolia} \footnote{CMAP, \'Ecole Polytechnique, Palaiseau, France, thibaut.mastrolia@polytechnique.edu.} \and Dylan {\sc Possama\"{i}} \footnote{Columbia University, IEOR department, USA, dp2917@columbia.edu}}
 
\date{\today}

\begin{document}

\maketitle

\begin{abstract}
We study the problem of demand response contracts in electricity markets by quantifying the impact of considering a continuum of consumers with mean--field interaction, whose consumption is impacted by a common noise. We formulate the problem as a Principal--Agent problem with moral hazard in which the Principal -- \textit{she} -- is an electricity producer who observes continuously the consumption of a continuum of risk--averse consumers, and designs contracts in order to reduce her production costs. More precisely, the producer incentivises each consumer to reduce the average and the volatility of his consumption in different usages, without observing the efforts he makes. We prove that the producer can benefit from considering the continuum of consumers by indexing contracts on the consumption of one Agent and aggregate consumption statistics from the distribution of the entire population of consumers. In the case of linear energy valuation, we provide closed--form expression for this new type of optimal contracts that maximises the utility of the producer. In most cases, we show that this new type of contracts allows the Principal to choose the risks she wants to bear, and to reduce the problem at hand to an uncorrelated one.
\medskip


\medskip


\medskip

\end{abstract}

\section{Introduction}

Consumption management, in particular energy efficiency and demand response, is undeniably one of the most important challenge in the energy sector, all the more since it is one of the basis to guarantee the fulfilment of EU objectives in terms of green energy and emission reduction. Indeed, on the one hand, electric energy cannot easily be stored, which is why utilities have always tried to directly match supply and demand by limiting the production rate of their power plants, by commissioning or decommissioning generators, or by importing electricity from other utilities. However, what can be achieved on the supply side is limited: some production units can take a long time to reach full capacity, some units can be very expensive to operate, or demand can sometimes exceed the capacity of all the available power plants combined. Demand response seeks to adjust {\it electricity demand} rather than {\it supply}. On the other hand, the business--as--usual criticism of renewable energy is based on a presumed inadequacy with regard to demand. In fact, adequacy could be addressed: the variability of renewable energy production could be monitored through consumption management. The development of smart meters, for example, is a big step in this direction, but the quality of their usage will be crucial. In addition, optimised energy tariffs could allow a better sharing of risks between producers and consumers. This is where most of the potential improvement in energy consumption lies. 

\medskip

Demand response is a change in the energy consumption of an electricity utility consumer to better match demand with supply. Utilities can report demand requests to their customers in a variety of ways, including contracts under which the consumer normally receives electricity at a lower price than the standard rate in exchange for significantly higher prices at certain peak periods chosen by the generator. The customer can then adjust his consumption by postponing tasks that are costly in terms of electricity, or by paying a high price. Many experiments have been developed to accurately assess the benefits of demand response programs on consumption, see the references cited by \citet*{aid2018optimal} for more details. In \cite{aid2018optimal}, the authors focus in particular on the large scale demand response experiment of Low Carbon London Pricing Trial. Nevertheless, demand response mechanisms face challenges which have to be tackled before one can claim that they provide a level of flexibility comparable to thermal power plants. In particular, demand response programs exhibit a substantial variance in the response of the consumer to price signal. This leads to uncertainty on the total response of the solicited population. This large variance, called the responsiveness effect, is a significant gap in demand response programs. 

\medskip 

In this perspective, A\"id, Possama\"i and Touzi formulate in \cite{aid2018optimal} the mechanism of demand response programs during a fixed finite period as a continuous--time Principal--Agent problem with moral hazard. The Agent -- \textit{He} -- is a risk--averse CARA consumer who has a baseline consumption of electricity, calibrated in view of his lifestyle (his preferences, the size and thermal insulation of his housing,...) and the electricity price. He can deviate from his original demand, by reducing both the average level of his consumption and its volatility if he has incentives to do it. The reduction of the consumption volatility can also be seen as an increase of consumer's responsiveness. His efforts are costly, and may depend on the nature of the corresponding usage of electricity (heating or air conditioning, lightning, television, washing machine, computers...). The Principal -- \textit{She} -- is a risk--averse CARA producer who has to satisfy the random electricity demand of the consumer. She is subject to energy generation costs, and also to consumption volatility costs, which account for the limited flexibility of electricity production. As always in contracting theory, she wants to find an optimal way to encourage the Agent to reduce the mean and the volatility of his consumption, in order to minimise the costs she faces. {\color{black}However, in a moral hazard framework, she cannot offer a compensation directly related to Agent's effort, because she cannot observe the latter, and she only observes Agent's deviation consumption in continuous--time.} Thanks to the recent works of \citet*{cvitanic2017moral, cvitanic2018dynamic}, the optimal contract in this framework consists in a sum of a deterministic payment that depends on the duration of the demand response, a linear payment on each infinitesimal deviation, and a linear payment on the realised squared volatility. The main results of this model are that optimal contracting allows the system to bear more risk as the resulting volatility may increase, and that the control of the consumption volatility can lead to a significant increase of responsiveness. 

\medskip

In our paper, we extend the framework defined in \cite{aid2018optimal} to a model with a continuum of Agents with mean--field interaction, whose consumption is impacted by a common noise. Until now, continuous--time Principal--Agent models with many agents were restricted to drift control, see \citet*{elie2019contracting}, \citet*{koo2008optimal}, or \citet*{goukasian2010optimal} for a finite number of interacting Agents, or \citet*{elie2019tale}, and more recently \citet*{carmona2018finite} for a continuum of Agents without any correlation\footnote{There was a much larger number of attempts in static or discrete--time settings, but always with finitely many agents. We refer the reader to the most notable works in this direction, among which we can find \citet*{holmstrom1982moral, mookherjee1984optimal, green1983comparison, harris1982asymmetric, nalebuff1983prizes, demski1984optimal, Itoh2004moral, rey-biel2008inequity, bartling2010intensity, grund2005envy, demougin2006output} as well as \citet*{neilson2010piece} or \citet*{kragl2015group}.}. The first motivation behind our extension is that it seems more realistic that the producer does not consider the agents one by one, but wants to optimise the consumption in mean for a large pool of similar and correlated consumers. Therefore, we consider that each consumer can deviate from his baseline consumption, and each consumption is subject to a common noise, which accounts for the common random environment where all the consumers evolve. In regard to energy, this common random environment can for instance be interpreted as consequences of meteorological variations. Hence, the problem of the producer is to manage a pool of similar consumers, whose consumption is subject to the same climate hazards. The second motivation is that considering a continuum of consumers with mean--field interaction can be profitable for the producer, as she has access to more information by observing all consumption profiles. Apart from the mean--field aspect and the correlation between Agents through the common noise, the problem formulation is voluntarily as close as possible to that of \cite{aid2018optimal}, and is developed in Section \ref{sec:formulation}.
{\color{black} Nevertheless, due to the complexity induced by the presence of a continuum of consumers in interaction, we do not rely on the results of \cite{aid2018optimal}, besides some straightforward algebra, and need to develop new techniques. More precisely, we simply adopt a similar framework, and use the calibrated parameters for our numerical part, while our technical results are inspired by the general framework developed by \citet*{cvitanic2018dynamic}, which we extend to a continuum of Agents with mean--field interactions.}

\medskip

We work under a classical mean--field framework, where all agents are identical: all consumers have the same characteristics, the same costs, the same risk--aversion parameter. This assumption is justified for a major electricity producer or provider, who has a sufficiently large number of similar consumers. It allows us to restrict the analysis to a representative Agent, who is a single consumer identical to others, and too small to have an impact on the global consumption. Moreover, this framework prevents us from indexing the compensation of a consumer on a particular consumption profile, except his own. This is indeed the case in the energy sector: the General Data Protection Regulation specifies that the Energy Distribution Organising Authorities (EDOA), the licensing authorities, have only the right to access aggregated (and therefore anonymous) data on electricity production and consumption in a given territory\footnote{Decree No. 2017-948 of 10 May 2017 on the procedures for making electricity and gas consumption data available to consumers and Decree No. 2017-976 of 10 May 2017 on the procedures for consumers to access electricity or natural gas consumption data and for suppliers to make such data available.}. Moreover, in France, a report\footnote{\textit{Pack de conformit\'e pour les Compteurs communicants} [Compliance package for communicating meters], Commission Nationale de l'Informatique et des Libert\'es [National Commission for Information Technology and Liberties].} stipulates that a provider or social landlord may use a person's consumption data to compile statistics, if the data are anonymous or aggregated and therefore do not allow the identification of a physical person.

\medskip

The goal of this paper is thus to find a way for the Principal to benefit from dealing with a continuum of consumers. The idea we develop is to add a component to the relevant contract proposed by the Principal to an Agent in models with volatility control, as in \cite{aid2018optimal, cvitanic2017moral, cvitanic2018dynamic}, which consists in two parts:
\begin{enumerate}[label=$(\roman*)$]
    \item one is indexed on the process controlled by the Agent, to incentivise him to make an effort on the average of his process;
    \item the second is indexed on the quadratic variation of the process controlled by the Agent, and incentivises him to make an effort on the volatility of his process. 
\end{enumerate}
Indeed, in our framework and taking into account the development of smart--meters, the Principal has access to a large quantity of anonymous consumption profiles. Therefore, she can compute empirical statistics from these data points. In particular, she can approximate the conditional law, with respect to the common noise, of the deviation consumption of the pool of consumers she manages. Indeed, this law is the limit of the empirical distribution of $N$--Agent's deviation consumption, and is conditional because of the correlation of the consumption by the common noise. Hence, she can for instance design a new contract in order to penalise a consumer who makes less effort, on average, than the rest of the pool, or to reward him if he makes more effort. The corresponding penalisation/reward is paid at a future time fixed in the contract. This approach is motivated by the recent development of applications that make it now possible to compare one's own consumption with that of similar households, or even with the lowest consuming households\footnote{see for example the phone application \textit{EDF \& Moi} by the french electricity producer and provider \textit{Electricité de France}, or the website of Synergy, an Australian producer and provider.}. Moreover, results of the study by \citet*{dolan2015neighbors} on energy efficiency have shown that comparing energy usage with similar households or providing financial incentives can lead to an average reduction in household energy usage of 7\%. They found on the one hand that the communication of the average consumption incentivises people to reduce their own consumption, and on the other hand that financial compensations are an efficient way to reduce consumption. Towards this objective, we study a new class of contracts, adding to the relevant contract for drift and volatility control, a component indexing the contract on the law of the deviation consumption of other consumers. Section \ref{sec:consumer_problem} is devoted to the intuition leading to this new type of contracts and the formal proofs are {\color{black} postponed} in Appendix \ref{sec:technical_proofs}. This section also investigates the representative consumer's optimal response to the proposed contract as well as the resulting mean--field equilibrium between consumers. 

\medskip

The resolution of the producer's problem is discussed in Section \ref{sec:principal_problem}. Based on the results of Section \ref{sec:consumer_problem}, the Principal's problem is reduced to a McKean--Vlasov problem, because she impacts through the contract, and so somehow indirectly controls, the efforts of the Agents as well as the conditional law of the consumption deviation of the Agents. The intuition for the Principal's problem in this mean--field framework comes from the $N$--Agents case. Following the reasoning of \'Elie, Mastrolia and Possama\"i in \cite{elie2019tale}, \textit{without common noise and for {\rm CARA} utility function}, the Principal becomes risk--neutral in the limit when $N \longrightarrow + \infty$, by classical propagation of chaos arguments. However, in our framework, \textit{with common noise}, the consumption deviations of the Agents become asymptotically independent, conditionally to the common noise. Therefore, a risk--averse Principal does not become risk--neutral in our case, and remains impacted by the residual risk arising from the common noise. Nevertheless, in addition to the case of a risk--averse Principal, we will study the case where her risk--aversion tends to zero, \textit{i.e.}, the case of a risk--neutral Principal. The study of the first--best case, when there is no moral hazard and thus the Principal can directly choose the efforts of the consumers, is {\color{black} postponed} to Appendix \ref{sec:firstbest}. 

\medskip

The main point of this work is to prove that the new contracts we developed are more profitable than the traditional ones. Section \ref{sec:classical_contract} is devoted to this comparison, by implementing classical contracts in our framework. In order to have closed--form solutions, we focus on a particular case: the linear energy value discrepancy case. The energy value discrepancy is the difference between a consumer's preference toward his deviation consumption and the production costs of this deviation. If this difference is positive, this means that the energy is more valuable for the consumer than it is costly for the producer. Conversely, if the difference is negative, a decrease in consumption will have less negative effects on a consumer's welfare than positive effects on the producer's savings. If this difference is linear with respect to the deviation consumption, it is easy to show that the utility of the producer is increased by the use of our new contracts. Moreover, in most cases, these new contracts induce more efforts from the consumers to reduce the average level of their consumption and with less volatility. More precisely, by using the values of the parameters calibrated in \cite{aid2018optimal}, we can see an increase:
\begin{itemize}
    \item in the utility of the Principal up to $50\%$ (respectively $15\%$);
    \item in the effort to reduce the average consumption up to $50\%$ (resp. $30\%$);
    \item in the variance responsiveness up to $4\%$ (resp. $3\%$);
\end{itemize}
depending on the correlation with the common noise, for a risk--neutral Principal (resp. for a Principal with the risk--aversion calibrated in \cite{aid2018optimal}). Moreover, the greater the variance explained by the common noise, the more significant the results are. Therefore, these new contracts could improve demand response during periods where consumption is strongly affected by weather conditions, for example in winter, when the risk of electricity blackouts is high, and thus demand response is more than needed.

\medskip

Throughout this work, we consider a Principal who can not observe the common noise, or at least that there exist some regulatory rules preventing her from using the common noise directly in the contract. This hypothesis is relatively well established in the field of energy consumption. Although some electricity suppliers offer different prices depending on the day or time of consumption (peak--period, off--peak period...), called \textit{time--variant pricing}\footnote{see for example the proposal of the Environmental Defense Fund in \url{https://www.edf.org/sites/default/files/time-variant_pricing_fact_sheet_-_april_2015.pdf}}, these tariff offers correspond more to an indirect indexation on the weather through the spot price of electricity. We find that the indexation of the contract on others is another way to indirectly index the contract on weather. Indeed, it allows the Principal to divide the deviation consumption in two parts: the part \textit{actually controlled} by the Agent, corresponding to the deviation corrected for climate hazards, and the common noise. Hence, she can offer a compensation indexed on the \textit{really controlled deviation} to encourage the Agent for making effort. If she is risk--averse, she can add to this contract a part indexed on others, which is in fact an indexation on the common noise, to share the remaining risk, even if regulatory rules prevent her from using the common noise directly in the contract. Therefore, in the case where the Principal is allowed to index the contract on the common noise, studied in Section \ref{sec:observable_common_noise}, we obtain the same form of contracts. We thus conclude that contracting on the conditional law or on the common noise is strictly equivalent.

\medskip

In addition to the practical contributions of our model, we develop along the way new technical results. As explained in \cite{cvitanic2017moral, cvitanic2018dynamic}, moral hazard problems in continuous--time in which the Agent controls the volatility of the output are notoriously harder to study. As such, and as far as we know, our model is the first one which combines a multi--agent setting (actually continuum of agents with mean--field interaction), with both drift and volatility controlled. Moreover and unlike in \cite{elie2019tale}, or \cite{carmona2018finite}, we also consider a setting with common noise. Though natural to model electricity consumers, the common noise induces a wide range of additional mathematical difficulties, which have only been recently addressed, see \citet*{carmona2016meana}, albeit only for pure mean--field games, without an additional Stackelberg on top, as in our setting. In particular, and again unlike in \cite{elie2019tale} or \cite{carmona2018finite}, the simplifying class of contracts we derive in Section \ref{sec:consumer_problem} cannot be proved almost by definition to be without loss of generality, and we need to use generalised notions of mean--field second--order BSDEs to obtain this fundamental result. Moreover, we emphasise that even though we work with a specific model, the {\it modus operandi} we develop in this paper can readily be extended to general moral hazard problems with a continuum of Agents with mean--field interaction and common noise. This is one of the reasons why we have made specific efforts to ensure that all our statements and definitions are completely mathematically rigorous, in particular the definition of the weak formulation for mean--field games with common noise and optimal control problems of McKean--Vlasov stochastic differential equations, as these two tools are the fundamental cornerstones of our approach. We believe that this will prove useful for other applications and generalisations of our techniques.

\medskip
In summary, this paper is organised as follows. Section \ref{sec:formulation} describes the model. Section \ref{sec:consumer_problem} provides the class of optimal contracts, and solves the representative consumer's problem, as well as the mean--field equilibrium. Section \ref{sec:principal_problem} is devoted to solve the Principal's problem under moral hazard with the new contracts. Section \ref{sec:classical_contract} allows the comparison of utilities and efforts with the case where only classical contracts are offered to the consumers. Section \ref{sec:observable_common_noise} provides the results when the contracts can be indexed directly on the common noise, and we investigate in Section \ref{sec:firstbest} the first--best problem, as a benchmark in which the Principal can directly monitor the efforts of the Agents. Section \ref{sec:conclusion} concludes.

\vspace{0.5em}
{\bf Notations.}  Throughout this paper, $T>0$ denotes some maturity fixed in the contract. We set $d$ a positive integer representing the number of different usages of electricity. Let $\mathbb{N}^\star :=\mathbb{N}\setminus\{0\}$. Throughout this paper, for every $d$--dimensional vector $b$ with $d \in \mathbb{N}^\star $, we denote by $b^{1},\ldots,b^{p}$ its coordinates, for any $1\leq i\leq d$. Let ${\bf 0}_d$ and ${\bf 1}_d$ be the vectors of size $d$ whose coordinates are all equal to respectively $0$ and $1$. For any $(\ell,c)\in\mathbb N^\star \times\mathbb N^\star $, we identify $\mathbb R^{\ell\times c}$ with the space of $\ell\times c$ matrices with real entries. Elements of the matrix $M\in\mathbb R^{\ell\times c}$ will be denoted by $(M^{i,j})_{1\leq i\leq \ell,\; 1\leq j\leq c}$, and the transpose of $M$ will be denoted by $M^\top\in\mathbb R^{c\times\ell}$. For any positive integer $n$ and for $(\alpha,\beta) \in \R^n\times\R^n$ we also denote by $\alpha \cdot \beta$ the usual inner product, with associated norm $\| \cdot \|$.

\medskip

For any positive integer $n$, and any $t \in [0,T]$, let $\Cc_t^n := \Cc ( [0,t], \R^n)$ denote the set of continuous functions from $[0,t]$ to $\R^n$, and $\Cc_t := \Cc_t^1$. On $\Cc ( [0,t], \R^n)$, define for $s \in [0,t]$ the evaluation mappings $\pi_s$ by $\pi_s(x) = x_s$ and the truncated supremum norms $\| \cdot \|_s$ by
\begin{align*}
  \| x \|_s = \sup_{u \in [0,s]} \|x_u\|, \; \text{ for } s \in [0,t].
\end{align*}
Unless otherwise stated, $\Cc_T^n$ is endowed with the norm $\| \cdot \|_T$. 

\medskip

For a measurable space $(\Omega, \Fc)$, let $\Pc(\Omega)$ denote the set of probability measures on $(\Omega, \Fc)$. In particular, for $t \in [0,T]$, we denote by $\Pc (\Cc_t^n)$ the set of all probability measures on $\Cc([0,t],\R^n)$. For $\mu \in \Pc (\Cc_T^n)$, let $\mu_t \in \Pc (\Cc_t^n)$ denote the image of $\mu$ under $\pi_{\cdot \wedge t}$. For $p \geq 0$ and a separable metric space $(E, \ell)$, let $\Pc^p(E)$ denote the set of $\mu \in \Pc(E)$ with $\int_E \ell^p (x, x^{\circ}) \mu ( \drm x) < + \infty$ for some (and thus
for any) $x^{\circ} \in E$. For $p \geq 1$ and $\mu, \nu \in \Pc^p(E)$, let $\ell_{E,p}$ denote the $p$--Wasserstein distance, given by
\begin{align*}
    \ell_{E,p} (\mu, \nu) := \inf \bigg\{ \bigg( \int_{E \times E}  \ell^p (x, y) \gamma (\drm x, \drm y) \bigg)^{1/p} : \gamma \in \Pc (E \times E) \text{ has marginals } \mu, \nu \bigg\}.
\end{align*}
The space $\Pc^p(E)$ is equipped with the metric $\ell_{E,p}$, and $\Pc(E)$ has the topology of weak convergence. Both are equipped with the corresponding Borel $\sigma$--fields, which coincides with the $\sigma$--field generated by the mappings $\mu \in \Pc^p(E)$ (resp. $\Pc(E)$)  $\longrightarrow \mu(F)$, $F$ being any Borel subset of $E$.

\section{Problem formulation}\label{sec:formulation}

\subsection{Informal description}\label{sec:informal}

\medskip

We work under a classical mean--field framework, where all agents are identical: all consumers have the same characteristics, the same costs, the same risk--aversion parameter... We thus restrict our study to a representative Agent, who is a single consumer, identical to a pool of others, and too small to impact the global consumption. In order to properly define the mathematical set up of our problem, we need a process representing the deviation consumption of the representative Agent, driven by an idiosyncratic noise for each usage ($d$--dimensional) and a common noise ($1$--dimensional). The Agent controls this process by choosing a pair $\nu := (\alpha, \beta)$, where $\alpha$ and $\beta$ are respectively $A$-- and $B$--valued, see \eqref{eq:def_A_B} below for the precise definition of $A$ and $B$. We denote for simplicity $U:= A \times B$. More specifically, $\alpha$ represents the effort of the consumer to reduce the nominal level of consumption and $\beta$ is the effort to reduce the variability of his consumption for each usage of electricity. We emphasise that $\alpha$ and $\beta$ are $d$--dimensional vectors, thus capturing the differentiation between different usages, e.g. refrigerator, heating or air conditioning, lightning, television, washing machine, computers... The set of admissible efforts, denoted by $\Uc = \mathcal{A} \times \mathcal{B}$, will be defined rigorously in the next subsection. We also define a vector $\sigma\in(0,+\infty)^d$ representing the variability of the Agent's consumption when he does not make any effort, and a constant $\sigma^{\circ} \in \R_+$ representing the correlation with the common noise. Therefore, for a chosen control $\nu = (\alpha, \beta) \in \Uc$, the Agent's control process can be written informally, for any $t\in[0,T]$, as
\begin{align}\label{eq:informal}
    X_t &= x_0 - \int_0^t \alpha_s \cdot {\bf 1}_d \drm s + \int_0^t  \sigma (\beta_s) \cdot \drm W_s + \int_0^t \sigma^{\circ} \drm W^{\circ}_s,\;
    \text{with} \; \sigma (b) := \big( \sigma^{1}  \sqrt{b^1},\dots, \sigma^{d}  \sqrt{b^d} \big)^\top, \text{ for } b\in(0,1]^d.
\end{align}

The state variable $X$ represents the consumer's deviation from the deterministic profile of his consumption. An effort $\nu$ induces a separable cost $c(\nu) := \frac12 c_\alpha (\alpha) + \frac12 c_\beta (\beta)$, where:
\begin{align*}
    c_{\alpha}(a) &:= \sum_{k=1}^{d} \dfrac{ ( a^{k} )^2 }{\rho^{k}}, \; a \in A,\; 
    \text{ and } \; c_{\beta}(b) := \sum_{k=1}^{d} \dfrac{ ( \sigma^{k} )^2 }{\lambda^{k} \eta^{k} } \Big( ( b^{k} )^{- \eta^{k} } -1 \Big), \; b \in B,
\end{align*}
for fixed $(\rho,\lambda,\eta)\in (0,+\infty)^d\times (0,+\infty)^d\times (1,+\infty)^d$. In particular, the cost of the effort in the drift term of $X$, denoted by $c_{\alpha}$, is a classical quadratic cost function, meaning that no effort for the Agent induces no intrinsic cost, and such that he has no interest to provide negative efforts. The cost associated with the effort in the volatility of $X$ prohibits the Agent from removing the volatility ($b^{k} >0$) and is equal to zero if the Agent makes no effort (case $b^{k}=1$). For technical reasons, we need to consider bounded efforts\footnote{Note that since $X$ is a deviation from a baseline consumption, the upper bound for the drift control is coherent, since the agent cannot consume a negative amount of electricity.}, we then set
\begin{align}\label{eq:def_A_B}
    A := [0, \rho^1 A_{\text{max}} ] \times \dots \times [0, \rho^d A_{\text{max}} ] \; \text{ and } \; B:= [B_{\text{min}}, 1]^d,
\end{align}
for some constants $A_{\text{max}} > 0$ and $B_{\text{min}} \in (0,1)$.

\medskip

{\color{black} One may note that, apart from the term induced by the common noise $W^\circ$, the previously defined mathematical set up is voluntarily the same as in \citet*{aid2018optimal}. Nevertheless, the introduction of the common noise forces us to rigorously write the definition of the weak formulation. In the following sections, we properly define a canonical space $\Omega$ for the representative Agent in Subsection \ref{sss:omega}, as well as a canonical space $\widehat \Omega$ for others in Subsection \ref{sss:widehat_omega}.}

\subsection{Theoretical formulation}\label{ss:theoretical_formulation}

{\color{black}
\subsubsection{Canonical space for the representative Agent}\label{sss:omega}}

To be consistent with the weak formulation of control problems, we let $\mathbb U$ be the collection of all finite and positive Borel measures on $[0,T]\times U$, whose projection on $[0,T]$ is the Lebesgue measure. In other words, every $q \in \mathbb \U$ can be disintegrated as $q(\mathrm{d}s,\mathrm{d} v)=q_s(\mathrm{d} v)\mathrm{d}s$, for an appropriate Borel measurable  kernel $q_s$. The weak formulation requires to consider a subset of $\U$, namely the set $\U_0$ of all $q \in \U$ such that the kernel $q_s$ is of the form $\delta_{\phi_s}(\mathrm{d} v)$ for some Borel function $\phi$, where as usual, $\delta_{\phi_s}$ is the Dirac mass at $\phi_s$.

\medskip

In order to combine the theoretical formulations of mean--field games and McKean--Vlasov problems, we are led to consider the following canonical space
\[
    \Omega :=  \Omega^{\circ} \times \Omega^1 \times \Omega^2 \times \U, \;
    \text{where } \; \Omega^{\circ} := \Cc_T, \;
    \Omega^1 := \Cc_T \times \Cc_T^d, \;
    \text{and } \; \Omega^2 := \Pc (\Cc_T),
\]
{\color{black}recalling from the notations that $\Cc_T^d := \Cc([0,T], \R^d)$ and $\Cc_T := \Cc([0,T], \R)$,} with canonical process $(W^{\circ}, X, W, \mu, \Lambda)$, where for any $(t, w^{\circ}, x, w, u, q) \in [0,T] \times \Omega$,
\[
    W^{\circ}_t(w^{\circ},x,w,u,q):=w^{\circ}(t), \; X_t(w^{\circ},x,w,u,q) :=x(t),\; W_t(w^{\circ},x,w,u,q):=w(t),
\]
    \vspace{-1.2em}
\[  
    \mu_t(w^{\circ},x,w,u,q):=u(t),\; \text{and} \; \Lambda_t(w^{\circ},x,w,u,q):=q.
\]

Less formally, $X$ represents the deviation consumption controlled by the representative consumer, affected by an idiosyncratic noise $W$ and a common noise $W^{\circ}$. The process $\mu$ allows to keep track of the law of $X$ conditionally to the common noise and the space $\U$ corresponds to the controls. The canonical filtration $\F:=(\Fc_t)_{t\in[0,T]}$ is defined as
\begin{align*}
    \mathcal F_t:=\sigma\Big( \big(W^{\circ}_s, X_s, W_s, \mu_s, \Delta_s(\varphi) \big): (s,\varphi)\in[0,t]\times \Cc_b \big( [0,T]\times U, \R \big) \Big),\; t\in [0,T],
\end{align*}
where $\Cc_b([0,T]\times U,\R)$ is the set of all bounded continuous functions from $[0,T]\times U$ to $\R$, and for any $(s,\varphi)\in[0,T]\times \Cc_b([0,T]\times U,\R)$, $\Delta_s(\varphi):=\int_0^s\int_U \varphi(r,v ) \Lambda(\mathrm{d}r,\mathrm{d} v).$
% Let $\M$ be the set of all probability measures on $(\Omega,\Fc_T)$ and define then the following set
% \begin{align*}
%     \Pc:=\Big\{\P\in\M:M(\varphi)\text{ is a $(\P,\mathbb F)$--local martingale for all $\varphi\in \Cc^2_b(\mathbb R^d,\R)$, and $\mathbb P[X_0=x_0,\Lambda\in\mathbb U_0]=1$}\Big\},    
% \end{align*}
% where $\Cc^2_b(\mathbb R^d,\R)$ is the set of bounded twice continuously differentiable functions from $\mathbb R^d$ to $\R$, whose first and second derivatives are also bounded, and for any $(s,\varphi)\in[0,T]\times \Cc^2_b(\mathbb R^d,\R)$
We will also need a smaller filtration containing only the information generated by the common noise and the conditional law of $X$. Namely, we define $\F^{\circ}:=(\Fc^{\circ}_t)_{t\in[0,T]}$ by $
    \Fc^{\circ}_t:=\sigma\big((W^{\circ}_s,\mu_s):s\in[0,t]\big),\; t\in[0,T].$

\begin{remark}
The restriction to the filtration $\F^{\circ}$ stems from the presence of common noise in our model. As pointed out by \textnormal{\citet*{carmona2016meana}}, it is commonplace for control problems in weak formulation, and actually already for weak solutions to {\rm SDEs}, that the underlying driving noise is not rich enough to obtain a solution which is only adapted to it. In our context, this translates into the fact that in general the law of $X$ may fail to be measurable only with respect to the information generated by the common noise $W^\circ$, which justifies the enlargement in the definition of $\F^\circ$. This is linked to the so--called \textit{compatibility condition} in the {\rm MFG} theory with common noise, which intuitively means that a given player in the game does have access to the full information generated by the idiosyncratic and common noises $W$ and $W^\circ$, as well as the distribution of all other players' states, and his controls are allowed to be randomised externally to these observations, but such a randomisation must be conditionally independent of future information given current information.
\end{remark}

Let $\Cc^2_b(\mathbb R \times \R^d \times\R,\R)$ be the set of bounded twice continuously differentiable functions from $\mathbb R \times \R^d \times\R$ to $\R$, whose first and second derivatives are also bounded, and for any $(s,\varphi)\in[0,T]\times \Cc^2_b(\mathbb R \times \R^d \times\R,\R)$, we set
\begin{align*}
    M_s(\varphi):=&\ \varphi(X_s,W_s,W^{\circ}_s) - \int_0^s \int_U\bigg( A(v) \cdot \nabla \varphi(X_r,W_r,W^{\circ}_r) + \frac12  {\rm Tr} \big[ D^2 \varphi(X_r,W_r,W^{\circ}_r) B(v)  B^\top (v) \big] \bigg)\Lambda(\mathrm{d}r, \mathrm{d} v),
\end{align*}
where $D^2 \varphi$ denotes the Hessian matrix of $\varphi$, $A$ and $B$ are respectively the drift vector and the diffusion matrix of the vector process $(X, W, W^{\circ})^\top$
\begin{align*}
    A(v) := 
    \begin{pmatrix}
        - a \cdot \mathbf{1}_d \\
        \mathbf{0}_d \\
        0
    \end{pmatrix},
    \; 
    B(v) :=
    \begin{pmatrix}
        0 & \sigma^\top( b ) & \sigma^{\circ} \\
        \mathbf{0}_d & \mathrm{I}_d & \mathbf{0}_d \\
        0 & \mathbf{0}_d^\top & 1
    \end{pmatrix},\; v:=(a,b)\in U.
\end{align*}

Therefore, the covariation matrix of the vector process $(X, W, W^{\circ})^\top$ is defined for all $v \in U$ by
\begin{align*}
    B(v) B^\top (v) =
    \begin{pmatrix}
        \Sigma(b) + \big(\sigma^\circ\big)^2 & \sigma^\top( b ) & \sigma^{\circ} \\
        \sigma ( b ) & \mathrm{I}_d & \mathbf{0}_d \\
        \sigma^\circ & \mathbf{0}_d^\top & 1
    \end{pmatrix},
\end{align*}
where $\Sigma(b) := \sigma^\top(b) \sigma (b)$ for all $b \in B$. We fix some initial conditions, namely a probability measure $\varrho$ on $\R$ representing the law at $0$ of $X$. 

{\color{black} \begin{definition}\label{def:Pc}
Let $\M$ be the set of probability measures on $(\Omega,\Fc_T)$. The subset $\Pc \subset \M$ is composed of all $\P$ such that
\begin{enumerate}[label=$(\roman*)$]
    \item $M(\varphi)$ is a $(\P,\mathbb F)$--local martingale on $[0,T]$ for all $\varphi \in \Cc^2_b(\R \times \R^d \times \R,\R);$
    \item $\P \circ (X_0)^{-1} = \varrho$, and there exists a measure $\iota$ on $\mathbb R^d\times\mathbb R$ such that $\P \circ \big((W_0,W_0^{\circ})\big)^{-1} = \iota;$
    \item $\P\big[\Lambda \in \U_0]=1;$
    \item for $\P$--a.e. $\omega \in \Omega$ and for every $t \in[0,T]$, we have $\mu_t(\omega)=\P^{\omega}_t \circ (X_{t\wedge \cdot})^{-1},$ 
    where $(\P_t^{\omega})_{\omega\in\Omega}$ is a family of regular conditional probability distribution\footnote{We recall that these objects are such that for any $\omega\in\Omega$, $\P_t^\omega$ is a probability measure on $(\Omega,\Fc)$, such that for any $A\in\Fc$, the map $\omega \longmapsto \P_t^\omega[A]$ is $\Fc^{\circ}_t$--measurable, and such that for any $\P$--integrable random variable $\xi$ on $(\Omega,\Fc)$, we have
    \[
    \mathbb E^\P[\xi|\Fc^{\circ}_t](\omega)=\E^{\P_t^\omega}[\xi],\; \text{for $\P$--a.e. $\omega\in\Omega$.}
    \] 
    Notice that since $(\Omega,\Fc)$ is a Polish space and $\Fc_t^{\circ}$ is countably generated, the existence of these r.c.p.d. is guaranteed for instance by \citet*[Theorems 2.6.5 and 2.6.7]{cohen2015stochastic}. } $($r.c.p.d for short$)$ for $\P$ given $\Fc_t^{\circ}$. We will denote by $\E^{\P_t^\omega}$ the expectation under the distribution $\P_t^\omega$. For ease of notation, we will often omit the $\omega$ in the notation for the expectation$;$
    \item $(W^{\circ},\mu)$ is $\P$--independent of $W$.
\end{enumerate}
\end{definition}}

{\color{black} Roughly speaking, the set $\Pc$ represents the set of admissible controls in the weak formulation. Nevertheless,} the previous formulation does not give us access directly to the dynamic of the consumption deviation $X$. It is however a classical result that, enlarging the canonical space if necessary, one can construct Brownian motions allowing to write rigorously the dynamics \eqref{eq:informal}, see for instance \citet[Theorem 4.5.2]{stroock1997multidimensional}. It turns out here that since we enlarged the canonical space right from the start to account for the idiosyncratic and common noises, any further enlargement is not required, {\color{black} see Lemma \ref{lemma:rep} below, whose proof is {\color{black} deferred} to Appendix \ref{ss:proof_lemma}.}
\begin{lemma}\label{lemma:rep}
    For all $\P\in\Pc$, $\Lambda(\mathrm{d}s,\mathrm{d} v) = \delta_{\nu^\P_s}(\mathrm{d} v)\mathrm{d}s$ $\P$--{\rm a.s.}, for some $\F$--predictable control process $\nu^\P := \big( \alpha^\P, \beta^\P \big)$ and 
    \begin{align*}
        X_t &= X_0 - \int_0^t \alpha^{\P}_s \cdot{\bf 1}_d \drm s + \int_0^t  \sigma (\beta^\P_s) \cdot \drm W_s + \int_0^t \sigma^{\circ} \drm W^{\circ}_s, ~t\in[0,T], \; \P-\textnormal{a.s.}
    \end{align*}
\end{lemma}
{\color{black} Notice that, thanks to the previous lemma, the set $\Uc = \mathcal{A} \times \mathcal{B}$ of admissible efforts, introduced at the beginning of Section \ref{sec:informal}, is now well defined.}

% Let $\varphi \in \Cc^2_b (\R \times \R^d \times \R, \R)$. Applying Itô's formula on $\varphi$, we obtain:
% \begin{align*}
%     \varphi(X_s,W_s,W^{\circ}_s) = &\ \varphi(X_0,W_0,W^{\circ}_0) 
%     + \int_0^s \nabla \varphi(X_r,W_r,W^{\circ}_r) \mathrm{d} (X_r, W_r, W^\circ_r)^\top 
%     + \dfrac12 \int_0^s {\rm Tr} \big[ D^2 \varphi(X_r,W_r,W^{\circ}_r)  \mathrm{d} \langle (X, W, W^\circ)^\top \rangle_r \big].
% \end{align*}
% \begin{align*}
%     \varphi(X_s,W_s,W^{\circ}_s) = &\ \varphi(X_0,W_0,W^{\circ}_0) 
%     + \int_0^s \nabla \varphi(X_r,W_r,W^{\circ}_r) \mathrm{d} (X_r, W_r, W^\circ_r)^\top
%     + \int_0^s \partial_w \varphi(X_r,W_r,W^{\circ}_r)  \mathrm{d} W_r \\
%     &+ \int_0^s \partial_{w^\circ} \varphi(X_r,W_r,W^{\circ}_r) \mathrm{d} W^\circ_r
%     + \dfrac12 \int_0^s \partial^2_{xx} \varphi(X_r,W_r,W^{\circ}_r) \mathrm{d} \langle X \rangle_r \\
%     &+ \dfrac12 \int_0^s \partial^2_{ww} \varphi(X_r,W_r,W^{\circ}_r)  \mathrm{d} \langle W \rangle_r
%     + \dfrac12 \int_0^s \partial^2_{w^\circ w^\circ} \varphi(X_r,W_r,W^{\circ}_r) \mathrm{d} \langle W^\circ \rangle_r \\
%     &+ \int_0^s \partial^2_{x w} \varphi(X_r,W_r,W^{\circ}_r) \mathrm{d} \langle X, W \rangle_r \\
%     &+ \int_0^s \partial^2_{x w^\circ} \varphi(X_r,W_r,W^{\circ}_r)  \mathrm{d} \langle X, W^\circ \rangle_r
%     + \int_0^s \partial^2_{w w^\circ} \varphi(X_r,W_r,W^{\circ}_r) \mathrm{d} \langle W, W^\circ \rangle_r 
% \end{align*}

\medskip

In order to apply the chain rule with common noise, as defined in \citet*[Theorem 4.17]{carmona2018probabilisticII}, we will need a copy of the process $X$, denoted by $\widetilde{X}$, driven by the same common noise $W^{\circ}$, and with the same conditional law $\mu$. For this purpose, we need to define a copy of the initial canonical space.
\begin{definition}[Copy of a space]\label{def:copy_space}
    Let $\Omega$ be a canonical space of the form $\Omega :=  \Omega^{\circ} \times \Omega^1 \times \Omega^2 \times \U$. A copy of $\Omega$ is defined by $\widetilde{\Omega} := \Omega^{\circ} \times \widetilde{\Omega}^1 \times \Omega^2 \times \widetilde \U$ where $\widetilde{\Omega}^1$ and $\widetilde \U$ are respectively standard copies of the spaces $\Omega^1$ and $\U$.
\end{definition}
This canonical space $\widetilde \Omega$ is supporting a canonical process $\big(W^{\circ}, \widetilde X, \widetilde W, \mu, \widetilde \Lambda\big)$, and the canonical filtration $\widetilde \F:=(\widetilde \Fc_t)_{t\in[0,T]}$ is defined exactly as $\F$.
%\begin{align*}
%    \widetilde{\Fc}_t:=\sigma\Big( \big(W^{\circ}_s, \widetilde X_s, \widetilde W_s, \mu^X_s, \widetilde \Delta_s(\varphi) \big): (s,\varphi)\in[0,t]\times \Cc_b \big( [0,T]\times U, \R \big) \Big),\; t\in [0,T],
%\end{align*}
%where for any $(s,\varphi)\in[0,T]\times \Cc_b([0,T]\times U,\R)$,
%\begin{align*}
%    \widetilde \Delta_s(\varphi):=\int_0^s\int_U \varphi(r,v ) \widetilde \Lambda(\mathrm{d}r,\mathrm{d} v).
%\end{align*}}
{\color{black}
For a given probability $\P \in \Pc$, we can associate in a unique way a probability $\widetilde \P$ satisfying Definition \ref{def:Pc} on $( \widetilde \Omega, \widetilde \Fc_T )$, where, in particular, $\widetilde \P$ and $\P$ have the same r.c.p.d.
% We let $\widetilde \M$ be the set of all probability measures on $( \widetilde \Omega, \widetilde \Fc_T )$, we can then define the subset $\widetilde{\Pc} \subset \widetilde \M$ in the same way we define $\Pc \subset \M$ in Definition \ref{def:Pc}.  
Therefore, abusing notations slightly, $\widetilde{\E}^{\P_t}$ will stand for the expectation on $(\widetilde \Omega, \widetilde \Fc_T)$ under the r.c.p.d. $\P_t$ for $\widetilde \P \in \Pc$ given $\Fc_t^\circ$. Then, by Lemma \ref{lemma:rep}, for any $\widetilde \P$, we have $\widetilde \Lambda(\mathrm{d}s,\mathrm{d} v) = \delta_{\nu^{\widetilde \P}_s}(\mathrm{d} v)\mathrm{d}s$ $\widetilde \P$--a.s., for some $\widetilde \F$--predictable process $\nu^{\widetilde \P} := ( \alpha^{\widetilde \P}, \beta^{\widetilde \P})$ and}
\begin{align}\label{eq:copy}
    \widetilde X_t = \widetilde X_0 - \int_0^t \alpha^{\widetilde \P}_s \cdot{\bf 1}_d \drm s + \int_0^t  \sigma (\beta^{\widetilde \P}_s) \cdot \drm \widetilde W_s + \int_0^t \sigma^{\circ} \drm W^{\circ}_s, \; t\in[0,T], \; \widetilde \P-\text{a.s.}
\end{align}
\begin{definition}[Copy of a process]\label{def:copy_process}
    The process $\widetilde X$ defined above by \eqref{eq:copy} is called a copy of $X$.
\end{definition}

\subsubsection{Canonical space of other Agents}\label{sss:widehat_omega}

{\color{black} To model the deviation consumption of other consumers (that is the ones different from the representative agent), affected by the same common noise $W^{\circ}$, we need to define an alternative probability space $\widehat \Omega$ by $\widehat{\Omega} := \Omega^{\circ} \times \widehat{\Omega}^1 \times \widehat{\Omega}^2 \times \widehat \U$. This canonical space is supporting a canonical process $\big(W^{\circ}, \widehat X, \widehat W, \widehat \mu, \widehat \Lambda\big)$, and the canonical filtration $\widehat \F:=(\widehat \Fc_t)_{t\in[0,T]}$ is defined in the same way as $\F$.}
%\begin{align*}
%    \widehat{\Fc}_t:=\sigma\Big( \big(W^{\circ}_s, \widehat X_s, \widehat W_s, \widehat \mu_s, \widehat \Delta_s(\varphi) \big): (s,\varphi)\in[0,t]\times \Cc_b \big( [0,T]\times U, \R \big) \Big),\; t\in [0,T],
%\end{align*}
We let $\widehat \M$ be the set of all probability measures on $( \widehat \Omega, \widehat \Fc_T )$, we can then define the subset $\widehat{\Pc} \subset \widehat \M$ in the same way we define $\Pc \subset \M$ in Definition \ref{def:Pc}. In particular, the notation $\widehat{\E}^{\widehat \P_t}$ will stand for the expectation under the r.c.p.d. $\widehat \P_t$ of $\widehat \P \in \widehat \Pc$ given $\widehat \Fc^\circ_t$, on the space $(\widehat \Omega, \widehat \Fc )$. Still applying Lemma \ref{lemma:rep}, for any $\widehat \P \in \widehat{\Pc}$, we can write
\begin{align*}
    \widehat X_t &= \widehat X_0 - \int_0^t \alpha^{\widehat \P}_s \cdot{\bf 1}_d \drm s + \int_0^t  \sigma (\beta^{\widehat \P}_s) \cdot \drm \widehat W_s + \int_0^t \sigma^{\circ} \drm W^{\circ}_s, ~t\in[0,T],
\end{align*}
where $\nu^{\widehat \P} := ( \alpha^{\widehat \P}, \beta^{\widehat \P})$ is some $\widehat \F$--predictable control process, chosen by others, satisfying $\widehat \Lambda(\mathrm{d}s,\mathrm{d} v) = \delta_{ \nu^{\widehat \P}_s}(\mathrm{d} v)\mathrm{d}s$ $\widehat \P$--a.s. {\color{black} Notice that this allows us to properly define the set $\widehat \Uc = \widehat{\mathcal{A}} \times \widehat{\mathcal{B}}$ of admissible efforts of others.}

\medskip

Copies of the deviation consumption of others, denoted by $\widecheck{X}$, are defined in the same way as copies $\widetilde X$ of $X$ by Definition \ref{def:copy_process}, on the space $\widecheck \Omega: = \Omega^{\circ} \times \widecheck \Omega^1 \times \widehat{\Omega}^2 \times \widecheck \U$, itself a copy of $\widehat \Omega$ in the sense of Definition \ref{def:copy_space}. This space is supporting a canonical process and the associated canonical filtration $\widecheck \F := ( \widecheck \Fc_t )_{t \in [0,T]}$. The notation $\widecheck{\E}^{\widehat \P_t}$ stands for the expectation under the r.c.p.d. $\widehat \P_t$ on the space $( \widecheck\Omega,\widecheck\Fc)$. 

\medskip

{\color{black} In words, the canonical process $(W^\circ, X, W, \mu, \Lambda)$ defined on the space $\Omega$ represents the choices of the representative Agent regarding his deviation consumption. Similarly, the canonical process $(W^\circ, \widehat X, \widehat W, \widehat \mu, \widehat \Lambda)$ defined on $\widehat \Omega$ allows us to represent the choices of other consumers, which may be different from the representative consumer, but are affected by the same common noise $W^\circ$. Then, notations involving $\widetilde{\cdot }$ will refer to copies of the initial space $\Omega$, while notations involving $\widecheck{\cdot}$ will refer to copies of the canonical space of others $\widehat \Omega$.
In order to compute his terminal payment $\xi$, the representative Agent is going to assume that the others have played some distribution $\widehat \mu$ (mainly the conditional law of $\widehat X$ under some $\widehat \P \in \widehat \Pc$ given $\widehat \F^\circ$ on $(\widehat \Omega, \widehat \Fc_T)$), and he is going to compute $\xi$ along his own deviation $X$ and $\widehat \mu$.}

%\medskip
%
%\textcolor{orange}{By Karandikar \cite{karandikar1995pathwise}, there is an $\F$--progressively measurable process, denoted by $\langle X \rangle = (\langle X \rangle_t)_{0 \le t \le T}$, which coincides with the quadratic variation of $X$, $\P-a.s.$, for all semi--martingale measure $\P$. We may then introduce the symmetric definite semi--positive matrix $\overline \sigma_t$ such that
%\begin{align*}
%    \overline \sigma_t^2 := \limsup_{\eps \searrow 0} \frac{\langle X \rangle_t - \langle X \rangle_{t-\eps}}{\eps},\ t\in[0,T].
%\end{align*}
%\noindent For technical reasons, we work under the classical ZFC set--theoretic axioms, as well as  the continuum hypothesis\footnote{This is needed for the application of Nutz's aggregation result for stochastic integrals in \cite{nutz2012pathwise}.}.}

\subsection{Definition of a contract}

In the work~\cite{aid2018optimal}, the Principal -- \textit{an energy producer} -- offers a contract to an Agent -- \textit{a consumer} -- indexed on his deviation consumption. In our investigation, the Principal faces a continuum of Agents with mean--field interaction and can therefore benefit from this additional information: she can offer contracts depending on both the deviation of a given consumer, and the aggregate statistics of the deviation of other consumers. 
%This form of contract allows the Principal to index the compensation of a consumer on his deviation, but also on the efforts of other consumers, through the distribution of their deviations. 
In the framework we are interested in, the electricity producer is not allowed to reveal the consumption of a particular consumer to another consumer. Hence, she cannot directly design a remuneration for a consumer with respect to the deviation consumption of another dedicated one. This is why we consider aggregated statistics in this case. We insist on the fact that it is a legal requirement for electricity producers to respect the privacy of their consumers, even when they have access to their consumption profile through modern smart--meters. Moreover, in accordance with the mean--field framework, the Principal is facing a mass of identical and indistinguishable consumers, and thus can not choose a deviation consumption of another consumer to contract on it. Formally, the Principal proposes to the representative Agent a contract $\xi$ which is a random variable measurable with respect to the natural filtration generated by both $X$, $\widehat{\mu}$, denoted $\F^{\text{obs}}$, recalling that:
\begin{itemize}
    \item $X$ is the deviation consumption of the representative Agent,
    \item $\widehat\mu$ is the law of the deviation consumption of other Agents, conditionally to the common noise. 
\end{itemize}

In other words, $\xi$ must be a measurable functional of the paths of $X$ and $\widehat\mu$:
\begin{align}\label{contract:dependency}
    \xi : (X, \widehat \mu) \in \Cc_T \times \Pc (\Cc_T)  
    \longmapsto \xi (X, \widehat \mu) \in \mathbb{R}.
\end{align}

\begin{remark}\label{rk:conditional_common_noise}
    Considering the conditional law $\widehat \mu$ naturally comes from the limit of the $N$--Agents case. Indeed, if the Principal monitors $N$ consumers, she wants to index the contract for the $i$--th consumer on his own deviation $X^i$ and on the empirical distribution of other consumers, $\overline{\mu}^{-i}$, defined as: 
    \begin{align*}
        \overline{\mu}^{-i} := \bigg( \dfrac{1}{N-1} \sum_{j=1, j\neq i}^N \delta_{X_t^j} \bigg)_{t \geq 0} 
    \end{align*}
    where the deviations $X^j$ are not independent since all consumers are suffering from the common noise $W^{\circ}$. Hence, in the mean--field framework, we may wonder about the convergence of the empirical measure $\overline{\mu}^{-i}_t$ as $N$ tends to $\infty$. In the absence of the common noise, the standard theory of propagation of chaos applies: asymptotically, particles become independent and the $\overline{\mu}^{-i}_t$ converges to their common asymptotic distribution. By contrast, when there is a common noise, even in the limit $N \longrightarrow \infty$, the particles must still keep track of the common noise $W^{\circ}$, so they cannot become independent. Nevertheless, it has been proved in \textnormal{\citet*{carmona2018probabilisticII}} that particles become asymptotically independent conditionally on the common noise, and that the empirical distribution converges towards the common conditional distribution of each particle given the common noise \textnormal{(}$\widehat \mu$ in our case\textnormal{)}.
\end{remark}

Given this contract, and the conditional law of the deviation consumption of other consumers $\widehat \mu$, the representative consumer solves the following optimisation problem
\begin{equation}\label{pb:consumer}
	V^{A}_0 \big( \xi, \widehat \mu \big) := \sup_{\P \in \Pc} J_0^A(\xi, \widehat{\mu}, \P), \; \text{where} \; J_0^A(\xi, \widehat{\mu}, \P) := \mathbb{E}^{\mathbb{P}} \bigg[ U^A \bigg( \xi(X, \widehat \mu) - \int_0^T \big( c \big( \nu^\P_t \big) - f \big( X_t \big) \big) \mathrm{d} t \bigg) \bigg],
\end{equation}
where $c:\R_+^d \times (0,1]^d \longrightarrow \mathbb R_+$ is the cost function associated with the effort $\nu$ made by the Agent, and the function $f:\mathbb R \longrightarrow \mathbb R$ denotes the preference of the Agent toward his deviation consumption. The function $f$ is required to be concave, increasing, and centred at the origin. This means that the reduction of the Agent's consumption causes him discomfort, and conversely, if the consumption deviation is non-negative, the agent consumes more electricity, which gives him satisfaction. Closed--form solutions will be obtained for linear $f$. The function $U^A$ is an exponential utility function, with risk aversion parameter of the representative consumer $R_A>0$, defined by $U^A (x) = - \erm^{-R_A x}$. {\color{black} One may note that, apart from the contract dependency, the Agent's problem is identical to the one defined in \cite{aid2018optimal}.

\medskip

For $V_0^A(\xi,\widehat\mu)$ to make sense, we require minimal integrability on the contracts, by imposing that
\begin{equation}\label{eq:integxi}
    \sup_{\P\in\Pc} \E^\P \Big[ \erm^{pR_A|\xi|} \Big]<+\infty, \; \text{for some} \; p > 1.
\end{equation}
In addition and for technical reasons, we restrict our attention to contracts $\xi$ satisfying also the technical assumption \eqref{integrability:rp:xi} given below, only necessary in order to solve the problem of a CARA risk averse Principal (and useless to solve the Agent problem). These two technical assumptions are always satisfied in the application to linear energy value discrepancy (see Section \ref{sec:applicationEVD}).
Moreover, as usual in contract theory, we assume that consumers have an endogenous reservation utility $R_0<0$, below which they refuse the contract offered by the producer. The underlying idea is that without compensation (that is for $\xi=0$), a consumer could already exert efforts and modify his consumption, and thus receive utility $R_0$. He thus would, of course, refuse any contract which does not provide him with at least what he could get by himself\footnote{We refer to Appendix \ref{sec:reservation} for more details on how to compute the value of $R_0$.}. The corresponding class of contracts of the form \eqref{contract:dependency}, satisfying \eqref{eq:integxi}, \eqref{integrability:rp:xi}, and the participation constraint, is denoted by $\Xi$.}

\medskip

Finally, notice here that the contracts offered by the Principal have been assumed to not be indexed on the common noise $W^\circ$. This can either mean that the Principal cannot observe it perfectly, or that there are regulatory reasons preventing the producer from using it directly in the contract. However, our formulation allows to incorporate the case where this becomes possible. Thus, in Section \ref{sec:observable_common_noise}, we will study the case where the Principal is allowed to use the common noise directly in the contract. In this particular case, she can offer to the representative Agent a contract $\xi$, measurable with respect to the natural filtration generated by $X$, $W^\circ$ and $\widehat \mu$, denoted by $\F^{\textnormal{obs},\circ}$. In other words, $\xi$ must be in this case a measurable functional of the paths of $X$, $W^\circ$ and $\widehat\mu$:
\begin{align}\label{contract:dependency_w0}
    \xi : \big( X, W^\circ, \widehat\mu \big) \in \Cc_T \times \Cc_T \times \Pc (\Cc_T) \longmapsto \xi \big( X, W^\circ, \widehat\mu \big) \in \mathbb{R}.
\end{align}
%
%Given this contract, the conditional law of the deviation consumption of other consumers $\widehat \mu$, and the paths of the common noise, the representative consumer faces the following optimisation problem
%\begin{equation}\label{pb:consumer_w0}
%	V^{A}_0 \big( \xi, W^\circ, \widehat \mu \big) := \sup_{\P \in \Pc} \mathbb{E}^{\mathbb{P}} \bigg[ U^A \bigg( \xi(X, W^\circ, \widehat \mu) - \int_0^T \big( c \big( \nu^\P_t \big) - f \big( X_t \big) \big) \mathrm{d} t \bigg) \bigg],
%\end{equation}
%with the same specifications as before. 
The corresponding class of contracts will be denoted by $\Xi^\circ$.

\subsection{Definition of a mean--field equilibrium}

We work under a classical mean--field framework, where all agents are identical. Hence, similarly to \citet*{carmona2016meana}, we define a mean--field equilibrium as follows.
\begin{definition}[Mean--field equilibrium]\label{mfe}
    Let $\xi\in\Xi$ be a contract. We denote by $\Mc^{\star} (\xi)$ the collection of all mean--field equilibria, i.e., pairs $(\P^{\star}, \mu^{\star}) \in \Pc \times \Pc (\Cc_T)$ such that
\begin{enumerate}[label=$(\roman*)$]
    \item given $\mu^{\star} \in \Pc (\Cc_T)$, the probability $\P^{\star} \in \Pc$ is optimal for \eqref{pb:consumer}, \text{i.e.},
    \begin{align*}
	    V^{A}_0(\xi, \mu^{\star}) = \mathbb{E}^{\mathbb{P^{\star}}} \bigg[ U^A \bigg( \xi (X, \mu^{\star}) - \int_0^T \big( c \big( \nu^{\P^{\star}}_t \big) - f \big( X_t \big) \big) \mathrm{d} t \bigg) \bigg];
    \end{align*}
    \item for $\P^{\star}$--a.e. $\omega \in \Omega$ and for every $t \in [0,T]$, we need to have
    \begin{align*}
        \mu^{\star}_t(\omega)=\P^{\omega}_t \circ (X_{t \wedge \cdot} )^{-1},
    \end{align*}
    where $(\P_t^{\omega})_{\omega\in\Omega}$ is a family of regular conditional probability distribution for $\P^{\star}$ given $\Fc_t^{\circ}$. 
\end{enumerate}
We extend readily this definition to contracts in $\Xi^\circ$, and denote by $\Mc^{\star,\circ}(\xi)$ the associated set of mean--field equilibria.
\end{definition}
{\color{black}
\begin{remark}\label{rk:measure_on_pathspace}
Less formally, a mean--field equilibrium is characterised by
\begin{enumerate}[label=$(\roman*)$]
    \item a probability law $\P^{\star}$ of a process $X^{\star}$, which is the optimal deviation consumption of each Agent$;$
    \item the conditional law $\mu^{\star}$ of $X^{\star}$ with respect to the common noise.
\end{enumerate}
    The attentive reader may note that our definition of a mean--field equilibrium, mainly the fixed point constraint, involves the probability measure on the path space, i.e., $\mu^\star \in \Pc(\Cc_T)$. Although this is not standard in the theory of mean--field games in a Markovian setting, this makes sense in our non--Markovian framework, and a similar condition can be found in the work of \textnormal{\citet*{carmona2015probabilistic}}.
\end{remark}}

% All consumers are supposed to be identical, hence similarly to \cite{carmona2015probabilistic}, we define a Mean Field Equilibrium as:
% \begin{Definition}[Mean Field Equilibrium]\label{mfe}
% Let $\xi$ be a random variable of the form \eqref{contract:dependency}. A Mean Field Equilibrium is a pair $\big(\mathbb{P}^{\star}, \nu^{\star}\big) \in \Mc$ such that
% \begin{itemize}
%     \item[(i)]  The process $\nu^{\star}$ of the representative consumer is optimal for \eqref{pb:consumer} under $\P^{\star}$,
%     \text{i.e.}
%     \begin{align*}
% 	V^{A}_0(\xi(X, \mu)) = \mathbb{E}^{\mathbb{P^{\star}}} \left[ U^A \bigg( \xi (X, \mu) - \int_0^T \left( c ( \nu^{\star}_t ) - f \big( X_t \big) \right) \mathrm{d} t \bigg) \right].
% \end{align*}
% \item[(ii)]  Since all consumers are identical, the law of $X$ under $\mathbb{P}^{\star}$ has to be $\mu$, \textit{i.e.} 
% \begin{align*}
%     \mathbb{P}^{\star} \circ  X^{-1} = \mu.
% \end{align*}
% \end{itemize}
% \end{Definition}

\subsection{The Producer}\label{ss:def_pb_producer}

We now turn to the problem of the Principal.
In the one Agent framework defined in \cite{aid2018optimal}, the Principal has an exponential utility function, with risk--aversion parameter $R_P>0$, defined by $U^P (x) = - \erm^{-R_P x}$ and wants to minimise:
\begin{enumerate}[label=$(\roman*)$]
    \item the compensation paid to the Agent: $\xi$.
    \item the cost of production, corresponding to additional costs induced by the deviation consumption: $g (X_t) $, where $g$ is concave and increasing. It means that if the Agent's deviation consumption is positive, the consumption has increased, hence the Principal has an additive cost of production. Conversely, a negative deviation consumption means that the consumption is decreasing, hence the Principal benefits from a decrease of production costs. 
    % \textcolor{blue}{ATTENTION : $g$ linéaire sinon on a $\lim_{N \rightarrow + \infty} \frac{1}{N} g (\sum_{i=1}^N X^{i}_s ) $ et on ne peut pas écrire $\mathbb{E}^1 \big[ g ( X_s ) \big]$.} 
    \item the quadratic variation of the deviation consumption: $\langle X\rangle_t $. This penalisation term allows the Principal to take into account the variations of consumption over time. This additional cost is particularly relevant in electricity markets, since the Producer has to follow the load curve, and the higher the volatility of the consumption, the more costly it is.
\end{enumerate}

%In order to obtain the Mean Field formulation of the Principal problem, one want to pass to the limit in~(\ref{eq:pb_principal_Nagents}) when $N \to + \infty$. In this context, for the sake of simplicity, we should make the assumption that the function $g$ is linear. Under this assumption, one have a smooth formulation of the limit of $L^N$:
%\begin{align*}
%	\lim_{N \to + \infty} L_t^N = \mathbb{E}^1 [\xi] + \int_0^t \mathbb{E}^1 \left[ g ( X_s ) \right] \mathrm{d} s + \dfrac{\theta}{2} \int_0^t \mathbb{E}^1 \left[ \drm \langle X^{i} \rangle_s \right],
%\end{align*}

The intuition for the Principal problem in the mean--field case comes from the $N$--Agents case. Formally, if we consider a $N$--players model, the Principal would minimise the (utility of the) sum of the previous costs. To ensure stability of these sums as $N$ grows, and therefore obtain a mean--field limit of the $N$--Agents problem, we can follow the line of \cite{elie2019tale} by assuming that each individual deviation consumption is scaled by the total number of Agents $N$. In their framework, \textit{without common noise and also with exponential utility functions}, the Principal becomes risk--neutral in the limit when $N \longrightarrow + \infty$, by classical propagation of chaos arguments. Another interpretation of the risk--neutrality of the Principal in this case is that the Principal is diversifying the risk by considering a large number of consumers: the random average penalised output in the $N$--players’ game converges to a deterministic quantity. In our framework, \textit{with common noise}, as explained in Remark \ref{rk:conditional_common_noise}, the consumption deviations of the Agents become asymptotically independent, conditionally to the common noise. Therefore, a risk--averse Principal does not become risk--neutral in our case, and remains impacted by the residual risk arising from the common noise. Nevertheless, we will consider both cases of a risk--averse and a risk--neutral Principal.

\vspace{0.5em}
As a consequence, similarly to \cite{elie2019tale} and \cite{aid2018optimal}, given a contract $\xi$, the performance criterion of the Principal is defined by
\begin{align}\label{eq:principal_value}
	J^P_0 (\xi, \P) &:= \mathbb{E}^{\P} \bigg[ U^P \bigg( - \mathbb{E}^{\P} \bigg[\xi+\int_0^Tg(X_s)\mathrm{d}s+\frac \theta 2\int_0^T\mathrm{d}\langle X\rangle_s\bigg| \Fc^\circ_T\bigg] \bigg)  \bigg], \text{ for any } \mathbb P\in \mathcal P,
\end{align}
where the function $U^P: \R \longrightarrow \R$ is the Principal's utility function and $\theta$ is a positive constant representing the costs induced by the quadratic variation of the consumption, accounting for the limited flexibility of production.

\medskip 

Anticipating the results we obtain in Section \ref{sec:consumer_problem}, for a contract $\xi \in \Xi$, any mean--field equilibrium $(\P^\star,\mu^\star) \in \Mc^{\star} (\xi)$ will give the same utility to the consumers, since they all have the same characteristics. Furthermore, in the absence of limited liability in our model (compensations $\xi$ need not be non--negative), the participation constraints of the consumers is going to be saturated, meaning that any optimal contract will provide them exactly their reservation utility level. Therefore, in the case where an optimal contract would lead to several possible equilibria, the consumers will be indifferent to the specific one chosen, implying that we can reasonably assume, as in the standard moral hazard literature, that the Principal can maximise her utility by choosing the optimal equilibrium for her. This leads to the following second best contracting problem
\begin{align*}
    V_0^P := \; \sup_{\xi \in \Xi} \; \sup_{(\P,\mu) \in \Mc^{\star} (\xi)} \; J_0^P (\xi, \P),
\end{align*}
with the usual convention $\sup_{\varnothing} = - \infty$. Notice that for the contracts in the class we will end up considering, there is only one possible equilibrium, which makes the above issue not really central to our analysis.

\medskip

In the moral hazard contracting problem considered here, the Principal has an interest in giving a contract for which there is at least a mean--field equilibrium, otherwise, by convention, her utility is equal to $- \infty$. Thus, contracts $\xi \in \Xi$ such that $\Mc^{\star} (\xi) = \varnothing$ will never be offered by the Principal, meaning that we can implicitly assume that there always will be an optimal response from the consumers to a contract proposed by the Principal. Moreover, since the contract has to satisfy the participation constraint, there exists $(\P^\star,\mu^\star)\in\Mc^\star(\xi)$ such that $V_0^A(\xi,\mu^\star)\geq R_0$. The set $\Xi$ of eligible contracts is now formally defined. We define similarly $\Xi^\circ$ and $V_0^{P,\circ}$.

\medskip
The study of the first--best contracting problem, \textit{i.e.}, when the Principal can observe in continuous--time the efforts of the Agents and thus can index the contract on the latter, is {\color{black} postponed} to Section \ref{sec:firstbest}.

\section{Agent's problem}\label{sec:consumer_problem}

We consider for now that the Principal only observes $X$ and $\widehat \mu$, which means that she is offering a $\F^{\textnormal{obs}}$--measurable contract as in \eqref{contract:dependency}. For a given conditional distribution $\widehat\mu$ for the other players, we consider the dynamic version of the value function of the representative consumer, $V_t^A$, which satisfies $V_0^A = V_0^A (\xi, \widehat \mu), \; \text{and} \; V_T^A = U^A (\xi (X, \widehat \mu))$. From this definition, we notice that the following explicit relationship between the payoff and the terminal value function holds
\begin{align}\label{eq:xiT}
    \xi(X, \widehat \mu) = - \dfrac{1}{R_A} \ln \big( - V_T^{A} \big).
\end{align}

In this section, we will start, for a given contract and given efforts chosen by other consumers, by introducing the appropriate Hamiltonian functional, which will allow to first compute formally the optimal response of the consumer. Intuitively, this Hamiltonian appears by applying the chain rule with common noise defined \citet*[Theorem 4.17]{carmona2018probabilisticII} to the dynamic value function of the consumer and considering the associated Master equation. Our next step is then to derive a class of so--called \textit{revealing} contracts, thus extending to a general mean--field game framework the main arguments of \cite{cvitanic2018dynamic}, which considered general moral hazard problems with one agent, and \cite{elie2019tale}, which considered mean--field game moral hazard problems where the agents controlled only the drift of the output process $X$. Informally, the class of revealing contract is obtained by still using the chain rule with common noise, but applying it to a transformed function of the consumer's dynamic value function, defined by \eqref{eq:xiT}. {\color{black} We insist on the fact that the analysis we make in the following subsection, to find the Agent's Hamiltonian as well as the relevant form of contracts, is informal. Indeed, we consider the Markovian framework, that is to say that we suppose that the Agent's dynamic value function at time $t$ only depends on $X_t$ and $\widehat \mu_t$, where here, $\widehat \mu_t$ is the conditional law of $\widehat X_t$ (and not of the paths of $\widehat X$ up to time $t$) knowing the common noise. This framework allows us to apply the chain rule with common noise, defined in \cite[Theorem 4.17]{carmona2018probabilisticII}. Nevertheless, the analysis we make, though informal at this point, relies strongly on recent progresses on the dynamic programming approach to the control of McKean--Vlasov SDEs and will rigorously be justified later in the paper, mainly in Appendix \ref{sec:technical_proofs}. Indeed, considering simple contracts inspired by the Markovian framework allows us to calculate the optimal efforts of the representative Agent and the associated mean--field equilibrium (see Theorem \ref{thm:mfe}) whose proof is based on the theory of second--order BSDEs (2BSDEs for short). The main result of this paper, mainly that the restriction to this type of so--called revealing contracts is in fact without loss of generality, is postponed to the next section (see Theorem \ref{thm:main}).}

{\color{black}
\subsection{Intuition from the Markovian framework}

One of the cornerstones of the approach to continuous--time moral hazard problems pioneered by Sannikov \cite{sannikov2008continuous}, and studied in full generality in \citet*{cvitanic2018dynamic}, is to obtain an appropriate probabilistic representation for incentive--compatible contracts. The goal of this section is to use informal dynamic programming type arguments to deduce such a representation in a context where a continuum of Agents with mean--field interactions is involved, each of them being able to control the volatility of the output process, and to then prove that mean--field equilibria are easily accessible for this class of contracts.

\medskip

We fix throughout this section a probability $\widehat \P \in \widehat \Pc$, chosen by other consumers on their own canonical space $\widehat \Omega$. Using Definition \ref{def:Pc}, we denote by $\widehat \P_t^\omega$ the r.c.p.d. for $\P$ given $\Fc_t^\circ$, and $\widehat \mu_t$ the associated conditional law of $\widehat X_{t \wedge \cdot}$. Almost by definition of $\widehat \P \in \widehat \Pc$, more precisely by Lemma \ref{lemma:rep}, there exists a control process representing the effort of other consumers, denoted $\widehat \nu := ( \widehat \alpha, \widehat \beta) \in \widehat \Uc = \widehat \Ac \times \widehat \Bc$, such that the dynamic of their deviation consumption $\widehat X$ is:
\begin{align*}
    \drm \widehat X_t = - \widehat \alpha_t \cdot \mathbf{1}_d \drm t + \sigma \big( \widehat \beta_t \big) \cdot \drm \widehat{W}_t + \sigma^\circ \drm W^{\circ}_t.
\end{align*}
Fixing $\widehat \P \in \widehat \Pc$ thus implies that $(\widehat \mu, \widehat \nu) \in \Pc(\Cc_T) \times \widehat \Uc$ are fixed too.

\medskip

Intuitively, in view of \eqref{eq:xiT}, we expect that a contract $\F^\textnormal{obs}$--measurable is any terminal value of the following process, as a function of $t$, the path of $X$ up to $t$, and $\widehat \mu_t$, the conditional law of $\widehat X_{t \wedge \cdot}$ with respect to the common noise:
\begin{align}\label{eq:contract_non_markov}
    \xi_t &:= - \dfrac{1}{R_A} \ln \big( - V_t^A \big) = u^A (t, X_{t\wedge \cdot}, \widehat \mu_t),
    % \xi_t &:= - \dfrac{1}{R_A} \ln \big( - V_t^A \big) = u^A (t, X_{t\wedge \cdot}, \widehat \mu_t), \;
    % \text{where } u^A := - \dfrac{1}{R_A} \ln (- \cdot) \circ v^A.
\end{align}
where $u^A : [0,T] \times \Cc_T \times \Pc(\Cc_T) \longrightarrow \R$. Therefore, the continuation utility $V_t^{A}$ of the consumer, given a contract $\xi\in\Xi$ and efforts of other consumers subsumed by the distribution $\widehat \mu$, may be written as
\begin{align*}
    V^{A}_t := v^A (t, X_{t \wedge \cdot}, \widehat \mu_{t}), \; \text{where} \; v^A := - \erm^{- R_A \cdot} \circ u^A,
\end{align*}
\textit{i.e.}, the process $V^{A}$ at time $t$ depends on $t$, on the path history of $X$, and on the conditional law $\widehat \mu$ of deviation consumption $\widehat X$ of others. Indeed, the contract being only indexed on $X$ and $\widehat \mu$, the continuation utility should not depend on all the information contained in $\F$ and $\widehat \F$. 

\medskip

To find the relevant form of contracts, the intuition is to focus on the Markovian framework, \textit{i.e.} when the function $v^A$, and thus the function $u^A$, only depend on $t$, $X_t$, and $\widehat \mu_t$, where here, $\widehat \mu_t$ is the conditional law of $\widehat X_t$ (and not of the paths of $\widehat X$ up to time $t$) with respect to the common noise. Note that in this particular case, both $v^A$ and $u^A$ are functions from $[0,T] \times \R \times \Pc(\R)$ with values in $\R$.

\subsubsection{Consumer's Hamiltonian}\label{ss:hamiltonian}

In the Markovian framework, \textit{i.e.}, when $V_t^A$ only depends on $X_t$ and $\widehat \mu_t$ where here, $\widehat \mu_t$ is the conditional law of $\widehat X_t$ (and not of the paths of $\widehat X$ up to time $t$) knowing the common noise, and if $v^A$ is smooth enough in the sense of \citet*[Section 4.3.4]{carmona2018probabilisticII}, we can apply the chain rule with common noise (for function of both the state and the measure, see \cite[Theorem 4.17]{carmona2018probabilisticII}) to $v^A : [0,T] \times \R \times \Pc(\R) \longrightarrow \R$:
\begin{align}\label{eq:ito_mesure_v}
    v^{A} (t, X_{t}, \widehat \mu_{t}) = &\ V_0^{A} 
    + \int_0^t \partial_s v^{A} (s, X_s, \widehat \mu_s) \drm s 
    + \int_0^t \partial_{x} v^{A} (s, X_s, \widehat \mu_s) \drm X_s
    + \int_0^t \widehat{\mathbb{E}}^{\widehat \P_s} \Big[ \partial_{\mu} v^{A}  (s, X_s, \widehat \mu_s) \big( \widehat X_s \big) \drm \widehat X_s \Big] \nonumber \\
    &+ \dfrac{1}{2} \int_0^t \partial_{x x}^2 v^{A}  (s, X_s, \widehat \mu_s) \drm \langle X \rangle_s 
    + \int_0^t \widehat{\E}^{\widehat \P_s} \Big[ \partial_{x} \partial_{\mu} v^{A} (s, X_s, \widehat \mu_s) \big( \widehat X_s \big) \drm \langle X, \widehat{X} \rangle_s \Big]
     \nonumber \\
    &+ \dfrac{1}{2}  \int_0^t \widehat{\E}^{\widehat \P_s} \Big[ \partial_v \partial_{\mu} v^{A} (s, X_s, \widehat \mu_s) \big( \widehat X_s \big) \drm \langle \widehat X \rangle_s \Big]
    + \dfrac{1}{2} \int_0^t \widehat{\E}^{\widehat \P_s} \widecheck{\E}^{\widecheck \P_s} \Big[ \partial_{\mu}^2 v^{A} (s, X_s, \widehat \mu_s) \big( \widehat X_s, \widecheck X_s \big) \drm \langle \widehat X, \widecheck{X} \rangle_s \Big],
\end{align}
where $\widecheck{X}$ is a copy of $\widehat X$ in the sense of Definition \ref{def:copy_process}. One may note that this particular It\=o's formula involves derivatives with respect to a measure: we refer to \cite[Section 4.3.4]{carmona2018probabilisticII} for a rigorous definition of these types of derivatives, and we denote by $\widetilde \Lc$ (resp. $\widetilde \Lc^2$)  the set of Borel measurable functionals from $\R$ (resp. $\R^2$) into $\R$, to consider the latter.

\medskip

Intuitively, as in classical control theory, the Hamiltonian of the representative Agent should be composed of all the drift terms appearing in the previous It\=o's expansion. Therefore, by computing the quadratic variations and covariations between the deviation consumption of the representative consumer and the others, mainly:
\begin{align*}
    &\ \drm \langle X \rangle_t = \big( \Sigma \big( \beta^\P_t \big) + (\sigma^\circ)^2 \big) \drm t, \;
    \drm \langle \widehat X \rangle_t = \big( \Sigma \big( \widehat \beta_t \big) + ( \sigma^\circ)^2 \big) \drm t, \;
    \drm \langle X, \widehat X \rangle_t = \drm \langle \widehat X, \widecheck{X} \rangle_t = ( \sigma^\circ)^2 \drm t,
\end{align*}
and using in addition the dynamics of $X$ and $\widehat X$, one obtain the following form for the hamiltonian, for $(t, x, y) \in [0,T] \times \R \times \R$, $p := (z, z^\mu, \gamma, \gamma^\mu, \gamma^{\mu,1}, \gamma^{\mu,2}) \in \R \times \widetilde \Lc \times \R \times \widetilde \Lc \times \widetilde \Lc^2 \times \widetilde \Lc$, and computed along $(\widehat \mu, \widehat \nu) \in \Pc(\Cc_T) \times \widehat \Uc$:
\begin{align*}
    H (t, x, y, p, \widehat \mu_t, \widehat \nu_t) := \sup_{v \in U} h (t, x, y, p, \widehat \mu_t, \widehat \nu_t, v),
\end{align*}
where, for $v \in U$,
\begin{align*}
    h (t, x, y, p, \widehat \mu_t, \widehat \nu_t, v) := &
    - R_A ( c (v) + f (x)) y
    - z a \cdot \mathbf{1}_d
    - \widehat{\mathbb{E}}^{\widehat{\P}_t} \big[ z^{\mu}  \big( \widehat X_{t} \big) \widehat \alpha_t \cdot \mathbf{1}_d \big]
    + \dfrac{1}{2} \gamma \Big( \Sigma (b) + (\sigma^{\circ})^2 \Big)  \\
    &+ (\sigma^{\circ})^2 \widehat{\mathbb{E}}^{\widehat{\P}_t} \big[ \gamma^{\mu} \big( \widehat X_{t}  \big) \big]
    + \dfrac{1}{2} \widehat{\mathbb{E}}^{\widehat{\P}_t} \Big[ \gamma^{\mu,1} \big( \widehat X_{t}  \big) \big( \Sigma \big( \widehat \beta_t \big) + ( \sigma^{\circ})^2 \big) \Big] 
    + \dfrac{1}{2}  \big( \sigma^{\circ} \big)^2 \widehat{\mathbb{E}}^{\widehat{\P}_t} \widecheck{\mathbb{E}}^{\widecheck{\P}_t} \big[ \gamma^{\mu,2} \big( \widehat X_{t}, \widecheck{X}_{t} \big) \big].
\end{align*}
Following the classical reasoning in control theory, the value function should satisfy the following Hamilton--Jacobi--Bellman (HJB) equation:
\begin{align}\label{eq:HJBva}
    & - \partial_t v^A(t,x,\widehat \mu_t) - H \big( t, x, \widehat \mu_t, v^{A} (t, x, \nabla v^{A} (t, x, \widehat \mu_t), \Delta^2 v^{A} (t, x, \widehat \mu_t), \widehat \mu_t, \widehat \nu_t) = 0,
\end{align}
where
\begin{align*}
    \nabla v^{A} (t, x, \widehat \mu_t) &:= \big( \partial_{x} v^{A} (t, x, \widehat \mu_t), \partial_{\mu} v^{A}  (t, x, \widehat \mu_t) \big), \\ \text{ and } \; 
    \Delta^2 v^{A} (t, x, \widehat \mu_t) &:= \big( 
    \partial_{x x}^2 v^{A}  (t, x, \widehat \mu_t),
    \partial_{x} \partial_{\mu} v^{A} (t, x, \widehat \mu_t), 
    \partial_v \partial_{\mu} v^{A} (t, x, \widehat \mu_t),
    \partial_{\mu}^2 v^{A} (t, x, \widehat \mu_t) \big).
\end{align*}

\subsubsection{Toward a relevant form of contract}

By still considering the Markovian framework, and assuming that we can apply to the function $u^A$ defined by \eqref{eq:contract_non_markov} the chain rule with common noise under $\Cc^{1,2,2}$--regularity, defined in \cite[Theorem 4.17]{carmona2018probabilisticII}, we can obtain a formula similar to \eqref{eq:ito_mesure_v} for $u^A$. 
By computing\footnote{{\color{black} Once again, we refer to \cite[Section 4.3.4]{carmona2018probabilisticII} for computations rules of derivatives, but to give an idea, we have: 
\begin{align*}
    &\partial_{\mu} u^A (t, x, \widehat \mu_t) (\widehat x) = - \dfrac{1}{R_A V_t^A} \partial_{\mu} v^A (t, x, \widehat \mu_t) (\widehat x), \text{ and }
    \partial_v \partial_{\mu} u^A (t, x, \widehat \mu_t) (\widehat x) = - \dfrac{1}{R_A V_t^A} \partial_v \partial_{\mu} v^A (t, x, \widehat \mu_t) (\widehat x).
\end{align*}
}} the partial derivatives of $u^A$ in terms of the partial derivatives of $v^A$, we obtain after some tedious but simple computations:
\begin{align}\label{eq:contract_markov}
    \xi_t = &\ - \dfrac{1}{R_A} \ln \big(-V_0^A \big)
    + \int_0^t \partial_s u^A (s, X_{s}, \widehat \mu_s) \drm s 
    + \int_0^t Z_s \drm X_s
    + \int_0^t \widehat{\mathbb{E}}^{\widehat{\P}_s} \big[ Z^\mu_s (\widehat X_{s}) \drm \widehat X_s \big] 
    + \dfrac{1}{2} \int_0^t \big(  \Gamma_s + R_A Z_s^2 \big) \drm \langle X \rangle_s \nonumber \\
    &+ \dfrac{1}{2} \int_0^t  \widehat{\mathbb{E}}^{\widehat{\P}_s} \Big[ \Gamma^{\mu,1}_s (\widehat X_{s}) \drm \langle \widehat X \rangle_s \Big] 
    + \dfrac{1}{2} \int_0^t \widehat{\mathbb{E}}^{\widehat{\P}_s} \widecheck{\mathbb{E}}^{\widecheck{\P}_s} \Big[ \Big( \Gamma_s^{\mu,2} \big(\widehat X_s, \widecheck X_s \big) + R_A Z_s^\mu \big( \widehat X_s \big) Z^\mu_s \big( \widecheck X_s \big) \Big) \drm \big\langle \widehat X, \widecheck{X} \big\rangle_s \Big] \nonumber \\
    &+ \int_0^t  \widehat{\mathbb{E}}^{\widehat{\P}_s} \Big[ \Big( \Gamma^{\mu}_s \big( \widehat X_{s} \big) + R_A Z_s Z^\mu_s \big( \widehat X_{s} \big) \Big) \drm \big\langle X, \widehat{X} \big\rangle_s \Big],
\end{align}
where the process $( Z, Z^\mu, \Gamma, \Gamma^\mu, \Gamma^{\mu,1}, \Gamma^{\mu,2})$ takes values in $\R \times \widetilde \Lc \times \R \times \widetilde \Lc \times \widetilde \Lc^2 \times \widetilde \Lc$ and is defined for all $t \in [0,T]$ by:
\begin{align*}
    \big( Z_t, Z_t^\mu, \Gamma_t, \Gamma_t^\mu, \Gamma_t^{\mu,1}, \Gamma_t^{\mu,2} \big) := - \dfrac{1}{R_A V_t^A} \big( \partial_x v^A, \partial_{\mu} v^A , \partial_{x x}^2 v^{A},
    \partial_{x} \partial_{\mu} v^{A}, 
    \partial_v \partial_{\mu} v^{A},
    \partial_{\mu}^2 v^{A}\big) (t, X_t, \widehat \mu_t).
\end{align*}
Using the HJB equation  satisfied by $v^A$, see \eqref{eq:HJBva}, we can state the HJB equation  satisfied by $u^A$:
\begin{align}\label{eq:HJBua}
    &- \partial_t u^A(t,x,\widehat \mu_t) - \widetilde H \big( t, x, Z_t, Z_t^\mu, \Gamma_t, \Gamma_t^\mu, \Gamma_t^{\mu,1}, \Gamma_t^{\mu,2}, \widehat \mu_t, \widehat \nu_t) = 0,
\end{align}
where $\widetilde H$ is a slightly modified version of the initial Hamiltonian $H$, more convenient when dealing with CARA utility functions, and satisfying for $(\widehat \mu, \widehat \nu) \in \Pc(\Cc_T) \times \widehat \Uc$, $(t, x) \in [0,T] \times \R$, and $(z, z^\mu, \gamma, \gamma^\mu, \gamma^{\mu,1}, \gamma^{\mu,2}) \in \R \times \widetilde \Lc \times \R \times \widetilde \Lc \times \widetilde \Lc^2 \times \widetilde \Lc$,
\begin{align}\label{eq:hamiltonian_markovian}
    \widetilde H (t, x, z, z^\mu, \gamma, \gamma^{\mu, 1}, \gamma^{\mu,2}, \gamma^{\mu}, \widehat \mu_t, \widehat \nu_t) = \dfrac{1}{2} H_d (z) + \dfrac{1}{2} H_v(\gamma) + H_c (x, \gamma) + H_\circ (z^\mu, \gamma^{\mu, 1}, \gamma^{\mu,2}, \gamma^{\mu}, \widehat \mu_t,  \widehat \nu_t), 
\end{align}
where
\begin{align*}
    H_d (z) := &- \inf_{a \in A} \big\{2 z a \cdot \mathbf{1}_d + c_\alpha(a) \big\}, \;
    H_v(\gamma) := - \inf_{ b \in B} \big\{c_{\beta} (b) - \gamma \Sigma(b) \big\}, \;
    H_c (x, \gamma) := \dfrac{1}{2} \gamma (\sigma^{\circ})^2 + f(x),
\end{align*}
\vspace{-1.2em}
\begin{align*}
    \text{and } \; H_\circ (z^\mu, \gamma^{\mu, 1}, \gamma^{\mu,2}, \gamma^{\mu}, \widehat \mu_t, \widehat \nu_t) :=& 
    - \widehat{\mathbb{E}}^{\widehat{\P}_t} \Big[ z^{\mu}\big( \widehat X_{t} \big) \widehat \alpha_t \cdot \mathbf{1}_d \Big]
    + (\sigma^{\circ})^2 \widehat{\mathbb{E}}^{\widehat{\P}_t} \Big[ \gamma^{\mu}\big( \widehat X_{t} \big) \Big]
    + \dfrac{1}{2} \widehat{\mathbb{E}}^{\widehat{\P}_t} \Big[ \gamma^{\mu,1}\big( \widehat X_{t} \big) \big( \Sigma \big(\widehat \beta_t \big) + (\sigma^{\circ})^2 \big) \Big] \\
    & + (\sigma^{\circ})^2 \dfrac{1}{2} \widehat{\mathbb{E}}^{\widehat{\P}_t} \widecheck{\mathbb{E}}^{\widecheck{\P}_t} \Big[  \gamma^{\mu,2} \big( \widehat X_{t} , \widecheck{X}_{t}  \big) \Big].
\end{align*}
Therefore, the Hamiltonian of the representative consumer in this case consists of four parts. The first three, $H_d$, $H_v$ and $H_c$, are the classical parts for drift and volatility control, which do not depend on the efforts and the distribution of other players' states. The last part, $H_\circ$, does depend on the law and the efforts of others, and act as a constant part for the representative consumer, since he cannot control it. Note that the optimisers are given by
    \begin{align}\label{eq:optimal_effort}
        a^{k,\star} (z) := \rho^k (z^{-} \wedge A_{\textnormal{max}}) \; \text{ and } \; b^{k,\star} (\gamma) := 1 \wedge \big( \lambda^k \gamma^{-} \big)^{\frac{-1}{\eta^k + 1}} \vee B_{\textnormal{min}}, \; \text{ for } k=1, \dots, d. 
    \end{align}
We thus claim that, in our framework, the Hamiltonian of the representative consumer should be somehow a path dependent version of  \eqref{eq:hamiltonian_markovian}, and that the relevant contracts should be of the form \eqref{eq:contract_markov}, parametrised by a process $( Z, Z^\mu, \Gamma, \Gamma^\mu, \Gamma^{\mu,1}, \Gamma^{\mu,2})$. Nevertheless, some modifications are necessary, mainly by considering a path dependent version, but also some simplifications are possible. In particular, by writing explicitly the quadratic variations of last three integrals of \eqref{eq:contract_markov}, we can show that the terms indexed by $\Gamma^\mu$, $\Gamma^{\mu,1}$ and $\Gamma^{\mu,2}$ can be simplified with some part of the Hamiltonian $\widetilde H$, and are therefore unnecessary. We provide details in the next section.}

\subsection{Solving the mean--field thanks to simple contracts}

{\color{black}The intuition in the Markovian framework developed in the previous subsection allows us to intuit the form of revealing contracts in our framework, given below in Definition \ref{def:simple_contracts}. In particular, the relevant form of contract is inspired by \eqref{eq:contract_markov}, adapted for a non--Markovian framework, and noticing that some simplification are possible. Therefore, starting from a contract form indexed by a tuple of processes $( Z, Z^\mu, \Gamma, \Gamma^\mu, \Gamma^{\mu,1}, \Gamma^{\mu,2})$, we finally obtain that the tuple of process $( Z, Z^\mu, \Gamma)$ should be sufficient to parametrise the relevant contract. 
This type of contracts, indexed only by $( Z, Z^\mu, \Gamma)$, then allows us to calculate the optimal efforts of the representative Agent and the associated mean--field equilibrium (see Theorem \ref{thm:mfe}) whose proof is based on the theory of 2BSDEs. The main result of this paper, mainly that the restriction to this type of so--called revealing contracts is in fact without loss of generality, is postponed to the next section (see Theorem \ref{thm:main}).

\medskip

Throughout the following, we denote for simplicity, for any positive integer $n$, by $\mathcal L^n$ the set of Borel measurable functionals from $\Cc([0,T],\R^n)$ into $\R$, and $\Lc:=\Lc^1$. Indeed, instead of considering that the indexation parameter $Z^\mu$ takes values in $\widetilde \Lc$ as in the previous section, we should now consider, in our non--Markovian framework, that for all $t \in [0,T]$, the function $Z_t^\mu$ can be applied to the paths of $\widehat X$ until $t$.

\subsubsection{Simple contracts}\label{ss:contract_zmu}

Recall that in view of \eqref{eq:xiT}, we expect that a contract $\F^\textnormal{obs}$--measurable is defined by \eqref{eq:contract_non_markov}, as a function of $t$, the path of $X$ up to $t$, and $\widehat \mu_t$, the conditional law of $\widehat X_{t \wedge \cdot}$. 
Thanks to the reasoning developed in the previous section, and noticing moreover that, replacing the Hamiltonian $\widetilde H$ by its value (see \eqref{eq:hamiltonian_markovian}) in the form of the contract \eqref{eq:contract_markov}, some simplifications are possible between parts of the Hamiltonian and terms related to the quadratic variations. In particular, by setting for $(x, z, z^{\mu}, \gamma) \in \R \times \R \times \Lc \times \R$ and $(\widehat \mu, \widehat \alpha) \in \Pc(\Cc_T) \times \widehat \Ac$:
\begin{align}\label{eq:hamiltonian}
    \Hc (x, \widehat \mu_t, z, z^{\mu}, \gamma, \widehat \alpha_t) := &\ \dfrac12 H_d (z) + \dfrac12 H_v(\gamma) + H_c (x, \gamma)
    -  \widehat{\mathbb{E}}^{\widehat{\P}_t} \big[z^{\mu} \big(\widehat X_{t \wedge \cdot} \big) \widehat \alpha_t \cdot \mathbf{1}_d \big],
\end{align}
we obtain that the contract should only be parametrised by a process $\zeta := (Z, Z^{\mu}, \Gamma)$, taking values in $\R \times \Lc \times \R$, and should satisfy
\begin{align}\label{contractform}
    \xi_t = &\ \xi_0  - \int_0^t \Hc (X_s, \widehat \mu_s, \zeta_s,  \widehat \alpha_s)  \drm s 
    + \int_0^t Z_s \drm X_s 
    + \int_0^t \widehat{\mathbb{E}}^{\widehat{\P}_s} \Big[ Z_s^{\mu} \big(\widehat X_{s \wedge \cdot}\big) \drm \widehat X_s \Big]
    + \dfrac{1}{2} \int_0^t \big( \Gamma_s + R_A Z_s^2 \big) \drm \langle X \rangle_s
    \nonumber \\
    &+ \dfrac{1}{2} R_A \int_0^t  \widehat{\mathbb{E}}^{\widehat{\P}_s} \widecheck{\mathbb{E}}^{\widecheck{\P}_s} \Big[ Z_s^{\mu} \big(\widehat X_{s \wedge \cdot} \big) Z_s^{\mu} \big(\widecheck{X}_{s \wedge \cdot} \big) \drm \big\langle \widehat X, \widecheck{X} \big\rangle_s \Big]
    + R_A \int_0^t Z_s \widehat{\mathbb{E}}^{\widehat{\P}_s}  \Big[Z_s^{\mu} \big(\widehat X_{s \wedge \cdot} \big) \drm \big\langle X, \widehat X \big\rangle_s \Big],
\end{align}
for some $\xi_0 \in \R$.

\medskip

Therefore, $\Hc$ defined by \eqref{eq:hamiltonian} is the relevant Hamiltonian in our non--Markovian framework. One may note that $\Hc$ is a simplified version of the Hamiltonian $\widetilde H$ defined by the equation \eqref{eq:hamiltonian_markovian}. Indeed, some parts of the Hamiltonian which are not controlled by the consumer simplify with some parts of the contract. Therefore, the triple $(\Gamma^{\mu,1}, \Gamma^{\mu,2},\Gamma^{\mu})$ no longer appears in the contract. We thus obtain a simplified form for the contract, only indexed by a triple $\zeta := (Z, Z^\mu, \Gamma)\in\R\times\Lc\times\R$. This triple $\zeta$ will be called the triple of payment rates. Moreover, one may note that the Hamiltonian as well as the contract for the representative Agent do not depend anymore on the other consumers' effort on the volatility, namely $\widehat \beta$, but still depends on the effort on the drift of other consumers, \textit{i.e.} $\widehat \alpha$.}

\medskip 

Though the expression \eqref{contractform} is appealing to be used as our generic contract form, it is not possible to use it directly in the context of a Principal--Agent problem with moral hazard. In fact, this form depends explicitly on the drift's effort of other consumers, namely $\widehat \alpha$, through the Hamiltonian, and this effort is not supposed to be observable, nor contractible upon, for the Principal. Nevertheless, we can overcome this difficulty by replacing $\widehat \alpha$ by the optimal drift process of other consumers, which has to be formally computed as the maximiser in the Hamiltonian denoted by $\widehat \alpha^{\star}$ and defined by \eqref{eq:optimal_effort} so that $\widehat a^{k,\star} (\widehat z) := \rho^k (\widehat z^{-} \wedge A_{\textnormal{max}}), \; k=1, \dots, d,$ where $\widehat z$ is the payment rate for the other consumers' drift effort. Indeed, at equilibrium, each consumer should consume optimally. Moreover, in our mean--field framework, the consumers are identical and indistinguishable. Therefore, the Principal will offer the same contract for all agents, that is the payment rate for drift effort will be the same for all the consumers, since no discrimination is allowed. Hence, the optimal drift process of other consumers will be $\widehat \alpha^\star(Z)$. We are thus led to consider a particular type of revealing contracts, precisely described in the definition below.

\begin{definition}[Simple contracts]\label{def:simple_contracts}
For any $\R\times\Lc\times\R$--valued $\F^{\rm obs}$--predictable process $\zeta:=(Z,Z^\mu,\Gamma)$, and any $\xi_0\in\R$, let us define the following process $\xi^{\xi_0, \zeta}$ for all $t\in[0,T]$ by
\begin{align}\label{contract_form_mu}
    \xi_t^{\xi_0,\zeta} := &\ \xi_0  - \int_0^t \Hc (X_s, \widehat \mu_s, \zeta_s, \widehat \alpha_s^\star)  \drm s 
    + \int_0^t Z_s \drm X_s 
    + \dfrac{1}{2} \int_0^t \big( \Gamma_s + R_A Z_s^2 \big) \drm \langle X \rangle_s
    + \int_0^t \widehat{\mathbb{E}}^{\widehat{\P}_s} \Big[ Z_s^{\mu} \big(\widehat X_{s\wedge\cdot}\big) \drm \widehat X_s \Big]
    \nonumber \\
    &+ \dfrac{1}{2} R_A \int_0^t  \widehat{\mathbb{E}}^{\widehat{\P}_s} \widecheck{\mathbb{E}}^{\widecheck{\P}_s} \Big[ Z_s^{\mu} \big(\widehat X_{s\wedge\cdot}\big) Z_s^{\mu} \big(\widecheck{X}_{s\wedge\cdot}\big) \drm \big\langle \widehat X, \widecheck{X} \big\rangle_s \Big]
    + R_A \int_0^t Z_s \widehat{\mathbb{E}}^{\widehat{\P}_s}  \Big[Z_s^{\mu} \big(\widehat X_{s\wedge\cdot}\big) \drm \big\langle X, \widehat X \big\rangle_s \Big],
\end{align}
where the function $\Hc$ is defined by \eqref{eq:hamiltonian}. We let then $\Vc$ the set of $\R\times\Lc\times\R$--valued $\F^{\rm obs}$--predictable process $\zeta$ such that 
%there exists some $p>1$ with
\[
\sup_{\P\in\Pc}\E^\P\bigg[\sup_{0\leq t\leq T}\mathrm{e}^{pR_A|\xi_t^{\xi_0,\zeta}|}\bigg]<+\infty,
\]
where $p$ is the same as in \textnormal{Condition \eqref{eq:integxi}}. We call random variables of the form $(\xi_T^{\xi_0,\zeta})$, for $(\xi_0,\zeta)\in\R\times\Vc$, simple contracts, and denote the corresponding set by $\Xi_{\mathrm{S}}$. Moreover, for any $\R\times\Lc\times\R$--valued $\F$--predictable process $\zeta:=(Z,Z^\mu,\Gamma)$, we will denote by $\overline \zeta = (Z, \overline Z^\mu, \Gamma)$ the $\R^3$--valued $\F$--predictable process, with $\overline Z^\mu_t = \widehat \E^{\widehat \P_t} [Z_{t}^\mu (\widehat X_{t \wedge \cdot})]$, for $t \in [0,T]$. We will say that $\overline \zeta \in \overline \Vc$ if $\zeta \in \Vc$.
\end{definition}

\begin{remark}
Notice that the integrability requirement in the definition of the set $\Vc$ is rather implicit. It is however clear that $\Vc$ is not empty as it contains trivially constant processes, since the drift and the volatility of $X$ are always bounded. Besides, this is exactly the integrability we need to be able to solve the {\rm MFG} for the agents given a contract in $\Xi_{\mathrm{S}}$, as the proof of \textnormal{Theorem \ref{thm:mfe}} below will make clear.
\end{remark}

\subsubsection{Interpretation of the form of contracts}

{\color{black} The form of contracts given in Definition \ref{def:simple_contracts}} is mainly composed of two parts: one is an indexation on the process controlled by the consumer, that is to say his deviation consumption, the other one is an indexation on other consumers through the law $\widehat \mu$. In particular, similarly to \cite{aid2018optimal}, the contract has a linear part in the level of consumption deviation $X$ and the corresponding quadratic variation $\langle X \rangle$, with linearity coefficients $Z$ and $\Gamma$. This part of the contract is the classic contract for drift and volatility control. The constant part is slightly different from the usual one in the moral hazard framework. In fact, we can divide it into three integrals:
\begin{align*}%\label{integraleH}
    \int_0^T \Hc (X_s, \widehat \mu_s, \zeta_s, \widehat \alpha_s^{\star}) \drm s  = \int_0^T \bigg(  \dfrac12 H_d (Z_s) +  \dfrac12  H_v (\Gamma_s) + f (X_s) \bigg) \drm s 
    - \int_0^T \widehat{\mathbb{E}}^{\widehat \P_s} \big[Z_s^{\mu} (\widehat X_{s \wedge \cdot}) \widehat \alpha_s^{\star} \cdot \mathbf{1}_d \big] \drm s
    + \dfrac12 ( \sigma^\circ)^2 \int_0^T \Gamma_s \drm s.
\end{align*}
The first one represents the certainty equivalent of the utility gain of the consumer that can be achieved by an optimal response to the contract, and is thus subtracted from the Principal's payment, in agreement with usual Principal--Agent moral hazard type of contracts. Moreover, due to the risk--aversion of the consumer, the infinitesimal payment $Z_t \drm X_t$ must be compensated by the additional payment $\frac12 R_A Z_t^2 \drm \langle X \rangle_t$. Following the same reasoning, the second integral and the additional payment 
\begin{align*}
    \dfrac{1}{2} R_A \big( \sigma^{\circ} \big)^2 \int_0^t  \Big( \widehat{\E}^{\widehat \P_s} \big[ Z_s^{\mu} (\widehat X_{s \wedge \cdot})  \big] \Big)^2 \drm s,
\end{align*}
in \eqref{contract_form_mu} are compensations for the infinitesimal payment $ \widehat{\E}^{\widehat \P_s} \big[ Z_s^{\mu} (\widehat X_{s \wedge \cdot}) \drm \widehat X_s \big]$ indexed on others, and the last integral is a compensation for the covariation induced by the two infinitesimal payments. In summary, the Principal will choose the triple of controls $\zeta = (Z, Z^\mu, \Gamma)$ where the payment rates $(Z, \Gamma)$ index the contract on the deviation consumption of the considered consumer, in agreement with usual Principal--Agent moral hazard type of contracts for drift and volatility control, and the payment rate $Z^\mu$ indexes the contract on the behaviour of other consumers, represented by the conditional law $\widehat \mu$.

\begin{remark}\label{rk:contract_w0_observable}
    Until now, we supposed that the Principal could not observe the common noise, or at least was not allowed to directly index compensations on it. Whenever she can do so, we notice that the contract defined by \eqref{contractform} can be written in the following way
    \begin{align*}
        \xi_t = &\ \xi_0 - \int_0^t \Hc^\circ (X_s, Z_s, \Gamma_s)  \drm s 
        + \int_0^t Z_s \drm X_s 
        + \sigma^{\circ} \int_0^t \widehat{\mathbb{E}}^{\widehat{\P}_s} \big[ Z_s^{\mu} \big(\widehat X_{s\wedge\cdot}\big) \big] \drm W^{\circ}_s 
        + \dfrac{1}{2} \int_0^t \left( \Gamma_s + R_A Z_s^2 \right) \drm \langle X \rangle_s \nonumber \\
        &+ \dfrac{1}{2} R_A \int_0^t \widehat{\mathbb{E}}^{\widehat{\P}_s} \widecheck{\mathbb{E}}^{\widecheck{\P}_s} \big[  Z_s^{\mu} \big(\widehat X_{s\wedge\cdot}\big) Z_s^{\mu} \big(\widecheck{X}_{s\wedge\cdot}\big) \drm \big\langle \widehat X, \widecheck{X} \big\rangle_s \big]
        + R_A  \int_0^t Z_s \widehat{\mathbb{E}}^{\widehat{\P}_s}  \big[Z_s^{\mu} \big(\widehat X_{s\wedge\cdot}\big) \drm \big\langle X, \widehat X \big\rangle_s \big], 
    \end{align*}
    where for $(x, z, \gamma) \in \R \times \R \times \R$,
    \begin{align}\label{eq:Hc_circ}
        \Hc^\circ (x, z, \gamma) = \dfrac{1}{2} H_d (z) + \dfrac{1}{2} H_v(\gamma) + H_c (x, \gamma),
    \end{align}
    and thus does not depend on the others' effort $\widehat \alpha$ anymore. We can even go further in the simplifications by noticing that, given the common noise, $\widehat X$ and $\widecheck X$ are independent. Hence, recalling the notation $ \overline Z_s^{\mu} = \widehat{\mathbb{E}}^{\widehat{\P}_s} \big[ Z_s^{\mu} \big(\widehat X_{s\wedge\cdot}\big) \big]$, we have
    \begin{align*}
        \int_0^t \widehat{\mathbb{E}}^{\widehat{\P}_s} \widecheck{\mathbb{E}}^{\widecheck{\P}_s} \big[ Z_s^{\mu} \big(\widehat X_{s\wedge\cdot}\big) Z_s^{\mu} \big(\widecheck{X}_{s\wedge\cdot}\big) \drm \big\langle \widehat X, \widetilde{X} \big\rangle_s \big] 
        &= \big( \sigma^{\circ} \big)^2 \int_0^t \widehat{\mathbb{E}}^{\widehat{\P}_s} \big[ Z_s^{\mu} \big(\widehat X_{s\wedge\cdot}\big) \big] \widecheck{\mathbb{E}}^{\widecheck{\P}_s} \big[ Z_s^{\mu} \big(\widecheck{X}_{s\wedge\cdot}\big) \big] \drm s 
        = \big( \sigma^{\circ} \big)^2 \int_0^t \big( \overline Z_s^{\mu} \big)^2 \drm s,  \\
  \int_0^t \widehat{\mathbb{E}}^{\widehat{\P}_s} \big[ Z_s  Z_s^{\mu} \big(\widehat X_{s\wedge\cdot}\big) \drm \big\langle X, \widehat X \big\rangle_s \big]
        &= \big( \sigma^{\circ} \big)^2 \int_0^t Z_s \widehat{\mathbb{E}}^{\widehat{\P}_s} \big[ Z_s^{\mu} \big(\widehat X_{s\wedge\cdot}\big) \big] \drm s
        = \big( \sigma^{\circ} \big)^2 \int_0^t Z_s \overline Z_s^{\mu} \drm s.
    \end{align*}
    Therefore, the actual compensation parameter of the contract in this case is the triple $\overline{\zeta} := \big(Z, \overline{Z}^\mu, \Gamma \big) \in \overline{\Vc}$ and the form of contracts becomes
    \begin{align}\label{eq:xi_simplify}
        \xi_0 - \int_0^t \Hc^\circ (X_s, Z_s, \Gamma_s)  \drm s 
        + \int_0^t Z_s \drm X_s 
        + \sigma^{\circ} \int_0^t \overline{Z}_s^{\mu} \drm W^{\circ}_s 
        + \dfrac{1}{2} \int_0^t \big( \Gamma_s + R_A Z_s^2 \big) \drm \langle X \rangle_s + \dfrac{1}{2} R_A  ( \sigma^{\circ} )^2 \int_0^t  \overline Z_s^{\mu} \big(\overline Z_s^{\mu}+ 2 Z_s\big)  \drm s.
    \end{align}
    
The previous form of contract is nothing more than a rewriting of the one given by \eqref{contractform}. This shows that indexing on the conditional law is actually a hidden indexing on the common noise: the compensation term in the contract depending on others is rewritten as a term depending only on the common noise. Therefore, in the case where the producer is allowed to use $W^\circ$, she can directly offer this type of contracts. Otherwise, if there are some regulatory reasons preventing her from using it directly in the contract, she can offer the contract in \textnormal{Definition \ref{def:simple_contracts}}. This being said, when the Principal uses contracts in $\Cc^\circ$, we need to add the common noise as a state variable in the value function of the representative consumer, and can then show similarly that the contract is indexed on the common noise through a parameter $Z^{\circ}$, and the payment rate $\zeta^\circ$ chosen by the Principal is measurable with respect to the natural filtration generated by $X$, $W^{\circ}$ and $\widehat \mu$. Hence, if the Principal observes the common noise, we simply have to extend the space of controls chosen by the Principal. In fact, the form of optimal contracts will be the same, leading to the same effort of the consumers and the same utility for the Principal. We refer to \textnormal{Section \ref{sec:observable_common_noise}} for the detailed contract and the resolution of optimal contracting in this particular case.
\end{remark}

The previous remark underlines the fact that if $\sigma^\circ = 0$, a simple contract in Definition \ref{def:simple_contracts} is exactly a standard contract for drift and volatility control (see \cite{aid2018optimal, cvitanic2017moral, cvitanic2018dynamic}). Therefore, in absence of common noise, it is straightforward to conclude that considering contracts indexed on the consumption of the population of consumers do not improve the results given in \cite{aid2018optimal}. We refer to Section \ref{sec:classical_contract} for more details on this result.

\subsubsection{Solving the mean--field game}
By considering simple contracts, we are able to compute the optimal efforts of the representative Agent, which were given informally by \eqref{eq:optimal_effort}. Intuitively, maximising the Hamiltonian given by \eqref{eq:hamiltonian} is sufficient to obtain optimal efforts, but the formal proof relies on the theory of 2BSDEs. We will note that the consumer's optimal efforts do not depend on the efforts of the others, which simplifies the task of obtaining the unique mean--field equilibrium given by Theorem \ref{thm:mfe}. In other words, each consumer optimises his deviation consumption independently of everyone else. Therefore, there is a unique mean--field equilibrium, given by the following theorem.

% \begin{align}\label{optimal_effort}
%     \alpha^{k,\star} \big( Z \big) &= \rho^k Z^{-}, \;
%     \text{and } \; \beta^{k,\star} \big( \Gamma \big) = 1 \wedge \left(  \lambda^{k} \Gamma^{-} \right)^{\frac{-1}{\eta^{k} + 1}}.
% \end{align}

\begin{theorem}\label{thm:mfe}
    Given a contract $\xi_T^{\xi_0,\zeta}\in\Xi_{\rm S}$ indexed by the triple of parameters $\zeta := (Z, Z^{\mu}, \Gamma)\in \mathcal V$, {\color{black}in the sense of \textnormal{Definition \ref{def:simple_contracts}}}, there exists a unique mean--field equilibrium in the sense of \textnormal{Definition \ref{mfe}} denoted by $(\mathbb P^{\star}, \mu^{\star})$ where 
    \begin{enumerate}[label=$(\roman*)$] 
      \item the optimal drift effort of the consumer is given by the process $\alpha^{\star}:=a^{\star}(Z)$ where 
        \begin{align*}
            a^{k,\star} (z) := \rho^k (z^{-} \wedge A_{\textnormal{max}}),\; z\in\R,\; k=1,\dots,d;
        \end{align*}
        \item the optimal volatility effort of the consumer is given by the process $\beta^{\star}:=b^{\star}(\Gamma)$ where
        \begin{align*}
            b^{k,\star} (\gamma) :=  1 \wedge \big( \lambda^k \gamma^{-} \big)^{\frac{-1}{\eta^k + 1}} \vee B_{\textnormal{min}},\; \gamma\in\R,\; k=1,\dots,d;
        \end{align*}
        \item $\mathbb P^{\star}$ is the law of $X$ driven by optimal controls
        \begin{align}\label{eq:optimal_deviation}
            \drm X_t = - \overline{\rho} \big( Z_t^{-}\wedge A_{\textnormal{max}} \big)  \mathrm{d} t + \sigma^{\star} (\Gamma_t) \cdot \mathrm{d}W_t + \sigma^{\circ} \mathrm{d}W^{\circ}_t;
        \end{align}
        \item $\mu^{\star}$ is the conditional law of $X$ given $\F^{\circ}$. 
    \end{enumerate}
\end{theorem}

{\color{black}The proof of the previous theorem, which is summarised below, relies on the theory of 2BSDEs, mostly postponed to Appendix \ref{sec:technical_proofs}: the crux of the argument here is to use the general result of Proposition \ref{prop:genbestreac} and Theorem \ref{th:genbestreac}, and show that we can construct directly a solution to 2BSDE \eqref{eq:2bsde2mf} whenever $\xi^{\xi_0,\zeta}\in\Xi_{\mathrm S}$. Indeed, contrary to the intuition presented in the previous section, the chain rule with common noise cannot be applied in our non--Markovian framework.}

\begin{proof}
{\color{black}We first assume that other consumers are playing according $\P^\star$, meaning in particular that their efforts are $\nu^\star = (\alpha^\star, \beta^\star)$ and that $\widehat \mu=\mu^\star$. In this case, one can notice that the simplification \eqref{eq:xi_simplify} of the contract holds, since other consumers are playing $\alpha^\star$.} Define then 
\[
Y_t:=-\mathrm{e}^{-R_A\xi_t^{\xi_0,\zeta}},\;  Z^1_t:=-R_AY_tZ_t,\;   Z^2_t:=-R_AY_t \widehat \E^{\P^\star_t}\big[Z^\mu_t(\widehat X_{t\wedge\cdot})\big],\;  \underline \Gamma_t:= -R_AY_t\Gamma_t,\; t\in[0,T],
\]
\[
 K_t:=\int_0^t\bigg(-R_AY_s \Hc^\circ(X_s,\zeta_s)-\frac12\underline\Gamma_s \Big( S_s + \big( \sigma^\circ \big)^2 \Big) - F(X_s, Y_s,  Z^1_s, S_s ) \bigg) \mathrm{d}s,\; t\in[0,T],
\]

where for any $S\geq 0$, $\Sigma^{-1}(S)$ is the pre--image of the singleton $\{S\}$ by the map $\Sigma:B\longrightarrow \R_+$, where we denote by $\Sigma(B)$ the image of $B$ by $\Sigma$, where the map $F:\R\times\R\times \R\times\Sigma(B)\longrightarrow \R$ is defined by
\begin{equation}\label{def:F}
F(x,y,z,S):=\sup_{(a,b)\in A\times\Sigma^{-1}(S)}\big\{-a\cdot\mathbf{1}_dz+R_Ay(c(a,b)-f(x))\big\},\; (x,y,z,S)\in \R\times\R\times \R\times\Sigma(B),
\end{equation}
and $\Hc^\circ$ by Equation \eqref{eq:Hc_circ}. A simple application of the usual It\=o's formula then leads to
\begin{align*}
    Y_t = &
    % - \erm^{-R_A \xi_T} 
    % + \int_t^T R_A Y_s \Big( - \Hc^\circ (X_s, \zeta_s)
    % + \dfrac{1}{2} \big( \Gamma_s + R_A Z_s^2 \big) \Big( \Sigma(\beta) + \big(\sigma^\circ \big)^2 \Big)  
    % + \dfrac{1}{2} R_A  \big( \sigma^{\circ} \big)^2 \overline Z_s^{\mu} \big(\overline Z_s^{\mu}+ 2 Z_s\big)  \Big) \drm s \\
    % &+ \int_t^T R_A Y_s Z_s \drm X_s 
    % + \sigma^{\circ}  \int_t^T R_A Y_s  \overline{Z}_s^{\mu} \drm W^{\circ}_s
    % - \dfrac12 \int_t^T R_A^2 Y_s \Big( Z^2_s \Sigma(\beta_s) + \big( Z_s + \overline{Z}_s^{\mu} \big)^2 \big( \sigma^\circ \big)^2 \Big) \drm s \\
    % = &- \erm^{-R_A \xi_T} 
    % + \int_t^T \Big( - R_A Y_s \Hc^\circ (X_s, \zeta_s)
    % - \dfrac{1}{2} \widetilde \Gamma_s \Big( \Sigma(\beta) + \big(\sigma^\circ \big)^2 \Big)
    % %+ \dfrac{1}{2} R_A^2 Y_s Z_s^2 \Sigma(\beta) 
    % % + \dfrac{1}{2} R_A^2 Y_s Z_s^2 \big(\sigma^\circ \big)^2  
    % % + \dfrac{1}{2} R_A^2 Y_s \big( \sigma^{\circ} \big)^2 \big(\overline Z_s^{\mu} \big)^2
    % % + R_A^2 Y_s \big( \sigma^{\circ} \big)^2 \overline Z_s^{\mu} Z_s 
    % \Big) \drm s
    % - \int_t^T \widetilde Z_s \drm X_s 
    % - \sigma^{\circ}  \int_t^T  \widetilde{Z}_s^{\mu} \drm W^{\circ}_s
    % %- \dfrac12 \int_t^T R_A^2 Y_s Z^2_s \Sigma(\beta_s) \drm s
    % % - \dfrac12 \int_t^T R_A^2 Y_s  \big( Z_s \big)^2 \big( \sigma^\circ \big)^2 
    % % - \dfrac12 \int_t^T R_A^2 Y_s  \big(\overline{Z}_s^{\mu} \big)^2 \big( \sigma^\circ \big)^2 
    % % - \int_t^T R_A^2 Y_s Z_s \overline{Z}_s^{\mu} \big( \sigma^\circ \big)^2 \drm s
    % \\
    % = &
    - \erm^{-R_A \xi_T} 
    + \int_t^T F\big(X_s,Y_s, Z^1_s,S_s\big) \drm s
    - \int_t^T  Z^1_s \drm X_s 
    - \sigma^{\circ}  \int_t^T  Z_s^{2} \drm W^{\circ}_s
    + \int_t^T \drm  K_s, \; t \in [0,T].
\end{align*}
% \[
% Y_t=-\mathrm{e}^{-R_A\xi_T^{\xi_0,\zeta}}+ \int_t^T F\big(X_s,Y_s,\widetilde Z_s,S_s\big)\mathrm{d}s
%  -\int_t^T\widetilde Z_s\mathrm{d}X_s-\sigma^\circ\int_t^T\widetilde Z^\mu_s\mathrm{d}W^\circ_s+\int_t^T\mathrm{d} \widetilde K_s,\; t\in[0,T].
% \]

By definition of $\Hc^\circ$ and $F$, we also directly check that $K$ is always a non--decreasing process, which vanishes on the support of any probability measure corresponding to the efforts $\beta^\star$ defined in the statement of the proposition. Indeed, 
\begin{align*}
    \drm K_s = 
    \dfrac{1}{2} R_A Y_s \Big(  \inf_{b \in B} \big\{c_{\beta} (b) - \underline\Gamma_s \Sigma(b) \big\} 
    - \inf_{b \in \Sigma^{-1}(S_s)} c_{\beta} (b) 
    + \underline\Gamma_s S_s  \Big) \drm s.
\end{align*}
To ensure that $(Y,( Z^1, Z^2)^\top, K)$ solves 2BSDE \eqref{eq:2bsde2mf}, it therefore remains to check that all the integrability requirements in Definition \ref{def:2BSDE} are satisfied, since the fixed--point constraint is satisfied by definition. The one for $Y$ is immediate by definition of the set $\Vc$. The required integrability on $(( Z^1, Z^2)^\top, K)$ then follows from \citet*[Theorem 2.1 and Proposition 2.1]{bouchard2018unified}.

\medskip
We have therefore obtained that the candidate provided in the statement of the proposition was indeed an equilibrium. Let us now prove uniqueness. Let $\widehat \nu = (\widehat \alpha, \widehat \beta)$ be the arbitrary effort of other consumers, and the associated conditional distribution $\widehat \mu$. In this case, a contract $\xi^{\xi_0,\zeta}\in\Xi_{\mathrm S}$ no longer admits the decomposition \eqref{eq:xi_simplify}. Nevertheless, $\Xi_{\mathrm S} \subset \Xi$ and by Proposition \ref{prop:genbestreac}, the optimal effort $\nu^\star$ is the maximiser of the map $F$, and does not depend on $\widehat \nu$. Although there is no uniqueness in general of the probability $\P^\star$, and therefore of the effort $\beta^\star$, the effort $\alpha^\star$ is unique and is the one defined in the statement of the theorem. To sum up, given a contract in $\Xi_{\mathrm S}$ and for arbitrary efforts $\widehat \nu$ of others, each consumer has a unique optimal drift effort $\alpha^\star$, independent of $\widehat \nu$. We can therefore already conclude that the optimal effort $\alpha^\star$ is the same for all consumers. Using the dynamic of $\widehat X$ with $\alpha^\star$, the contract $\xi^{\xi_0,\zeta}$ admits the decomposition \eqref{eq:xi_simplify}. Therefore, we can apply the reasoning above to construct a solution to the 2BSDE \eqref{eq:2bsde2mf}, and we have now uniqueness of the probability $\P^\star$, and therefore on the volatility effort $\beta^\star$, given by point $(ii)$ of the theorem. To conclude, given an arbitrary law $\widehat \mu$ and a contract $\xi^{\xi_0,\zeta}\in\Xi_{\mathrm S}$, the optimal effort is $\nu^\star = (\alpha^\star, \beta^\star)$, inducing the law $\P^\star$ and the conditional law $\mu^\star$. It is therefore the unique equilibrium.
\end{proof}

Throughout this work, we will denote by $v^\star(z,\gamma) := ( a^\star(z), b^\star(\gamma) )$ the optimal response of a given agent. The cost associated to this effort will be denoted by $c^\star(z,\gamma)$.

\subsubsection{Interpretation of the optimal efforts}\label{ss:interpretation_effort}

For a tuple of payment rates $\zeta \in \R\times\Lc\times\R$,

\begin{enumerate}[label=$(\roman*)$] 
    \item the $z$ component of the payment induces an effort of the consumer on all usages to reduce his consumption in average. The effort on the $k$--th usage is proportional to its cost, $1/\rho^k$, non increasing with $z$, positive if $z$ is negative and zero otherwise. Hence, the more $z$ is negative, the more the Agent will reduce his consumption deviation in average;
    \item the component $\gamma$ induces an effort only on the usages whose cost $1/\lambda^k$ is lower than the payment. If $\gamma$ is non negative, $b^{k,\star} (\gamma) = 1$ for all usages, hence the consumer makes no effort on the volatility of his consumption deviation. The more $\gamma$ is negative, the more $b$ will be close to zero, i.e. the more the Agent will reduce the volatility of his consumption deviation;
    \item the $z^{\mu}$ component has no influence on the consumer's efforts: although his payment is indexed on the deviation consumption of others, the consumer will not take it into account to optimise his deviation consumption. This model can be criticised but seems rather logical in the sense that a consumer optimises his consumption independently of what his neighbours do, even if the price of electricity depends on the global demand. %However, one might want to build contracts that incentives consumers to align their consumption with the consumption of their neighbours, but it is not the case we study here;
\end{enumerate}

Hence, the efforts of the consumer are the same as those defined in \cite{aid2018optimal}, although the contracts we consider have more components. The additional components do not affect the optimal effort of the consumer.

\begin{remark}\label{rk:drift_vol_ctrl}
    These results are consistent with classical results on drift and volatility control $($see {\rm \cite{cvitanic2017moral}} or {\rm\cite{cvitanic2018dynamic}}$)$, contracts indexed by $(z, \gamma)$ are sufficient to incentivise the Agent for making effort on drift and volatility. Therefore it is quite natural that another parameter in the contract will not directly affect the effort, but may increase the value function of the Principal.
\end{remark}

To sum up this section, we provided a new form of contracts in Definition \ref{def:simple_contracts}, called simple contracts, allowing us, by Theorem \ref{thm:mfe}, to compute the optimal efforts of the consumers and the associated unique mean--field equilibrium. The aim of the following section is to prove that there is no loss of generality to consider only simple contracts, and to solve the Principal's problem thus restricted to these simple contracts.

\section{Principal's problem}\label{sec:principal_problem}

We recall that the optimisation problem of the Principal has been defined as follows
\[
V_0^P := \sup_{\xi \in \Xi} \sup_{(\P, \mu) \in \Mc^{\star} (\xi)}  \mathbb{E}^{\P} \bigg[ U^P \bigg( - \mathbb{E}^{\P} \bigg[\xi+\int_0^Tg(X_s)\mathrm{d}s+\frac \theta 2\int_0^T\mathrm{d}\langle X\rangle_s\bigg| \Fc^\circ_T\bigg] \bigg)  \bigg].
\]
Following the general approach of \cite{cvitanic2018dynamic}, we expect that there is no loss of generality for the Principal to restrict to contracts in $\Xi_{\rm S}$, {\color{black} in the sense of Definition \ref{def:simple_contracts}}, instead of $\Xi$. This property had been obtained before in \citet*{elie2019tale} for general moral hazard problems with a continuum of Agents with mean--field interaction, but who were constrained to simply control the drift of the diffusion $X$. We show here that this general result also extends to cases where volatility can be controlled as well, using 2BSDEs theory. For notational simplicity, we define for any $(\xi_0,\zeta)\in\R\times\Vc$ the following process
\[
L_t^{\xi_0,\zeta} := \xi_t^{\xi_0,\zeta} + \int_0^t g ( X_s ) \mathrm{d} s + \dfrac{\theta}{2} \int_0^t \drm \langle X \rangle_s,  \text{ for }  t \in [0,T].
\]

\begin{theorem}\label{thm:main}
The following equality holds
\begin{align*}
V_0^P &= \sup_{(\xi_0,\zeta)\in [U_A^{-1} (R_0),+\infty)\times\Vc} 
\mathbb{E}^{\P^\star} \Big[ U^P \Big( - \mathbb{E}^{\P^\star} \Big[ L^{\xi_0,\zeta}_T \Big| \Fc^\circ_T \Big] \Big)  \Big]=\sup_{\zeta\in\Vc} 
\mathbb{E}^{\P^\star} \Big[ U^P \Big( - \mathbb{E}^{\P^\star} \Big[L^{U_A^{-1} (R_0),\zeta}_T \Big| \Fc^\circ_T\Big] \Big)  \Big].
\end{align*}
\end{theorem}

The proof of this main theorem is {\color{black} postponed} to Appendix \ref{proof:main}. 

\medskip

From now on, it is assumed that $f$ and $g$ have linear growths. Thus, the result of Proposition \ref{prop:RU} holds. To lighten the notations, the exponent $U_A^{-1} (R_0)$ will often be omitted, since, using Proposition \ref{prop:RU}, it is fixed once and for all by $U_A^{-1} (R_0) = \psi(0,X_0)$. Notice that from the Principal's point of view, when considering contracts in $\Xi_{\mathrm S}$, and when the consumers are at the unique equilibrium, the deviation of other consumers $\widehat X$ is nothing but a copy of $X$, in the sense of Definition \ref{def:copy_process}, which we denoted by $\widetilde X$. {\color{black}Besides, noticing then that the distributions $\widehat \P$ and  $\P^\star$ coincide, in particular that $\widehat \mu = \mu^\star$, and imposing $\xi_0:= U_A^{-1} (R_0) $ from now on, one obtain:}
\begin{align*}
    \xi_t^{\zeta} = &\ \xi_0- \int_0^t \Hc (X_s, \mu^\star_s, \zeta_s, \alpha_s^{\star})  \drm s 
    + \int_0^t Z_s \drm X_s 
    + \int_0^t \widetilde{\E}^{\P^\star_s} \big[ Z_s^{\mu} (\widetilde X_s) \drm \widetilde X_s \big]
    + \dfrac{1}{2} \int_0^t \big( \Gamma_s + R_A Z_s^2 \big) \drm \langle X \rangle_s
    \nonumber \\
    &+ \dfrac{1}{2} R_A \big( \sigma^{\circ} \big)^2 \int_0^t  \Big( \widetilde{\E}^{\P^\star_s} \big[ Z_s^{\mu} (\widetilde X_s)  \big] \Big)^2 \drm s
    + R_A \big( \sigma^{\circ} \big)^2 \int_0^t Z_s \widetilde{\E}^{\P^\star_s}  \big[Z_s^{\mu} (\widetilde X_s) \big] \drm s.
\end{align*}
{\color{black} Given the form of the Principal's control problem and of the contract, the intuition is that her value function should depend only on time and on the conditional law $\mu^Y$ of the state variable $Y = (X,L)^\top$.}

% \medskip

% \textcolor{red}{Cas d'un contrat $\zeta = (0,0,0)$ ? } (Remarque 5.1 p11 de \cite{aid2018optimal}). 

\subsection{The general case}

{\color{black} The attentive reader have noticed that the right--hand side of Theorem \ref{thm:main} looks like the value function of a stochastic control problem of a McKean--Vlasov stochastic differential equation with common noise. However, one of the two state variable, mainly $L$, seems to be considered in the strong formulation (it is indexed by the control $\zeta$), while the other state variable $X$ is considered in weak formulation (the control $\zeta$ only impacts the distribution of $X$ through $\P^\star$). As highlighted by \citet*[Remark 5.1.3]{cvitanic2012contract}, it makes little sense to consider a control problem of this form directly. Therefore, from our point of view, there is no reason why we should adopt anything but the weak formulation to state the problem of the Principal, contrary to what is usually done in Principal--Agent problems (see, e.g., \cite{cvitanic2018dynamic}), since this is the one which makes sense for the problem of the Agent\footnote{{\color{black}Notice that at the end of the day, this is not really an issue. Indeed, provided that the problem has enough regularity (typically some semi--continuity of the terminal and running reward with respect to state), one can expect the strong and weak formulations to coincide. This is proved in a setting with common noise by \citet*{djete2019mckean,djete2019general}.}}. We will thus formulate it below.}

\medskip

Let $\V$ be the collection of all finite and positive Borel measures on $[0,T]\times \R^3$, whose projection on $[0,T]$ is the Lebesgue measure, and we call $\V_0$ the set of all $q \in \V$ of the form $\delta_{\phi_s}(\mathrm{d} \overline v )\mathrm{d}t$ for some Borel function $\phi$. The intuition is that the Principal's problem depends only on time and on the conditional law $\mu^Y$ of the state variable $Y = (X,L)^\top$. Following the same methodology used for the Agent's problem, to properly define the weak formulation of the Principal's problem, we are led to consider the following canonical space
\[
    \Omega^P :=  \Omega^{\circ} \times \Omega^{1,P} \times \Omega^{2,P} \times \V, \; \text{where}\;
    \Omega^{P,1} := \Cc ([0,T],\R^2 \times \R^d), \;
    \text{and } \; \Omega^{P,2} := \Pc_T (\R^2),
\]
with canonical process $(W^{\circ}, Y, W, \mu^Y, \Lambda^P)$, where for any $(t, w^{\circ}, y, w, u, q) \in [0,T] \times \Omega^P$
\[
    W^{\circ}_t(w^{\circ},y,w,u,q):=w^{\circ}(t), \; Y_t(w^{\circ},y,w,u,q) :=y(t),\; W_t(w^{\circ},y,w,u,q):=w(t),
\]
    \vspace{-1.2em}
\[  
    \mu^Y_t(w^{\circ},y,w,u,q):=u(t),\; \Lambda^P_t(w^{\circ},y,w,u,q):=q.
\]

Less formally, for all $t\in[0,T]$, $\mu^Y_t \in \Pc(\R^2)$ will be the conditional distribution of $Y_t = (X_t,L_t)$, and we will denote by $\mu^X$ and $\mu^L$ the marginal distributions of $\mu^Y$. When no confusion is possible, in order to lighten the notations, we will often omit the space for the integrals with respect to the conditional distribution, by denoting for example:
\begin{align*}
    \int \phi (x, \ell) \mu^Y (\drm x, \drm \ell) := \int_{\R^2} \phi (x, \ell) \mu^Y (\drm x, \drm \ell), \; \text{ for any } \; \phi : \R^2 \longrightarrow \R.
\end{align*}

The canonical filtration $\F^P:=(\Fc^P_t)_{t\in[0,T]}$ is defined as
\begin{align*}
    \mathcal F^P_t:=\sigma\Big( \big(W^{\circ}_s, Y_s, W_s, \mu^Y_s, \Delta_s(\varphi) \big): (s,\varphi)\in[0,t]\times \Cc_b \big( [0,T]\times \R^3, \R \big) \Big),\; t\in [0,T],
\end{align*}
where $\Cc_b([0,T]\times \R^3,\R)$ is the set of all bounded continuous functions from $[0,T]\times \R^3$ to $\R$, and for any $(s,\varphi)\in[0,T]\times \Cc_b([0,T]\times \R^3,\R)$, $\Delta_s(\varphi):=\int_0^s\int_{\R^3} \varphi(r, \overline v ) \Lambda^P(\mathrm{d}r,\mathrm{d} \overline v).$ We also define $\F^{P,\circ}:=(\Fc^{P,\circ}_t)_{t\in[0,T]}$, a smaller filtration containing only the information generated by the common noise and the conditional law of $Y$, 
    $\Fc^{P,\circ}_t:=\sigma\big((W^{\circ}_s,\mu^Y_s):s\in[0,t]\big),\; t\in[0,T].$
% Using the notation $\overline Z_s^\mu := \widetilde{\E}^{\mu^X_s} \big[ Z_s^{\mu} (\widetilde X_s)  \big]$ for $s \in [0,T]$, by Definition \ref{def:simple_contracts} the process $\overline{\zeta} := \big( Z, \overline Z^{\mu}, \Gamma \big) \in \overline \Vc$, and the dynamic of the two-dimensional process $Y$ is given by:
% \begin{align*}
%     \drm Y_t = 
%     &\ \begin{pmatrix}
%         - \overline{\rho} \big( Z_t^{-} \wedge A_{\textnormal{max}} \big) \\
%         b \big( X_t, \overline{\zeta}_t \big)
%     \end{pmatrix} \drm t
%     + 
%     \begin{pmatrix}
%         \sigma^{\star} ( \Gamma_t )^\top \\
%         Z_t \sigma^{\star} ( \Gamma_t)^\top
%     \end{pmatrix} \mathrm{d} W_t
%     + 
%     \begin{pmatrix}
%         \sigma^{\circ} \\
%         \sigma^{\circ} \big( Z_t + \overline Z_t^{\mu} \big)
%     \end{pmatrix}  \mathrm{d} W_t^{\circ} \\
%     \text{where } b (X_t, \overline{\zeta}_t) := &\ 
% 	c^\star (Z_t, \Gamma_t)
%     + g(X_t) - f ( X_t ) 
%     + \dfrac{1}{2} R_A  Z_t^2 \Sigma^{\star} ( \Gamma_t )
%     + \dfrac{1}{2} R_A \big(\sigma^{\circ}\big)^2 \big( Z_t + \overline Z_t^{\mu} \big)^2
%     + \dfrac{\theta}{2} \big( \Sigma^{\star} ( \Gamma_t ) + \big(\sigma^{\circ}\big)^2 \big).
% \end{align*}
Let $\Cc^2_b(\mathbb R^2 \times \R^d \times\R,\R)$ be the set of bounded twice continuously differentiable functions from $\mathbb R^2 \times \R^d \times\R$ to $\R$, whose first and second derivatives are also bounded, and for any $(s,\varphi)\in[0,T]\times \Cc^2_b(\mathbb R^2 \times \R^d \times\R,\R)$, we set
\begin{align*}
    M^P_s(\varphi):=&\ \varphi(Y_s,W_s,W^{\circ}_s) - \int_0^s \int_U\bigg( A_P(\overline v) \cdot \nabla \varphi(Y_r,W_r,W^{\circ}_r) + \frac12  {\rm Tr} \big[ D^2 \varphi(Y_r,W_r,W^{\circ}_r) B_P(\overline v)  B_P^\top (\overline v) \big] \bigg)\Lambda^P(\mathrm{d}r, \mathrm{d} \overline v),
\end{align*}
where $D^2 \varphi$ denotes the Hessian matrix of $\varphi$, $A_P$ and $B_P$ are respectively the drift vector and the diffusion matrix of the vector process $(Y, W, W^{\circ})^\top$
\begin{align}\label{eq:drift_vol}
    A_P(\overline v) := 
    \begin{pmatrix}
        - \overline{\rho} \big( z^{-}\wedge A_{\textnormal{max}} \big) \\
        b \big( X_t, \overline v \big) \\
        \mathbf{0}_d \\
        0
    \end{pmatrix},
    \; 
    B_P(\overline v) :=
    \begin{pmatrix}
        0 & 0 & \sigma^\star(\gamma)^\top & \sigma^{\circ} \\
        0 & 0 & z \sigma^\star(\gamma)^\top & (z + \overline z^\mu )\sigma^{\circ} \\
        \mathbf{0}_d & \mathbf{0}_d & \mathrm{I}_d & \mathbf{0}_d \\
        0 & 0 & \mathbf{0}_d^\top & 1
    \end{pmatrix},
\end{align}
where for $\overline v:=(z,\overline z^\mu, \gamma)\in \R^3$
\begin{align*}
    b (x, \overline v) := &\ 
	c^\star (z, \gamma)
    + g(x) - f (x) 
    + \dfrac{1}{2} R_A z^2 \Sigma^{\star} ( \gamma )
    + \dfrac{1}{2} R_A \big(\sigma^{\circ}\big)^2 \big( z + \overline z^{\mu} \big)^2
    + \dfrac{\theta}{2} \big( \Sigma^{\star} ( \gamma ) + \big(\sigma^{\circ}\big)^2 \big).
\end{align*}

We fix some initial condition, namely a probability measure $\varrho^P$ on $\R^2$ representing the law at $0$ of $Y$. The marginals of $\varrho^P$ are given by $\varrho$, the law of $X_0$, and $\psi(0,\cdot) \circ \varrho$, the law of $U_A^{-1} (R_0) = \psi(0,X_0)$. 
{\color{black}
\begin{definition}\label{def:Qc_principal}
Let $\M^P$ be the set of all probability measures on $(\Omega^P,\Fc^P_T)$. The subset $\Qc \subset \M^P$ is composed of all $\P$ such that 
\begin{enumerate}[label=$(\roman*)$]
    \item $M^P(\varphi)$ is a $(\P,\mathbb F^\P)$--local martingale on $[0,T]$ for all $\varphi \in \Cc^2_b(\R^2 \times \R^d \times \R,\R);$
    \item $\P \circ (Y_0)^{-1} = \varrho^P$, and $\P \circ \big((W_0,W_0^{\circ})\big)^{-1} = \iota;$
    \item $\P \big[\Lambda^P \in \V_0]=1;$
    \item for $\P$--a.e. $\omega \in \Omega^P$ and for every $t \in[0,T]$, we have $\mu^Y_t(\omega)=\P^{\omega}_t \circ (Y_{t})^{-1},$
    where $(\P_t^{\omega})_{\omega\in\Omega}$ is a family of {\rm r.c.p.d.} for $\P$ given $\Fc_t^{P,\circ}$. Similarly as before, we will denote by $\E^{\P_t}$ the conditional expectation under the {\rm r.c.p.d.} $\P_t^\omega;$
    \item $(W^{\circ},\mu^Y)$ is $\P$--independent of $W$.
\end{enumerate}
\end{definition}

\begin{remark}
    One may notice that the previous definition, in particular the point $(iv)$, does not involve a probability measure on the path space anymore, contrary to \textnormal{Definition \ref{def:Pc}}, and as noticed in \textnormal{Remark \ref{rk:measure_on_pathspace}}. Indeed, the form of contracts we are considering makes the Principal's problem Markovian in that sense. 
\end{remark}}

Following the reasoning developed in Subsection \ref{ss:theoretical_formulation}, we can construct a copy of the canonical space $\Omega^P$ and a copy of $Y$ in the sense of Definitions \ref{def:copy_space} and \ref{def:copy_process}. Thanks to the previous formulation, we can write the weak formulation of the Principal's problem as follows
\begin{align*}
V_0^P &= \sup_{\P \in \Qc} 
%\sup_{(\P, \mu) \in \Mc^{\star} (\xi_T^{U_A(R_0),\zeta})} 
\mathbb{E}^{\P} \Big[ U^P \Big( - \mathbb{E}^{\P_T} \Big[L^{\overline \zeta^\P}_T \Big] \Big)  \Big],
\end{align*}
{\color{black}where, given some $\P \in \Qc$, the notation $\E^{\P_t}$ will refer to the conditional expectation under the r.c.p.d. $\P_t$ of some $\P \in \Qc$ given $\Fc_t^{P,\circ}$ for all $t \in [0,T]$, in the sense of Definition \ref{def:Qc_principal} $(iv)$.}

% Using the notation $\overline Z_s^\mu := \widetilde{\E}^{\mu^X_s} \big[ Z_s^{\mu} (\widetilde X_s)  \big]$ for $s \in [0,T]$, by Definition \ref{def:simple_contracts} the process $\overline{\zeta} := \big( Z, \overline Z^{\mu}, \Gamma \big) \in \overline \Vc$, and the dynamic of the two-dimensional process $Y$ is given by:
% \begin{align*}
%     \drm Y_t = 
%     &\ \begin{pmatrix}
%         - \overline{\rho} \big( Z_t^{-} \wedge A_{\textnormal{max}} \big) \\
%         b \big( X_t, \overline{\zeta}_t \big)
%     \end{pmatrix} \drm t
%     + 
%     \begin{pmatrix}
%         \sigma^{\star} ( \Gamma_t )^\top \\
%         Z_t \sigma^{\star} ( \Gamma_t)^\top
%     \end{pmatrix} \mathrm{d} W_t
%     + 
%     \begin{pmatrix}
%         \sigma^{\circ} \\
%         \sigma^{\circ} \big( Z_t + \overline Z_t^{\mu} \big)
%     \end{pmatrix}  \mathrm{d} W_t^{\circ} \\
%     \text{where } b (X_t, \overline{\zeta}_t) := &\ 
% 	c^\star (Z_t, \Gamma_t)
%     + g(X_t) - f ( X_t ) 
%     + \dfrac{1}{2} R_A  Z_t^2 \Sigma^{\star} ( \Gamma_t )
%     + \dfrac{1}{2} R_A \big(\sigma^{\circ}\big)^2 \big( Z_t + \overline Z_t^{\mu} \big)^2
%     + \dfrac{\theta}{2} \big( \Sigma^{\star} ( \Gamma_t ) + \big(\sigma^{\circ}\big)^2 \big).
% \end{align*}

\begin{remark}
    First, recall that simple contracts were defined as random variables of the form $( \xi_T^{\xi_0, \zeta})$, for $(\xi_0, \zeta) \in \R \times \Vc$. By \textnormal{Equation \eqref{eq:drift_vol}}, we notice that the drift vector and the diffusion matrix of the process $(Y,W,W^\circ)^\top$ are defined as function of $\overline v = (z,\overline z^\mu, \gamma)$. This is why we consider that the Principal controls, through a probability $\P \in \Qc$, the triple of controls $\overline \zeta^\P = (Z, \overline Z^\mu, \Gamma) \in \overline \Vc$, where for all $t \in [0,T]$, $\overline Z_t^\mu = \widetilde \E^{\P_t} [Z_t^\mu (\widetilde X_t) ]$, instead of $\zeta^\P \in \Vc$.
    Moreover, the conditional laws of $X$ and $L$ do not impact their dynamics, therefore the Principal's problem does not seem to be a McKean--Vlasov control problem, but only a standard control problem. Nevertheless, the Principal's criterion reveals the conditional law of $L$, which transforms the problem into a McKean--Vlasov one.
\end{remark}

%We denote by $v^P$ the dynamic value function of the Principal, which is a function of time and of the conditional law $\mu_t$ of $Y_{t \wedge \cdot}$. In order to apply a Chain Rule with common noise, we will need some regularity on $v^P$.
%\todo[inline]{D: no need to define the dynamic version here, the only thing we want is the value at $0$. Moreover, you need to re-define the set of measures for the problem of the Principal, with the martingale problems for both $L$ and $X$ similarly as in the problem of the Agent, and write explicitly the the value function at $0$. This is important because now the notation $L^{U_A(R_0),\zeta}$ will not be used, since $L$, as $X$, is a canonical process.}

In order to apply the chain rule with common noise to functions depending on time and conditional distribution, we define the regularity assumption needed.

\begin{definition}[$\Cc^{1,2}$--regularity]\label{def:C12}
    A function $u : [0,T] \times \Pc (\R^d) \longrightarrow \R, (t, \mu) \longmapsto u(t, \mu)$ is smooth enough in the sense of Chain Rule under $\Cc^{1,2}$--regularity if
    \begin{enumerate}[label=$(\roman*)$]
        \item $u$ is differentiable with respect to $t$, and the partial derivative $\partial_t u : [0,T] \times \Pc (\R^d) \longrightarrow \R$ is continuous;
        \item for all $t \in [0,T]$, the mapping $\mu \in \Pc (\R^d) \longmapsto u(t, \mu)$ is simply $\Cc^2$, in the sense defined in {\rm\cite[Section 4.3.2]{carmona2018probabilisticII}}, and satisfies the assumptions of {\rm\cite[Theorem 4.14]{carmona2018probabilisticII}}.
    \end{enumerate}
\end{definition}

The previous definition allows us to consider the natural extension for time dependent functions of the chain rule under $\Cc^2$--regularity defined in \cite[Theorem 4.14]{carmona2018probabilisticII}. Therefore, for any function $v : [0,T] \times \Pc (\R^2) \longrightarrow \R$ smooth enough in the sense of Definition \ref{def:C12}, the chain rule under $\Cc^{1,2}$--regularity is written as
\begin{align*}
    \drm v ( t, \mu^Y_t ) = &\ \partial_t v ( t, \mu^Y_t ) \drm t 
    + \mathbb{E}^{\P_t} \big[ \partial_{\mu} v ( t, \mu^Y_t ) (Y_t) \cdot \drm Y_t \big] 
    + \dfrac{1}{2} \mathbb{E}^{\P_t} \widetilde{\mathbb{E}}^{\P_t} \Big[ \mathrm{Tr} \big[ \partial_{\mu}^2 v ( t, \mu^Y_t ) \big( Y_t, \widetilde{Y}_t \big) \drm \langle Y, \widetilde{Y} \rangle_t \big] \Big] \\
    &+ \dfrac{1}{2} \mathbb{E}^{\P_t} \Big[ \mathrm{Tr} \big[ \partial_{y} \partial_{\mu}  v ( t, \mu^Y_t ) (Y_t) \drm \langle Y \rangle_t \big] \Big],
\end{align*}
where $\widetilde{Y}$ is a copy of $Y$ in the sense of Definition \ref{def:copy_process} {\color{black}and $\E^{\P_t}$ must be understood according to Definition \ref{def:Qc_principal} $(iv)$}.

\begin{remark}
Notice that another way to obtain this chain rule is to simplify the chain rule in {\rm\cite[Theorem 4.17]{carmona2018probabilisticII}} by considering a function which does not depend on the state process. 
% Therefore, abusing notations slightly, we will say that a function $u$ defined on $[0,T] \times \Pc (\R^d)$ is smooth enough to apply the Chain Rule under $\Cc^{1,2}-$Regularity if its natural extension to $[0,T] \times \R^d \times \Pc (\R^d)$ is smooth enough in the sense of \cite[Section 4.3.4]{carmona2018probabilisticII}.
\end{remark}

{\color{black} We are thus led to consider the following HJB equation, written on the space of measures:}
\begin{align}\label{eq:pde_vP}
    0 = &\ \partial_t v ( t, \mu^Y_t )
    + \dfrac{1}{2} \big( \sigma^{\circ} \big)^2 \bigg( 
    \iint \partial_{\mu^X}^2  v ( t, \mu^Y_t ) \big(y, \widetilde y \big) \mu^Y_t (\drm y) \mu^Y_t (\drm \widetilde y)
    + \int \Big(\theta\partial_{\mu^L} v ( t, \mu^Y_t ) (y) + \partial_{x} \partial_{\mu^X} v ( t, \mu^Y_t ) (y) \Big) \mu^Y_t (\drm y) \bigg)
    \nonumber \\
    &+ \int \partial_{\mu^L} v ( t, \mu^Y_t ) (y) (g-f)(x) \mu^Y_t (\drm y) 
    + \dfrac{1}{2} \sup_{\overline v \in \R^3} h\big(\mu^Y_t, \partial_{\mu} v (t, \mu^Y_t), \partial_{y} \partial_{\mu} v (t, \mu^Y_t), \partial^2_{\mu} v (t, \mu^Y_t), \overline v\big),
\end{align}
% where 
% \begin{align*}
%     \partial_{\mu} v ( t, \mu ) (y) &= 
%     \begin{pmatrix}
%         \partial_{\mu^X} v (t, \mu) (y) \\
%         \partial_{\mu^L} v (t, \mu) (y)
%     \end{pmatrix}, \\
%     \partial_{\mu}^2  v ( t, \mu ) \big( y, \widetilde{y} \big) &= 
%     \begin{pmatrix}
%         \partial_{\mu^X}^2  v (t, \mu) \big( y, \widetilde{y} \big) & \partial_{\mu^X \mu^L}^2  v (t, \mu) \big( y, \widetilde{y} \big) \\
%         \partial_{\mu^X \mu^L}^2  v (t, \mu) \big( y, \widetilde{y} \big) & \partial_{\mu^{L}}^2  v (t, \mu) \big( y, \widetilde{y} \big)
%     \end{pmatrix}, \\
%     \text{and } \; \partial_{y} \partial_{\mu}  v ( t, \mu ) (y) &= 
%     \begin{pmatrix}
%         \partial_{x} \partial_{\mu^X}  v (t, \mu) ( y) & \partial_{\ell} \partial_{\mu^X}  v (t, \mu) ( y ) \\
%         \partial_{x} \partial_{\mu^L}  v (t, \mu) ( y ) & \partial_{\ell} \partial_{\mu^L}  v (t, \mu) ( y )
%     \end{pmatrix}.
% \end{align*}
% \todo[inline]{D: $\nabla _\mu$ and $\Delta_\mu$ are not defined. Better to use the usual notations for derivatives.}
with terminal condition $v ( T, \mu^Y_T) = U^P \big( - \mathbb{E}^{\mu^L_T} [L_T] \big)$ and where, for any $v_{\mu} \in \Lc^{1 \times 2}$, $v_{y,\mu} \in \Lc^{2 \times 2}$ and $v_{\mu, \mu} \in \big(\Lc^2\big)^{2 \times 2}$,
% \begin{align}\label{eq:def_h}
%    \theta(\mu, \Delta_{\mu} v (t, \mu), \nabla_{\mu} v(t, \mu), \overline v) := &\ 
%     - 2 \overline{\rho} \big( z^{-} \wedge A_{\textnormal{max}} \big) \int \partial_{\mu^X} v ( t, \mu ) (y) \mu (\drm y) \nonumber \\
%     &+ \bigg( 2 c^\star (z,\gamma)
%     + R_A  z^2 \Sigma^{\star} (\gamma)
%     + R_A \big(\sigma^{\circ}\big)^2 \big(z + \overline z^{\mu} \big)^2
%     +\theta\Sigma^{\star} ( \gamma ) \bigg) 
%     \int \partial_{\mu^L} v ( t, \mu ) (y) \mu (\drm y) \nonumber \\
%     &+ 2 \big( \sigma^{\circ} \big)^2 \big( z + \overline z^{\mu} \big) \int \int
%     \partial_{\mu^X \mu^L}^2  v ( t, \mu ) (y, \widetilde y ) \mu (\drm y) \mu (\drm \widetilde y) \nonumber \\
%     &+ \big( \sigma^{\circ} \big)^2 \big( z + \overline z^{\mu} \big)^2 \int \int \partial_{\mu^L}^2  v ( t, \mu ) \big(y, \widetilde y \big) \mu (\drm y) \mu (\drm \widetilde y) \nonumber \\
%     &+ \int \Big( 
%     \partial_{x} \partial_{\mu^X} v ( t, \mu ) (y) \Sigma^{\star} (\gamma) 
%     + \partial_{\ell} \partial_{\mu^L}  v ( t, \mu ) (y) \Big( z^2 \Sigma^{\star}(\gamma) + \big( \sigma^{\circ} \big)^2 ( z + \overline z^{\mu} )^2 \Big) \nonumber \\
%     &+ 2 \partial_{\ell} \partial_{\mu^X}  v ( t, \mu ) (y)  \Big( z \Sigma^{\star} (\gamma) + \big( \sigma^{\circ} \big)^2 (z + \overline z^{\mu}) \Big) \Big) \mu (\drm y).
% \end{align}
\begin{align}\label{eq:def_h}
    h(\mu, v_{\mu}, v_{y,\mu}, v_{\mu,\mu}, \overline v) :=
    %\theta(\mu, \Delta_{\mu} v (t, \mu), \nabla_{\mu} v(t, \mu), \overline v) := &
    &- 2 \overline{\rho} \big( z^{-} \wedge A_{\textnormal{max}} \big) \int v^1_{\mu} (y) \mu (\drm y)
    + 2 c^\star (z,\gamma) \int v^2_{\mu} (y) \mu (\drm y) \nonumber \\
    &+ \Sigma^{\star} (\gamma) \int \Big(
    \big(\theta+ R_A  z^2 \big) v^2_{\mu} (y)
    + v^{1,1}_{y,\mu}(y)
    + z^2 v^{2,2}_{y,\mu}(y)
    + 2 z v^{1,2}_{y,\mu}(y)
    \Big) (y) \mu (\drm y)
    \nonumber \\
    &+ \big(\sigma^{\circ}\big)^2 \big(z + \overline z^{\mu} \big)^2 \int \bigg(
    R_A v^2_{\mu}(y) + v^{2,2}_{y,\mu}(y)
    + \int v^{2,2}_{\mu,\mu} \big(y, \widetilde y \big) \mu (\drm \widetilde y) \bigg) \mu (\drm y) \nonumber \\
    &+ 2 \big( \sigma^{\circ} \big)^2 \big( z + \overline z^{\mu} \big) \int \bigg(
    v^{1,2}_{\mu,\mu} (y, \widetilde y ) \mu (\drm \widetilde y)
    + v^{1,2}_{y,\mu}(y) \bigg) \mu (\drm y),
\end{align}

\begin{theorem}\label{thm:Pb_Principal_general}
    If there is a solution $v$ to {\rm PDE} \eqref{eq:pde_vP}, smooth enough in the sense of \textnormal{Definition \ref{def:C12}}, with partial derivatives satisfying for each $\P\in\Qc$
    \begin{align}\label{eq:conditions_martingale_general}
        \E^\P \bigg[ \bigg(\int_0^T \bigg( \int \partial_{\mu^X} v(t, \mu^Y_t)(x, \ell) \mu^Y_t (\drm x, \drm \ell) \bigg)^2\mathrm{d}t\bigg)^{1/2}\bigg] + \E^\P \bigg[ \sup_{0\leq t\leq T}\bigg| \int\partial_{\mu^L} v ( t, \mu^Y_t ) (x, \ell)\mu^Y_t (\drm x, \drm \ell)\bigg|^{\frac{p^\prime}{p^\prime-1}} \bigg] < + \infty,
    \end{align}
    where $p^\prime>1$ is the exponent appearing in \textnormal{Lemma \ref{lemma:integZ}},
    and a function $\overline v^\star : [0,T] \times \Pc (\R^2) \longrightarrow \R^3$ satisfying
    % \todo[inline]{D: $\overline \zeta^\star$ is indeed a triplet of real numbers, but its more importantly a function of $t$ and $\mu$, and the actual optimal triple for the contract is not $\overline \zeta^\star$, but it is $\overline \zeta^\star$ evaluated along the distribution of $L$, so we should also change this (to be changed in al subsequent results as well). This is the same as in classical control: the HJB equation  gives you a function $a^\star(t,x)$, but the optimal control is $\alpha^\star_t:=a^\star(t,X_t)$.}
    \begin{align*}
        h(\mu^Y_t, \partial_{\mu} v (t, \mu^Y_t), \partial_{y} \partial_{\mu} v (t, \mu^Y_t), \partial^2_{\mu} v (t, \mu^Y_t), \overline v^\star (t, \mu^Y_t))
        = \sup_{\overline v \in \R^3} h(\mu^Y_t, \partial_{\mu} v (t, \mu^Y_t), \partial_{y} \partial_{\mu} v (t, \mu^Y_t), \partial^2_{\mu} v (t, \mu^Y_t), \overline v),
    \end{align*}
    then
    \begin{enumerate}[label=$(\roman*)$]
        \item $v(0,\mu^Y_0) = V_0^{P}$;
        \item the process $\overline \zeta^\star$ defined for all $t \in [0,T]$ by $\overline \zeta_t^\star := \overline v^\star(t, \mu^Y_t)$ is an optimal triple of parameters for the contract. 
    \end{enumerate}
\end{theorem}

{\color{black} The result is quite similar to those found in the literature on McKean--Vlasov problems, see, e.g., \citet*{bensoussan2013mean, bensoussan2015master, bensoussan2017interpretation, pham2018bellman, pham2017dynamic, bayraktar2018randomized}, or \citet*{djete2019mckean} for the most recent results.} We refer to Appendix \ref{proof:Pb_Principal_general} for the proof of this Proposition. In the following two subsections, we will specify the utility function $U^P$ in order to consider two different cases:
\begin{enumerate}
    \item[$\bullet$] the case of a risk--averse Principal, with a CARA utility function and a risk--aversion parameter $R_P > 0$, in Subsection \ref{sss:principal_cara};
    \item[$\bullet$] the case of a risk--neutral Principal, with $U^P (x) = x$, in Subsection \ref{sss:principal_RN}.
\end{enumerate}

All the functions defined in the first case will be indexed by $P$ to make the dependency explicit on $R_P$, for example $V^{P}$ for the value of the Principal's problem. For the sake of consistency, they will be indexed by $0$ in the second case, meaning that, informally, it is sufficient to set $R_P=0$ in a function defined in the risk--averse case to obtain the function in risk--neutral case.

\subsubsection{Principal with CARA utility}\label{sss:principal_cara}

Under a constant relative risk aversion specification of the utility function of the producer, that is $U^P (x) := - \erm^{- R_P x}$, for some risk--aversion $R_P > 0$, we are looking for a solution $v^P$ of \eqref{eq:pde_vP}  with the form 
\begin{align}\label{eq:def_uP}
    v^P (t, \mu^Y_t) = - \erm^{ R_P \big( \mathbb{E}^{\P_t} [L_t] - u^P (t, \mu_t^X) \big)}, \;
    \text{with} \; u^P (T, \mu_T^X) = 0.
\end{align}

To lighten to notations, we will denote, for all function $u : [0,T] \times \Pc (\R) \longrightarrow \R$,
\begin{align*}
    \overline{u}_{\mu^X} ( t, \mu ) &:= \int \partial_{\mu^X} u \big( t, \mu \big) (x) \mu (\drm x), \; 
    \overline{u}_{x, \mu^X} ( t, \mu ) := \int \partial_x \partial_{\mu^X} u \big( t, \mu \big) (x) \mu (\drm x), \\
    \text{ and }  \overline{u}_{\mu^X, \mu^X} ( t, \mu ) &:= \iint  \partial_{\mu^X}^2 u^P \big( t, \mu \big) ( x, \widetilde{x}) \mu (\drm x) \mu (\drm \widetilde x), \; \text{ for } \;  (t, \mu) \in [0,T] \times \Pc (\R),
\end{align*}
and we define the risk--aversion ratio $\overline{R}$ as $
    1/\overline R := 1/R_A + 1/R_P,$
with the convention $\overline{R} = 0$ if $R_A$ or $R_P$ is equal to zero. It remains to solve a simplified HJB equation  for the function $u^P$:
\begin{align}\label{eq:pde_principal_RA_general}
    0 = &\ - \partial_t u^P \big( t, \mu^X_t \big)
    + \int (g-f)(x) \mu_t^X (\drm x)
    + \dfrac{\theta}{2} \big(\sigma^{\circ}\big)^2
    - \dfrac{1}{2} \big( \sigma^{\circ} \big)^2 \overline{u}^P_{x,\mu^X} \big( t, \mu^X_t \big)
    + \dfrac{1}{2} \Big( \big(\sigma^{\circ}\big)^2 \overline{R} - \overline{\rho} \Big) \big( \overline{u}^P_{\mu^X} \big( t, \mu^X_t \big) \big)^2 \nonumber \\
    &- \dfrac{1}{2} \big( \sigma^{\circ} \big)^2 \overline{u}^P_{\mu^X, \mu^X} \big( t, \mu^X_t \big)
    + \dfrac{1}{2} \inf_{z \in \R} \Big\{ 
    F_0 \big( q \big( z,\overline{u}^P_{x,\mu^X} \big( t, \mu^X_t \big) \big) \big)
    + \overline{\rho} \big( (z^{-} \wedge A_{\textnormal{max}}) + \overline{u}^P_{\mu^X} \big( t, \mu^X_t \big) \big)^2
    \Big\},
\end{align}
with terminal condition $u^P \big(T, \mu^X_T \big) = 0$, and where 
\begin{align*}
    q (z, u) &:=\theta+ R_A z^2 - u \;
    \text{ and } \; F_0 (q) := \inf_{\gamma < 0} \{ \Sigma^{\star} ( \gamma ) q + c_{\beta}^{\star} ( \gamma) \}.
\end{align*}

In order to apply Theorem \ref{thm:Pb_Principal_general}, we however need to ensure that Condition \eqref{eq:conditions_martingale_general} holds for the function $v^P$. Therefore, we assume that there exists some $\overline p > \frac{p^\prime}{p^\prime-1}$ such that the following technical condition on the integrability of the contracts is enforced
\begin{equation}\label{integrability:rp:xi}
\tag{\textbf{CARA}}
\E^\P \bigg[ \sup_{0\leq t\leq T} \erm^{\overline p R_P \E^\P \big[ \xi_t^{\overline \zeta} \big| \Fc_t^{\circ} \big]}\bigg]<+\infty.
\end{equation} 
We will make use of the notation $\varepsilon = \sqrt{\overline p (p^\prime-1)/p^\prime} $. Under this assumption on contracts, we derive from Theorem \ref{thm:Pb_Principal_general} the following verification result.
\begin{proposition}\label{prop:Pb_Principal_RA_general}
    Let $u$ be a solution to {\rm PDE} \eqref{eq:pde_principal_RA_general}, smooth enough in the sense of {\rm Definition \ref{def:C12}} and satisfying
    \begin{align}\label{eq:conditions_martingale_RP}
       \E^\P \bigg[  \bigg(\int_0^T \big| \overline u_{\mu^X} \big( t, \mu^X_t \big) \big|^2 \drm t \bigg)^{\frac {p^\prime}{2}} + \sup_{0 \leq t \leq T} \exp\bigg( - \frac{q^\prime p^\prime}{p^\prime-1} R_P u (t, \mu^X_t)\bigg)  \bigg] < + \infty,
    \end{align}
   for $q' = \varepsilon/(\varepsilon-1)$. Moreover, let $\overline v^\star : [0,T] \times \Pc(\R) \longrightarrow \R^3$ be a function satisfying
    \begin{align*}
        h^P \big( \overline{u}_{\mu^X} \big( t, \mu_t^X \big), \overline{u}_{x, \mu^X} \big( t, \mu_t^X \big), \overline v^\star \big( t, \mu_t^X \big) \big) = \inf_{\overline v \in \R^3} h^P(\overline{u}_{\mu^X} \big( t, \mu_t^X \big), \overline{u}_{x, \mu^X} \big( t, \mu_t^X \big), \overline v \big),
    \end{align*}
    where 
    \begin{align}\label{eq:def_hP}
        h^P(u_1, u_2, \overline v) := &\
        \Sigma^{\star} (\gamma) \big(h + R_A z^2 - u_2 \big)
        + c_\beta^\star (\gamma)
        + c_\alpha^\star (z) 
        + 2 \overline{\rho} \big( z^{-} \wedge A_{\textnormal{max}} \big) u_1 \nonumber \\
        &+ \big(\sigma^{\circ}\big)^2 (R_A + R_P) \big(z + \overline z^{\mu} \big)^2
        - 2 R_P \big( \sigma^{\circ} \big)^2 \big( z + \overline z^{\mu} \big) u_1,
    \end{align}
    then
    \begin{enumerate}[label=$(\roman*)$]
        \item $V_0^{P} = - \erm^{ R_P ( \xi_0 - u ( 0, \mu^X_0 ) ) }$;
        \item the optimal payment rate to induce a reduction of the average consumption deviation is a process $Z^\star$, defined for all $t \in [0,T]$ by $Z^\star_t = z^\star \big(t, \mu^X_t \big)$, where the function $z^\star$ is solution of the minimisation problem in \eqref{eq:pde_principal_RA_general} and satisfies
        \begin{align*}
            z^{\star} (t, \mu) =0, \text{ when } \; \overline{u}_{\mu^X} (t, \mu ) \geq 0,
            \text{ and } \; z^{\star} (t, \mu) \in \big[\overline{u}_{\mu^X} (t, \mu ) \vee - A_{\textnormal{max}}, 0 \big]  \text{ when } \; \overline{u}_{\mu^X} \big(t, \mu \big) \leq 0;
        \end{align*}
        \item the optimal payment rate $\overline{Z}^{\mu,\star}$ is a process defined for all $t \in [0,T]$ by $\overline{Z}_t^{\mu,\star} = \overline{z}^{\mu,\star} \big( t, \mu^X_t \big)$ where
        \begin{align*}
            \overline{z}^{\mu,\star} ( t, \mu ) := - z^{\star} ( t, \mu ) + \dfrac{R_P}{R_A + R_P} \overline{u}_{\mu^X} ( t, \mu ).
        \end{align*}
        Recalling that $\overline{Z}_s^{\mu,\star} := \widetilde{\mathbb{E}}^{\P_s} \big[ Z_s^{\mu,\star} (\widetilde X_s) \big]$, we can arbitrary set that $z^{\mu,\star} ( t, \mu) (x) \equiv \overline{z}^{\mu,\star}( t, \mu )$ for all $x \in \R$.
        \item the optimal payment rate $\Gamma^{\star}$ to induce a reduction of the volatility of the consumption deviation is a process defined for all $t \in [0,T]$ by $\Gamma^{\star}_t = \gamma^{\star} \big( t, \mu^X_t \big)$ where
        \begin{align*}
            \gamma^{\star} ( t, \mu ) := - \max \bigg\{ \theta - \overline{u}_{x,\mu^X} ( t, \mu ) + R_A \big( z^{\star} ( t, \mu ) \big)^2, \dfrac{1}{\overline{\lambda}} \bigg\}, \; \text{ where  } \overline{\lambda} = \max_{k=1, \dots, d} \lambda^k;
        \end{align*}
        \item the second-best optimal contract is given by
        \begin{align*}
  U_A^{-1} (R_0)&  - \int_0^T \Hc (X_s, \mu^X_s, \overline \zeta_s^{\star}, \alpha_s^{\star})  \drm s 
            + \int_0^T Z^{\star}_s \big( \drm X_s - \widetilde{\mathbb{E}}^{\P_s} \big[ \drm \widetilde X_s \big] \big)
            + \dfrac{R_P}{R_A + R_P} \int_0^T \overline{u}_{\mu^X} \big(s, \mu^X_s \big) \widetilde{\mathbb{E}}^{\P_s} \big[ \drm \widetilde X_s \big] \nonumber \\
            &+ \dfrac{1}{2} \int_0^T \big( \Gamma^{\star}_s + R_A \big(Z^{\star}_s\big)^2 \big) \drm \langle X \rangle_s
            + \dfrac{1}{2} \dfrac{R_A R^2_P}{(R_A + R_P)^2} \big( \sigma^{\circ} \big)^2 \int_0^T  \big( \overline{u}_{\mu^X} \big(s, \mu^X_s \big) \big)^2 \drm s
            - \dfrac12 R_A \big( \sigma^{\circ} \big)^2 \int_0^T \big( Z^{\star}_s \big)^2 \drm s,
        \end{align*}
        where $\mathcal H$ was defined above in \eqref{eq:hamiltonian}.
    \end{enumerate}
\end{proposition}

We refer to Appendix \ref{proof:Pb_Principal_RA_general} for the proof of this proposition, which is a consequence of Theorem \ref{thm:Pb_Principal_general}.

\medskip

\textbf{Interpretation of the optimal contract.} 
The most interesting part of the optimal contract is given by Proposition \ref{prop:Pb_Principal_RA_general} $(v)$, and concerns the infinitesimal payment $Z^{\star}_s \big( \drm X_s - \widetilde{\mathbb{E}}^{\P_s} \big[ \drm \widetilde X_s \big] \big)$. Indeed, since $Z^\star \leq 0$ by Proposition \ref{prop:Pb_Principal_RA_general} $(ii)$, this payment is positive if the consumer's deviation consumption is below the mean of others' deviation. Conversely, if the consumer make less effort than the rest of the pool, this part of the compensation will be negative. Therefore, a part of the compensation is based on the comparison between the deviation consumption of a consumer and the mean of others' deviation consumption. 

\medskip

In addition to that, we can derive a more illuminating form of contract by denoting by $X^\circ$ the deviation consumption without common noise, whose dynamic under optimal efforts is given by
\begin{align*}
    \drm X^\circ_t = - \overline{\rho} \big( (Z^\star_t)^{-} \wedge A_{\textnormal{max}} \big) \mathrm{d} t + \sigma^{\star} ( \Gamma^\star_t ) \cdot \mathrm{d}W_t.
\end{align*}
The contract can thus be written in terms of the common noise as
\begin{align}\label{eq:contract_w0}
 U_A^{-1} (R_0)&  - \int_0^t \Hc \big(X^\circ_s + \sigma^\circ W^\circ_s, \overline \zeta_s^{\star} \big)  \drm s 
    + \int_0^t Z^{\star}_s \drm X^\circ_s
    + \dfrac{1}{2} \int_0^t \big( \Gamma^{\star}_s + R_A \big(Z^{\star}_s\big)^2 \big) \drm \langle X^\circ \rangle_s \nonumber 
   + \dfrac{R_P}{R_A + R_P} \sigma^\circ \int_0^t \overline{u}^P_{\mu^X} \big(s, \mu^X_s \big) \drm W^\circ_s \\
   & + \dfrac{1}{2} \dfrac{R_A R^2_P}{(R_A + R_P)^2} \big( \sigma^{\circ} \big)^2 \int_0^t  \big( \overline{u}^P_{\mu^X} \big(s, \mu^X_s \big) \big)^2 \drm s,
   \end{align}
where $\Hc (x, \overline \zeta^{\star}) :=   \dfrac12 H_d (z^\star) + \dfrac12 H_v (\gamma^\star) + f(x).$ We can then divide the study of the contract in two parts
    \begin{enumerate}[label=$(\roman*)$] 
        \item The first line of the contract defined in \eqref{eq:contract_w0} is the classical contract form for drift and volatility control, indexed on the process $X^\circ$, that is the part of the deviation consumption which is really controlled by the Agent:
        \begin{itemize}
            \item the contract is linear in the level of $X^\circ$ and its quadratic variation $\langle X^\circ \rangle$;
            \item by responding to the contract with optimal efforts, the consumer gains a certain amount of utility. Therefore, the Principal can subtract the certainty equivalent of this utility gain from the contract, that is the constant part $\int_0^t \Hc \big(X^\circ_s + \sigma^\circ W^\circ_s, \overline \zeta_s^{\star} \big)  \drm s$;
            \item due to the risk--aversion of the consumer, the additional payment $\frac{1}{2} R_A \big(Z^{\star}_t\big)^2 \drm \langle X^\circ \rangle_t$ is needed to compensate the infinitesimal payment $Z_t^\star \drm X^\circ_t$.
        \end{itemize}
Since the common noise $W^\circ$ represents in our framework the climate hazards, the process $X^\circ$ can also be seen as the deviation consumption adjusted for climate hazards. Therefore, this part of the contract is a fixed compensation, independent of weather conditions.
        \item The other part of the contract is an indexation on the common noise, that is the remaining risk the Principal wants to give to the Agent. As for the infinitesimal payment $Z_t \drm X_t$, due to the risk aversion of the consumer, the term $
            \frac{R_P}{R_A + R_P} \sigma^\circ \overline{u}^P_{\mu^X} \big(s, \mu^X_s \big) \drm W^\circ_s,$
        must be compensated by 
        \begin{align*}
            \dfrac12 R_A \bigg( \dfrac{R_P}{R_A + R_P} \sigma^\circ  \overline{u}^P_{\mu^X} \big(s, \mu^X_s \big) \bigg)^2 \drm \langle W^\circ \rangle_s = \dfrac{1}{2} \dfrac{R_A R^2_P}{(R_A + R_P)^2} \big( \sigma^{\circ} \big)^2 \big( \overline{u}^P_{\mu^X} \big(s, \mu^X_s \big) \big)^2 \drm s.
        \end{align*}
    \end{enumerate}
    
We can already notice that, if the Principal is risk--neutral, she will not use at all the common noise to provide incentives to the Agent. Indeed, since she is risk--neutral and the consumers are risk--averse, it is too costly to share the risk, and she can bear it alone. We refer to the next Subsection for the detailed contract in the risk--neutral case.

\medskip

The conclusion of this interpretation is that the indexation of the contract on others allows the Principal to divide the deviation consumption of an Agent into two parts: the part which is really controlled by the Agent, $X^\circ$, and the common noise. Hence, she offers a compensation indexed on the controlled deviation $X^\circ$ to encourage the Agent for making effort on the drift and the volatility. Moreover, if she is risk--averse, she adds to this contract a part indexed on the common noise, to share the remaining risk, even if regulatory rules prevent her from using the common noise directly in the contract.

\begin{remark}\label{rk:consumer_RN_general}
    In the case where the consumers are risk--neutral $(R_A =0)$, {\rm PDE} \eqref{eq:pde_principal_RA_general} reduces to
    \begin{align*}
        0 = &\ - \partial_t u^P \big( t, \mu_t^X \big)
        + \int (g-f)(x) \mu_t^X (\drm x)
        + \dfrac{\theta}{2} \big(\sigma^{\circ}\big)^2
        - \dfrac{1}{2} \big( \sigma^{\circ} \big)^2 \overline{u}^P_{x, \mu^X} \big( t, \mu_t^X \big)
        - \dfrac{1}{2} \overline{\rho} \big( \overline{u}^P_{\mu^X} \big( t, \mu_t^X \big) \big)^2  \\
        &- \dfrac{1}{2} \big( \sigma^{\circ} \big)^2 \overline{u}^P_{\mu^X, \mu^X} \big( t, \mu_t^X \big)
        + \dfrac{1}{2} F_0 \big( \theta - \overline{u}^P_{x, \mu^X} \big( t, \mu_t^X \big) \big)
        + \dfrac{1}{2} \inf_{z \in \R} \Big\{ 
        \overline{\rho} \big( (z^{-} \wedge A_{\textnormal{max}}) + \overline{u}^P_{\mu^X} \big( t, \mu_t^X \big) \big)^2
        \Big\}.
    \end{align*}
    
Noting that the infimum is attained for all $t \in [0,T]$ at $Z_t^{\star} = \overline{u}^P_{\mu^X} \big( t, \mu_t^X \big)$, the optimal payment rate $\Gamma^{\star}$ to induce a reduction of the volatility of the consumption deviation is
    \begin{align*}
        \Gamma_t^{\star} = - \max \bigg\{ \theta - \overline{u}^P_{x,\mu^X} \big( t, \mu_t^X \big), \dfrac{1}{\overline{\lambda}} \bigg\},
    \end{align*}
    and the optimal payment rate $\overline{Z}^{\mu, \star}$ is
    \begin{align*}
        \overline{Z}_t^{\mu,\star} = - Z_t^{\star} + \overline{u}^P_{\mu^X} \big( t, \mu_t^X \big) = 0 \; \text{ for all } t \in [0,T].
    \end{align*}
    
Hence, the certainty equivalent of the Principal's utility and the resulting optimal contract does not depend on the producer's risk aversion $R_P$. Moreover, since $\overline{Z}^{\mu, \star}=0$, the contract is not indexed on the law of other consumers. This result is particularly interesting since this contract has a classical form in the sense of {\rm \cite{aid2018optimal}} extended to a common noise model. Note that unlike the literature on Principal--Agent problems considering one Agent, the case $R_A=0$ does not coincide with the first--best problem $($see \textnormal{Appendix \ref{sec:firstbest})}.
    \end{remark}

% \todo[inline]{E: Compare with the first best case. Je n'ai pas l'impression que ce soit pareil. Dans le First Best, on a encore $R_A$ dans l'HJB equation !}

\subsubsection{Risk--neutral Principal}\label{sss:principal_RN}

The value function of a risk--neutral Principal is defined as $V_0^{0} = \sup_{\P \in \Qc} \mathbb{E}^{\P} \big[- \mathbb{E}^{\P_T} \big[ L^{\overline \zeta^\P}_T \big] \big]$. Under this specification for $U^P$, we are looking for a solution $v^0$ of \eqref{eq:pde_vP} with the form
\begin{align}\label{eq:def_u0}
    v^0 (t, \mu_t^Y) = - \mathbb{E}^{\P_t} \Big[ L^{\overline \zeta}_t \Big] + u^0 \big( t, \mu_t^X \big), \;
    \text{with} \; u^0 \big(T, \mu_T^X \big) = 0.
\end{align}

By adapting the reasoning of Proposition \ref{prop:Pb_Principal_RA_general}, we obtain the following HJB equation  associated to $u^0$
\begin{align}\label{eq:pde_principal_RN_general}
    0 = &\ - \partial_t u^{0} (t, \mu_t^X)
    + \int (g-f)(x) \mu_t^X (\drm x) 
    + \dfrac{\theta}{2} \big(\sigma^{\circ}\big)^2
    - \dfrac{1}{2} \big( \sigma^{\circ} \big)^2 \overline{u}^0_{x, \mu^X} \big( t, \mu^X_t \big)
    - \dfrac{1}{2} \overline{\rho} \big( \overline{u}^0_{\mu^X} \big( t, \mu^X_t \big) \big)^2 \nonumber \\
    &- \dfrac{1}{2} \big( \sigma^{\circ} \big)^2 \overline u_{\mu^X, \mu^X}^{0}  \big( t, \mu^X_t \big)
    + \dfrac{1}{2} \inf_{z \in \R} \Big\{ 
    F_0 \big( q \big(z,\overline{u}^0_{x,\mu^X} \big( t, \mu^X_t \big) \big) \big)
    + \overline{\rho} \big( \big( z^{-} \wedge A_{\textnormal{max}} \big) + \overline{u}^0_{\mu^X} \big( t, \mu^X_t \big) \big)^2
    \Big\},
\end{align}
with terminal condition $u^0 \big( T, \mu^X_T \big) = 0$. This result can be found intuitively by setting $R_P=0$ in \eqref{eq:pde_principal_RA_general}.
Similarly, we can deduce the following result from Proposition \ref{prop:Pb_Principal_RA_general} with $R_P=0$. 

\begin{proposition}\label{prop:Pb_Principal_RN_general}
    If there is a solution $u$ to the {\rm PDE} \eqref{eq:pde_principal_RN_general}, smooth enough in the sense of Definition \ref{def:C12}, such that the following condition is satisfied
    \begin{align}\label{eq:conditions_martingale_RN}
        \E^\P \bigg[ \bigg(\int_0^T \big| \overline u_{\mu^X} \big( t,\mu_t^X \big) \big|^2\mathrm{d}t\bigg)^{1/2}\bigg] < + \infty,
    \end{align} 
    and a function $\overline v^\star : [0,T] \times \Pc (\R) \longrightarrow \R^3$ satisfying 
    \begin{align*}
        h^0 \big(\overline{u}_{\mu^X} \big(t, \mu^X_t \big), \overline{u}_{x, \mu^X} \big(t, \mu^X_t \big), \overline v^\star \big(t, \mu^X_t \big) \big) = \inf_{\overline v \in \R^3} h^0 \big( \overline{u}_{\mu^X} \big(t, \mu^X_t \big), \overline{u}_{x, \mu^X} \big(t, \mu^X_t \big), \overline v \big),
    \end{align*}
    then
    \begin{enumerate}[label=$(\roman*)$]
        \item $V_0^{0} = - \xi_0 + u ( 0, \mu^X_0 )$;
         \item the optimal payment rate to induce a reduction of the average consumption deviation is the process $Z^\star$ defined for all $t \in [0,T]$ by $Z^\star_t = z^\star \big( t, \mu_t^X \big)$ where the function $z^\star$ is the optimiser  of the minimisation problem in \eqref{eq:pde_principal_RN_general}, and satisfies
        \begin{align*}
            z^{\star} (t,\mu) =0, \mbox{ when } \overline{u}_{\mu^X} (t,\mu) \geq 0,
            \mbox{ and } z^{\star}(t,\mu) \in \big[\overline{u}_{\mu^X} (t,\mu) \vee - A_{\textnormal{max}}, 0 \big]  \mbox{ when } \overline{u}_{\mu^X} (t,\mu) \leq 0;
        \end{align*}
        \item the optimal payment rate $\overline{Z}^{\mu, \star}$ is equal to $-Z^{\star}$;
        \item the optimal payment rate $\Gamma^{\star}$ is defined for all $t \in [0,T]$ by $\Gamma_t^{\star} := \gamma_t^\star \big( t, \mu^X_t \big)$ where
        \begin{align*}
            \gamma^{\star} (t, \mu) := - \max \bigg\{\theta- \overline{u}_{x,\mu^X} (t, \mu)  + R_A (z^{\star} (t, \mu))^2, \dfrac{1}{\overline{\lambda}} \bigg\};
        \end{align*}
               \item let $\overline \zeta^\star := \big( Z^\star, \overline Z^{\mu, \star}, \Gamma^\star \big)$, then the second--best optimal contract is given by
        \begin{align*}
         \hspace{-3em}  \xi_0  - \int_0^t \Hc (X_s, \mu^X_s, \overline \zeta_s^{\star}, \alpha_s^{\star})  \drm s 
            + \int_0^t Z^{\star}_s \big( \drm X_s - \widetilde{\mathbb{E}}^{\P_s} \big[ \drm \widetilde X_s \big] \big)
            + \dfrac{1}{2} \int_0^t \big( \Gamma^{\star}_s + R_A \big(Z^{\star}_s\big)^2 \big) \drm \langle X \rangle_s- \dfrac12 R_A \big( \sigma^{\circ} \big)^2 \int_0^t \big( Z^{\star}_s \big)^2 \drm s.
        \end{align*}
    \end{enumerate}
\end{proposition}

\begin{proof}
    The proof consists simply in a slight adaptation of the proof in Appendix \ref{proof:Pb_Principal_RA_general}, by showing that the function $v^0$, defined as \eqref{eq:def_u0}, satisfies the assumptions necessary for the application of Theorem \ref{thm:Pb_Principal_general}. Noticing that
    \begin{align*}
        \int \partial_{\mu^X} v(t, \mu^Y_t)(x, \ell) \mu^Y_t (\drm x, \drm \ell) = \overline u_{\mu^X} \big( t,\mu_t^X \big) \; \text{ and } \; \int \partial_{\mu^L} v(t, \mu^Y_t)(x, \ell)  \mu^Y_t (\drm x, \drm \ell) = -1,
    \end{align*}
    and since $u$ satisfies the condition \eqref{eq:conditions_martingale_RN}, we directly deduce that $v$ satisfies \eqref{eq:conditions_martingale_general}.
\end{proof}

\begin{remark}
    Notice that the risk--neutral Principal's problem can be rewritten as follows
    $V_0^{0} = \sup_{\P \in \Qc} \mathbb{E}^{\P} \big[- L^{\overline \zeta^\P}_T \big]$. Therefore, it is a standard stochastic control problem and can thus be solved in the classical way. Nevertheless, for the sake of consistency throughout this paper, we choose to solve it using {\rm Theorem \ref{thm:Pb_Principal_general}}.
\end{remark}

{\textbf{Interpretation of the optimal contract.} As in the previous subsection, to better see the compensation the consumer will receive, we can denote by $X^\circ$ his deviation consumption without common noise, and write the contract in term of common noise:
\begin{align}\label{eq:contract_w0_RN}
    &\ \xi_t^{\overline\zeta^\star} = \xi_0 - \int_0^t \Hc \big(X^\circ_s + \sigma^\circ W^\circ_s, \overline \zeta^\star_s \big)  \drm s 
    + \int_0^t Z^{\star}_s \drm X^\circ_s
    + \dfrac{1}{2} \int_0^t \big( \Gamma^{\star}_s + R_A \big(Z^{\star}_s\big)^2 \big) \drm \langle X^\circ \rangle_s.
\end{align}

Therefore, in the risk--neutral case, the optimal contract is the classical contract form for drift and volatility control, indexed on the process $X^\circ$, the deviation consumption adjusted for climatic hazards. This result comes from the fact that the Principal uses the indexation of the contract on the deviation of others, $\overline{z}^{\mu}$, to isolate the common noise. Hence, she has a clear view of what is the deviation consumption of the Agents, without the common noise ($X^{\circ}$). Therefore, she can offer a payment which is only indexed on the part of the consumption which is really optimised by the Agents. In this particular case of risk--averse Agents and risk--neutral Principal, the Principal can bear the risk alone and offers a contract which does not depend on the common noise.

\subsection{Application to linear energy value discrepancy (EVD)}\label{sec:applicationEVD}

{\color{black} The energy value discrepancy is the difference between a consumer's preference toward his deviation consumption (represented by the function $f$) and the production costs of this deviation (represented by $g$).
Following the line of \cite{aid2018optimal}, to obtain closed--form solutions, we consider in this section the case where 
\begin{align*}
    (f-g) (x) = \delta x, \; x \in \mathbb{R}.
\end{align*} 
%, comparing the Agents' preferences for their consumption with the Principal's preference for production can be summarised by studying the sign of $\delta$. 
Intuitively, if $\delta$ is positive, this means that the energy is more valuable for the consumer than it is costly for the producer. Therefore, a reduction of the consumption has a negative effect more important on the consumers' utilities than the positive effect on the producer's utility. Similarly, $\delta$ negative implies that an increase of consumption induces more cost for the producer than the benefit generated for the consumer. Therefore, intuitively, if $\delta$ is negative, it will be easier for the Principal to incentivise the consumers to reduce their consumption. }

\medskip

Under this assumption, we derive a closed-form solution in both cases of a risk--averse and risk--neutral Principal and more explicit optimal payment rates. To lighten the notation, we define
\begin{align*}
    \overline h(t, z) := F_0 \big(\theta+ R_A z^2 \big) + \overline{\rho} \big( (z^{-} \wedge A_{\textnormal{max}}) + \delta (T-t) \big)^2,\; (t,z)\in[0,T]\times\R.
\end{align*}
We sum up the results in the following proposition.
% which is of the form
% \begin{align*}
%     u^P (t, \mu^X) &= b(t) \int x \mu^X (\drm x) + c(t),
%     \text{ with terminal conditions } \; b(T) = c(T) = 0.
% \end{align*}

% The partial derivatives of $u^P$ are
% \begin{align*}
%     \partial_t u^P ( t, \mu^X ) &= b'(t) \int x \mu^X (\drm x) + c'(t), \; \text{and} \;
%     \partial_{\mu^{X}} u^P ( t, \mu^X )  (x) = b(t),
% \end{align*}
% and the others are worth zero. Hence, by plugging this guess into the PDE \eqref{eq:pde_principal_RA_general}, we obtain the following restrictions:
% \begin{align*}
%     c'(t) = &\ m^P(t), \; 
%     b'(t) = - \delta, \; b(T) = c(T) = 0, \nonumber \\
%     \text{where } \; m^P (t) := &\ \dfrac{\theta}{2} \big( \sigma^{\circ} \big)^2 
%     + \dfrac{1}{2} \bigg( \big( \sigma^{\circ} \big)^2 \dfrac{R_P R_A}{R_A + R_P} - \overline{\rho} \bigg) b(t)^2 
%     + \dfrac{1}{2}  \inf_{z \in \R} \Big\{ 
%     F_0 \big(\theta+ R_A z^2 \big)
%     + \overline{\rho} \big( (z^{-} \wedge A_{\textnormal{max}}) + b(t) \big)^2
%     \Big\}.
% \end{align*}

% This provides
% \begin{align*}
%     b(t) = \delta (T-t), \; \text{ and } \; c(t) = - \int_t^T m^P(s) \drm s, \; t \in [0, T ].
% \end{align*}
\begin{proposition}\label{prop:Pb_Principal_linear}
    Let the energy value discrepancy be linear, i.e. $(f-g)(x) = \delta x$, $x \in \mathbb{R}$. Define the certainty equivalent function $u^P$ as
    \begin{align*}
        u^P (t, \mu_t^X) := &\ \delta (T-t) \int x \mu_t^X (\drm x) - \int_t^T m^P(s) \drm s, \;
        \text{where} \; m^P (t) := \dfrac{\theta}{2} \big( \sigma^{\circ} \big)^2 
        + \dfrac{1}{2} \Big( \big( \sigma^{\circ} \big)^2 \overline R - \overline{\rho} \Big) \delta^2 (T-t)^2
        + \dfrac{1}{2}  \inf_{z \in \R} \overline h(t, z),
    \end{align*}
    then
    \begin{enumerate}[label=$(\roman*)$]
        \item the optimal payment rate process $\overline \zeta^\star = \big(Z^\star, \overline Z^{\mu, \star}, \Gamma^\star \big)$ is a deterministic function of time, independent of $\sigma^\circ$ and is defined by
        \begin{align*}
            Z_t^{\star} &= \mathrm{Arg} \min_{z \in \mathbb{R}} \overline h(t, z), \;
            \overline{Z}_t^{\mu, \star} = -Z_t^{\star} + \dfrac{R_P}{R_A + R_P} \delta (T-t) \;
            \text{ and } \; \Gamma_t^{\star} = - \max \bigg\{\theta+ R_A (Z_t^{\star})^2, \dfrac{1}{\overline{\lambda}} \bigg\}.
        \end{align*}
        for all $t \in [0,T]$.
    \item the value function of a producer with a {\rm CARA} utility function and a risk--aversion parameter $R_P$ is given by $V_0^{P} = - \erm^{ R_P ( \xi_0- u^P ( 0, \mu^X_0 ) ) }$;
    \item the value function of a risk--neutral producer is given by $V_0^{0} = - \xi_0 + u^{0} ( 0, \mu^X_0 )$.
    \end{enumerate}
\end{proposition}

\begin{proof}
    The proof of this proposition is a straightforward application of Propositions \ref{prop:Pb_Principal_RA_general} and \ref{prop:Pb_Principal_RN_general} with the specification $(f-g)(x) = \delta x$, since the functions $u^P$ and $u^0$ respectively satisfy Conditions \eqref{eq:conditions_martingale_RP} and \eqref{eq:conditions_martingale_RN}. Indeed, noticing that $\overline u^P_{\mu^X} \big(t, \mu^X_t \big) = \overline u^0_{\mu^X}  \big(t, \mu^X_t \big) = \delta (T-t)$, we clearly have 
    \begin{align*}
        \E^\P \bigg[  \bigg(\int_0^T \big| \overline u^P_{\mu^X} \big( t, \mu^X_t \big) \big|^2 \drm t \bigg)^{\frac {p^\prime}{2}}  \bigg]
        = \delta^{p^\prime} \dfrac{T^{3p^\prime/2}}{3^{p^\prime/2}}  < + \infty \; \text{ and } \; 
        \E^\P \bigg[ \bigg(\int_0^T \big| \overline u^0_{\mu^X} \big( t,\mu_t^X \big) \big|^2\mathrm{d}t\bigg)^{1/2}\bigg]
        = \delta \dfrac{T^{3/2}}{\sqrt{3}} < + \infty.
    \end{align*}    
    Therefore, $u^0$ satisfies the condition \eqref{eq:conditions_martingale_RN}. To show that $u^P$ satisfies \eqref{eq:conditions_martingale_RP}, it remains to prove that
    \begin{align*}
        \E^\P \bigg[ \sup_{0 \leq t \leq T} \erm^{ - \frac{q^\prime p^\prime }{p^\prime -1} R_P \big( \delta (T-t) \E^\P [X_t | \Fc^\circ_t] + \int_t^T m^P(s) \drm s \big) }  \bigg],
        < + \infty
    \end{align*}
    which is true since $m^P$ is continuous and $X$ has bounded drift and volatility.
\end{proof}

\medskip

The above proposition underlines the fact that the optimal payment rates $(Z^\star, \Gamma^\star)$ are the same, in both cases of a risk--averse or risk--neutral Principal. Hence, the efforts of the consumers on their deviation consumptions will be the same, whatever the risk aversion of the Principal. The Principal controls the risk she wants to bear thanks to the control $\overline Z^{\mu}$. Indeed, in the risk--neutral case, the Principal does not care about the risk, hence $\overline Z^{\mu}$ is such that the contract does not depend on the common noise. On the other hand, in the risk--averse case, the Agent is remunerated for a part of the common noise: the risk induced by the common noise is shared between the Agent and the Principal. 

\medskip
Moreover, one can notice that in this particular case, the only information that the Principal uses from the conditional law $\mu^X$ is actually the conditional mean. Indeed, the only term with $\mu^X$ appears in the certainty equivalent function $u^P$ of the producer, with the form $\int x \mu_t^X (\drm x)$ which is equal to $\E^{\P} [X_t | \Fc^\circ_t]$.

\begin{remark}\label{rk:Zstar}
    For all $t \in [0,T]$, the optimal payment rate $Z_t^\star$ is equal to zero if $\delta$ is non-negative and lies between $\delta (T-t) \vee - A_{\textnormal{max}}$ and $0$ if $\delta < 0$.
\end{remark}

\section{Comparison with classical contracts}\label{sec:classical_contract}

The aim of this section is to study the benefits of indexing the contracts on the distribution of the deviation of other consumers.
% when incentives are limited to payments for efforts of the considered consumer.
We thus consider, as a benchmark case, the producer's problem when incentives are limited to payments for efforts of the considered consumer. More precisely, we consider only contracts controlled by $\zeta^{0} := (Z, 0, \Gamma)$ instead of $\zeta := (Z, Z^\mu, \Gamma)$. We denote by $\mathcal V^0$ the corresponding restriction of $\mathcal V$. We provide the optimal contract $\xi^0$ and effort of the consumers in Subsection \ref{ss:contract_zeta0}. We focus on the linear energy value discrepancy (EVD) case to compare the results, in terms of consumers' effort and producer's utility, between the two types of contracts, in both cases of a risk--averse (see Subsection \ref{ss:cara_zeta0}) or a risk--neutral producer (Subsection \ref{ss:RN_zeta0}). 

% \medskip

% {\color{black} To compare the results in terms of effort and producer's utility between the two types of contract, we need to compute numerically the optimal payment rate $Z^\star$, since it is the only closed--form solution missing in the linear energy value discrepancy (EVD) case, see Proposition \ref{prop:Pb_Principal_linear}. }

\subsection{Optimal form of contracts}\label{ss:contract_zeta0}

We consider the same model,
% \begin{align*}
% 	V^{A}_0 \big( \xi^0 \big) := \sup_{\mathbb{P} \in \Pc} \mathbb{E}^{\mathbb{P}} \Big[ - \erm^{ - R_A Y_T } \Big], \; \text{ where } \; Y_t &:= \xi^0_t - \int_0^t \big( c \big( \nu^\P_s \big) - f ( X_s) \big) \mathrm{d} s,
% \end{align*}
with the same dynamic of the deviation consumption $X$,
% \begin{align*}
%     \drm X_t = - \alpha_t \cdot \mathbf{1}_d \mathrm{d} t + \sigma \big( \beta_t \big) \cdot \mathrm{d}W_t + \sigma^{\circ} \mathrm{d}W^{\circ}_t.
% \end{align*}
but we restrict our study to contracts offered by the Principal to a consumer depending only his deviation. This type of contracts has the classical form of contracts for drift and volatility control (see \cite{aid2018optimal}, \cite{cvitanic2017moral}, \cite{cvitanic2018dynamic}):
\begin{align*}
	\xi_T^0 = &\ \xi_0  - \int_0^T \Hc^\circ (X_t, \zeta^{0}_t)  \drm t 
    + \int_0^T Z_t \drm X_t
    + \dfrac{1}{2} \int_0^T \big( \Gamma_t + R_A Z_t^2 \big) \drm \langle X \rangle_t ,
\end{align*}
where $\zeta^{0} := (Z, 0,\Gamma)$, an $\R \times\{0\}\times  \R$--valued $\F$--predictable process, is the set of parameters chosen by the Principal and with $\xi_0= U_A^{-1} (R_0)$.

\medskip

% Given the form of the contract, the dynamic of $Y$, defined in the consumer's problem, is given by
% \begin{align*}
%     \drm Y_s &= \big( h^0(\zeta^0_s, X_s, \nu_s) - \Hc^0 \big( \zeta^0_s, X_s \big) \big) \drm s
%     + \drm M^0_s 
%     + \dfrac{1}{2} R_A \drm \langle M^0 \rangle_s, \\ 
%     \text{with} \; Y_0 &= \xi_0, \; 
%      h^0(\zeta, x, \nu) := - \alpha \cdot \mathbf{1}_d z + \dfrac12 \gamma \Big( \Sigma (\beta) + \big( \sigma^{\circ} \big)^2 \Big) - c(\nu) + f(x),
% \end{align*}
% and the process $(M^0_t)_{t \in [0,T]}$ is defined by
% \begin{align*}
%     M^0_t &:=  \int_0^t Z_s \big( \sigma \big( \beta_s \big) \cdot \mathrm{d}W_s + \sigma^{\circ} \drm W_s^{\circ} \big).
% \end{align*}

% \medskip

% The problem of the representative consumer is then equivalent to
% \begin{align*}
% 	V^{A}_0 ( \xi^0)
%     = \sup_{\P \in \Pc} \E^{\P} \Big[- \Ec (-R_A M^0_T) \erm^{- R_A ( \xi_0
%     + \int_0^T ( h^0(\zeta^0_s, X_s, \nu_s) - \Hc^0 ( \zeta^0_s, X_s) ) \drm s ) } \Big].
% \end{align*}
% where $\Ec (-R_A M^0_T)$ is a martingale under $\P$. Noticing that
% \begin{align*}
%     h(\zeta, x, v) - \Hc \big( \zeta, x \big) \leq 0
% \end{align*}
% with equality for the optimal drift and volatility efforts 

{\color{black} Intuitively, we have shown in Theorem \ref{thm:mfe} that the consumer's optimal response to a contract indexed by $\zeta = (Z, Z^\mu, \Gamma)$ is independent of $Z^\mu$, and in particular $\nu^{\star} \big(\zeta^0 \big) = (\alpha^{\star}(Z), \beta^{\star}(\Gamma))$. Therefore, in the case of classical contracts, mainly with $Z^\mu = 0$, the latter should be the same function of $(Z, \Gamma)$.
% , we obtain:
% \begin{align*}
% 	V^{A}_0 ( \xi^0)
%     \leq \erm^{- R_A \xi_0 } \sup_{\P \in \Pc} \E^{\P} \Big[- \Ec (-R_A M^0_T)  \Big] = \erm^{- R_A \xi_0 },
% \end{align*}
% with equality for $\nu^{\star}$.
This result coincides with the results of \cite{aid2018optimal}, and is consistent with the Remark \ref{rk:drift_vol_ctrl}. Thus, these contracts lead obviously to the same unique mean--field equilibrium as the one defined in Theorem \ref{thm:mfe}.}

% \medskip

% After solving the consumer's problem, one obtain contract on the form:
% \begin{align*}
%     \xi^0_T = &\ \xi_0 + \dfrac{1}{2} \int_0^T \big(
%     \overline{\rho} (Z_t^{-} \wedge A_{\textnormal{max}})^2
%     + c_{\beta}^{\star} ( \Gamma_t ) 
%     - 2 f ( X_t) \big) \mathrm{d} t
% 	+ \int_0^T  Z_t \sigma^{\star} ( \Gamma_t ) \cdot \mathrm{d}W_t 
% 	+ \dfrac{1}{2} R_A \int_0^T  Z_t^2 \Sigma^{\star} ( \Gamma_t) \mathrm{d} t \\
% 	&+ \int_0^T Z_t \sigma^{\circ}  \drm W_t^{\circ}
% 	+ \dfrac{1}{2} R_A \int_0^T Z_t^2 \big(\sigma^{\circ}\big)^2 \mathrm{d} t.
% \end{align*}

\medskip

The Principal has to choose optimally the indexation parameters $ (Z, \Gamma)$ to maximise her utility. Her value function is the same as before, but restricted to the controls $\zeta^{0} \in \Vc^0$
\begin{equation}\label{pb:benchmark:principal}
	V_0^{0,P} := \sup_{\zeta^0 \in \Vc^0} \E^{\P} \Big[ U^P \Big( - \mathbb{E}^{\P} \Big[ L^{\zeta^0}_T \Big| \Fc^{\circ}_T \Big] \Big) \Big].
\end{equation}
% \todo[inline]{D: $L$ should either be controlled here, that is to say with exponents for the initial value and $\zeta^0$, or you should say that we can define a set of measures $\Pc^0$ (or whatever notation you use for the set of measures for the Principal) similarly as in Section 4.1, where $\overline Z^\mu$ is $0$.}

Following the lines of the Section \ref{sec:principal_problem}, 
we obtain an HJB equation  similar to \eqref{eq:pde_vP}, but with the supremum on $\zeta^0$ (intuitively by considering $Z^\mu = 0$ in \eqref{eq:pde_vP}).
Therefore, we can establish an analogous result to Theorem \ref{thm:Pb_Principal_general}.% the following Proposition is the equivalent of the Theorem \ref{thm:Pb_Principal_general} in this case.
% \begin{proposition}\label{prop:Pb_Principal_general_zeta0}
%     If there is a solution $v^0$ to the PDE \eqref{eq:pde_vP}, smooth enough in the sens of the Chain Rule under $\Cc^{1,2}-$Regularity, and a couple $ \zeta^{0,\star} \in \R^2$ satisfying:
%     \begin{align*}
%         h(\mu, \Delta_{\mu} v^0, \nabla_{\mu} v^{0}, \zeta^{0,\star}) = \sup_{ \zeta^0 \in \R^2} h(\mu, \Delta_{\mu} v^{0}, \nabla_{\mu} v^{0}, \zeta^{0}),
%     \end{align*}
%     such that the following condition is satisfied:
%     \begin{align}\label{eq:conditions_martingale_zeta0}
%         \mathbb{E}^\P \bigg[ \int_0^T \bigg( \int \big( - \partial_{\mu^X} v^0 ( t, \mu ) (x, \ell)  + Z^\star_t \partial_{\mu^L} v^0 ( t, \mu ) (x, \ell) \big) \mu (\drm x, \drm \ell) \bigg)^2 \drm t \bigg] < + \infty,
%     \end{align}
%     then:
%     \begin{enumerate}[label=(\roman*)]
%         \item $v^0(0,\mu) = V^{0,P}$;
%         \item $\zeta^{0,\star}$ is an optimal couple of parameters for the contract. 
%     \end{enumerate}
% \end{proposition}

% The proof of this Proposition is a slight adaptation of the proof in Appendix \ref{proof:Pb_Principal_general}, but intuitively, the result comes directly from the Theorem \ref{thm:Pb_Principal_general} with $z^{\mu} = 0$.

\subsection{Principal with CARA utility}\label{ss:cara_zeta0}

\subsubsection{Theoretical results}

Under a CARA specification of the utility function of the producer, by mimicking Section \ref{sss:principal_cara}, we expect a solution $v^{0,P}$ to the HJB equation  associated with \eqref{pb:benchmark:principal} to be given by
\begin{align}\label{eq:def_u0P}
    v^{0,P} (t, \mu^Y_t) = - \erm^{ R_P \big( \mathbb{E}^{\P_t} \big[ L^{\zeta^0}_t \big] - u^{0,P} (t, \mu_t^X) \big)},
\end{align}
where $u^{0,P}$ satisfies the following PDE (similar to \eqref{eq:pde_principal_RA_general}) 
\begin{align}\label{eq:pde_principal_RA_general_zeta0}
    0 = &\ - \partial_t u^{0,P} 
    +  \int (g-f)(x) \mu_t^X (\drm x) 
    + \dfrac{\theta}{2} \big( \sigma^{\circ} \big)^2 
    - \dfrac{1}{2} \big( \sigma^{\circ} \big)^2 \overline{u}^{0,P}_{x, \mu^X} 
    - \dfrac{1}{2} \big( \sigma^{\circ} \big)^2 \overline{u}^{0,P}_{\mu^X, \mu^X} 
    - \dfrac{1}{2}  \overline{\rho} \big( \overline{u}^{0,P}_{\mu^X}  \big)^2
    \nonumber \\
    &+ \dfrac{1}{2}  \inf_{z \in \R} \Big\{ 
    F_0 \big( q \big( z,\overline{u}^{0,P}_{x,\mu^X} \big) \big)
    + R_A \big(\sigma^{\circ}\big)^2 z^2 
    + \overline{\rho} \big( \big( z^{-} \wedge A_{\textnormal{max}} \big) + \overline{u}^{0,P}_{\mu^X} \big)^2
    + R_P \big( \sigma^{\circ} \big)^2 \big( - z + \overline{u}^{0,P}_{\mu^X} \big)^2
    \Big\},
\end{align}
with terminal condition $u^{0,P} ( T, \mu_T^X) = 0.$
We can state the equivalent of Proposition \ref{prop:Pb_Principal_RA_general} in our benchmark case with classical contracts.
\begin{proposition}\label{prop:Pb_Principal_RA_general_zeta0}
    If there is a solution $u$ to {\rm PDE} \eqref{eq:pde_principal_RA_general_zeta0}, smooth enough in the sense of {\rm Definition \ref{def:C12}}, satisfying the condition \eqref{eq:conditions_martingale_RP}, and a function $\overline v^{0,\star} : [0,T] \times \Pc (\Cc_T) \longrightarrow \R\times \{0\}\times \mathbb R$ that reaches the infimum of $h^P$ defined by \eqref{eq:def_hP},
    % satisfying
    % \begin{align*}
    %     h^{P}(\mu^X, \overline{u}_{\mu^X}, \overline{u}_{x, \mu^X}, \zeta^{0,\star}) = \inf_{\zeta^{0} \in \R^2} h^{P}(\mu^X, \overline{u}_{\mu^X}, \overline{u}_{x, \mu^X}, \zeta^{0}),
    % \end{align*}
    then
    \begin{enumerate}[label=$(\roman*)$]
        \item $V_0^{0,P} = - \erm^{ R_P ( \xi_0 - u ( 0, \mu^X_0 ) ) }$;
        \item the optimal payment rate to induce a reduction of the average consumption deviation is the process $Z^{0,\star}$ defined for all $t \in [0,T]$ by $Z^{0,\star}_t := z^{0,\star} \big(t, \mu^X_t \big)$ where the function $z^{0,\star}$ is the minimiser of the minimisation problem in \eqref{eq:pde_principal_RA_general_zeta0};
        \item the optimal payment rate $\Gamma^{0,\star}$ to induce a reduction of the volatility of the consumption deviation is defined for all $t \in [0,T]$ by $\Gamma_t^{0,\star} := \gamma^{0,\star} \big( t, \mu^X_t \big)$ where
        \begin{align*}
            \gamma^{0,\star} (t, \mu) := - \max \bigg\{\theta- \overline{u}_{x,\mu^X} (t, \mu)  + R_A (z^{0,\star} (t, \mu) )^2, \dfrac{1}{\overline{\lambda}} \bigg\}.
        \end{align*}
    \end{enumerate}
\end{proposition}

\begin{remark}\label{remark:classicalcontract}
    In the case where the consumers are risk--neutral $(R_A =0)$, the {\rm PDE} reduces to
    \begin{align*}
        0 = &\ - \partial_t u^{0,P}
        +  \int (g-f)(x) \mu^X (\drm x) 
        + \dfrac{\theta}{2} \big( \sigma^{\circ} \big)^2 
        - \dfrac{1}{2} \big( \sigma^{\circ} \big)^2 \overline{u}^{0,P}_{x, \mu^X}
        - \dfrac{1}{2} \big( \sigma^{\circ} \big)^2 \overline{u}^{0,P}_{\mu^X, \mu^X}
        - \dfrac{1}{2}  \overline{\rho} \big( \overline{u}^{0,P}_{\mu^X} \big)^2 \\
        &+ \dfrac{1}{2} F_0 \big( \theta - \overline{u}^{0,P}_{x,\mu^X} \big)
        + \dfrac{1}{2}  \inf_{z \in \R} \bigg\{ 
        \overline{\rho} \big( \big( z^{-} \wedge A_{\textnormal{max}} \big) + \overline{u}^{0,P}_{\mu^X} \big)^2
        + R_P \big( \sigma^{\circ} \big)^2 \big( - z + \overline{u}^{0,P}_{\mu^X} \big)^2
        \bigg\},
    \end{align*}
    
    % Hence, the optimal payment rate for the drift lie between $0$ and $\widetilde{\delta}$. Moreover, the optimal payment rate $\gamma^{\star}$ to induce a reduction of the volatility of the consumption deviation is
    % \begin{align*}
    %     \gamma^{\star} := - \max \left\{\theta- \int \partial_x \partial_{\mu^X} u^P ( t, \mu^X ) (x) \mu^X (\drm x), \dfrac{1}{\overline{\lambda}} \right\}.
    % \end{align*}
    
    Hence, the infimum is attained for $z^\star := \overline{u}^{0,P}_{\mu^X}$ and, similarly to {\rm Remark \ref{rk:consumer_RN_general}}, the resulting optimal contract do depend on the producer's risk aversion $R_P$. 
\end{remark}

In order to obtain closed--form solutions, we can study the case of the linear EVD, similarly to Proposition \ref{prop:Pb_Principal_linear}. To lighten the notations, we denote
\begin{align*}
    \overline h^P (t,z) := F_0 (\theta + R_A z^2)
    + R_A \big(\sigma^{\circ}\big)^2 z^2 
    + \overline{\rho} \big( (z^{-} \wedge A_{\textnormal{max}}) + \delta (T-t) \big)^2
    + R_P \big( \sigma^{\circ} \big)^2 \big( - z + \delta (T-t) \big)^2.
\end{align*}
\begin{proposition}\label{prop:Pb_Principal_RA_linear_simple}
    Let the energy value discrepancy be linear, i.e. $(f-g)(x) = \delta x$, $x \in \mathbb{R}$. Then,
    \begin{enumerate}[label=$(\roman*)$]
        \item the producer's value function is given by $V_0^{0,P} = - \erm^{ R_P ( \xi_0 - u^{0,P} ( 0, \mu^X_0 )) }$ where the certainty equivalent function $u^{0,P}$ is characterised by 
        \begin{align*}
            u^{0,P} (t, \mu_t^X) &= \delta (T-t) \int x \mu_t^X (\drm x) - \int_t^T m^{0,P}(s) \drm s, \; \text{ where } \;
        % \end{align*}
        % where
        % \begin{align*}
            m^{0,P} (t) := \dfrac{\theta}{2} \big( \sigma^{\circ} \big)^2
            - \dfrac{1}{2}  \overline{\rho} \delta^2 (T-t)^2
            + \dfrac{1}{2}  \inf_{z \in \R} \overline h^P (t,z),
        \end{align*}
        \item the optimal payment rate process $\zeta^{0,\star} = \big( Z^{0,\star}, 0, \Gamma^{0,\star} \big)$ is a deterministic function of time given by
        \begin{align*}
            Z_t^{0,\star} = \mathrm{Arg} \min_{z \in \mathbb{R}} \overline h^P (t,z) \;
            \text{ and } \; \Gamma_t^{0,\star} = - \max \bigg\{ \theta + R_A \big( Z_t^{0,\star} \big)^2, \dfrac{1}{\overline{\lambda}} \bigg\} \; \text{ for } t \in [0,T].
        \end{align*}
    \end{enumerate}
\end{proposition}

The main point is to study the consumers' effort and the producer's utility when she offers \textit{new contracts} compared to \textit{classical ones}, through a comparative analysis of the previous proposition with Proposition \ref{prop:Pb_Principal_linear}.

\subsubsection{Numerical results and discussion}\label{sss:cara_zeta0_comparison}

To compare the efforts of consumers on the average level of their consumption, we need to compare $Z^\star$ with $Z^{0,\star}$, and $\Gamma^\star$ with $\Gamma^{0,\star}$ for their efforts on the volatility, since, by Theorem \ref{thm:mfe}, for $t \in [0,T]$,
\begin{align*}
    \alpha_t^\star := a^\star (Z_t^\star) \; \text{ and } \; \beta_t^\star := b^\star (\Gamma_t^\star).
\end{align*}
Although some inequalities can be obtained analytically, it may be required to numerically compute the optimal payment rate $Z^\star$, since this is the only quantity for which we cannot get a closed--form expression in the linear EVD case, see Proposition \ref{prop:Pb_Principal_linear}. For numerical computations, the calibration of the parameters detailed in \cite{aid2018optimal} is used. Indeed, even if they considered in \cite{aid2018optimal} the consumption, by fixing the initial condition of the consumption to be zero, the consumption process $X$ in their work becomes directly the observed reduction of consumption, which is our canonical process $X$. More precisely, we will consider most of the time $T=5.5$ h, $\theta=4 \times 10^{-3}$ p/kW$^2$h, $\kappa=11.76$ p/kWh, $\delta=-55.44$ p/kWh, $R_P=6 \times 10^{-3}$ p$^{-1}$, $R_A=5,7 \times 10^{-3}$ p$^{-1}$, $\rho=9.3  \times 10^{-5}$ kW$^2$h$^{-1}$p$^{-1}$, $\eta=1$, $\lambda=2.8  \times 10^{-2}$ kW$^2$h/p (p stands for pence). We will explicitly indicate different values for specific parameters, for instance when we investigate the influence of $R_P$ on the utility of the Principal. Concerning the volatility calibrated in \cite{aid2018optimal} (namely the nominal volatility, with value $0.085$ kW/h$^{1/2}$), we will suppose that it is equal to the total volatility without effort in our model, i.e. $\sqrt{\Sigma(1) + ( \sigma^\circ )^2}$. Hence, we will consider different values for $\sigma^\circ$ but such that the total volatility without effort remains constant, equal to $0.085$. Notice that the numerical computations will be performed for $d=1$.

\medskip

\textbf{Comparison of efforts.} In order to compare the effort of the consumers, we need to distinguish two cases, depending on the sign of the linear EVD parameter $\delta$.

\medskip

$(i)$ We first focus on the most representative case where $\delta \leq 0$, since empirical results in \cite{aid2018optimal} provide $\delta = -55.44$. However, analytic comparison of efforts in this case is not clear. Figure \ref{fig:delta_neg} represents the optimal efforts of the consumers on the drift and the volatility for two values of $R_P$ close to the one calibrated in \cite{aid2018optimal}, and when $50\%$ of the variance is explained by the common noise (\textit{i.e.} $\sigma^\circ = \sigma = 0.085/\sqrt{2} \approx 0.0601$). The blue lines represent the optimal efforts in the case of classical contracts, while the orange lines are dedicated to our new contracts. Recall that, with contracts indexed on others, the effort of the Agents on the drift and on the volatility are independent of the risk aversion of the Principal. Therefore, 
\begin{itemize}
    \item for $R_P = 0.006$ (upper graphs), thanks to the contract indexed on others, the Principal can incentivise the consumers to make more effort on the drift and on the volatility (orange curves compared to blue ones), for the entire duration of the contract;
    \item if the risk aversion of the Principal increases ($R_P = 0.03$, lower graphs), with the classical contract, she is asking the consumer to make more efforts, whereas with the contract indexed on others, she requires the same effort regardless of her risk aversion. Thus, we obtain that with the new contract, the consumer makes still more efforts in the beginning of the contract, but less after. 
\end{itemize}

\begin{figure}[!ht]
    \begin{center}
        \includegraphics[height=8cm]{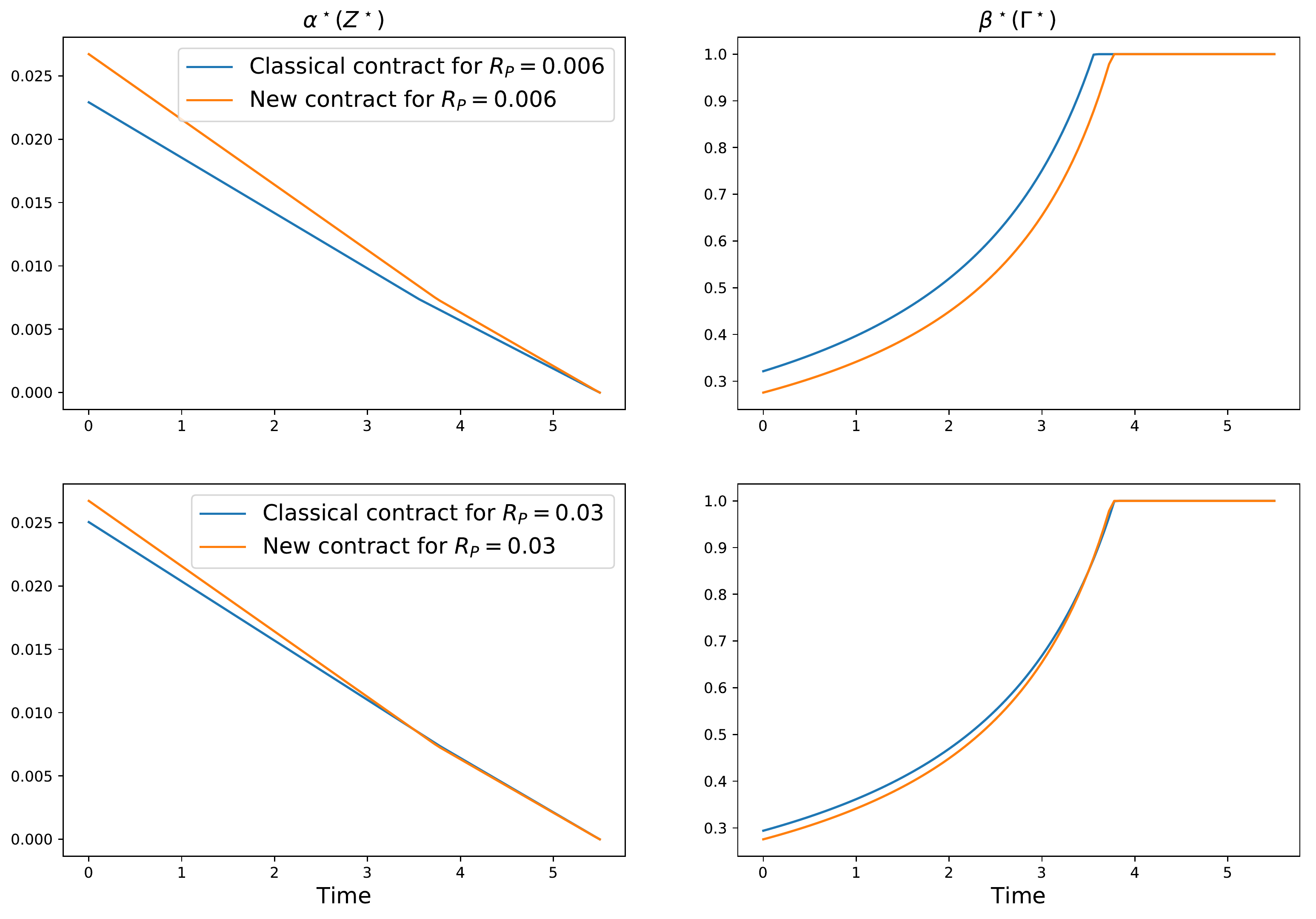}
        \caption{Comparison of efforts in the linear EVD case. \\ Parameters: $\sigma^\circ = \sigma = 0.085/\sqrt{2}$ and $R_P = 0.006$ (upper graphs) or $R_P = 0.03$ (lower graphs).}
        \label{fig:delta_neg}
    \end{center}
\end{figure}  

The fact that the efforts are decreasing in time can be explained as follows. At every moment $t \in [0,T]$ of the contract, what matters to the Principal is rather the integral between $0$ and $t$ of the efforts than the instantaneous efforts. Indeed, she cares about the deviation $X_t$ for all $t \in [0,T]$, which is defined by $X_t = - \int_0^t \alpha_t \drm t + \int_0^t \sigma \sqrt{\beta_t} \drm W_t + \int_0^t \sigma^\circ \drm W^\circ_t$, and the quadratic variation of $X_T$ given by $\langle X \rangle_T = \int_0^T \big( \sigma^2 \beta_t + \big(\sigma^\circ\big)^2 \big) \drm t$. Therefore, it is in the Principal's interest to ask for more effort at the beginning of the contract. In order to assess the benefit of our new contrats on the efforts of the consumers, we compute the following two quantities, 
\begin{align*}
    \Delta \alpha^\star := \frac{\int_0^T \big( a^\star (Z^\star_t) - a^\star \big(Z_t^{0,\star} \big) \big) \drm t }{\int_0^T a^\star \big(Z_t^{0,\star} \big) \drm t} \; \text{ and } \; \Delta \beta^\star := - \frac{ \sigma^2 \int_0^T \big(  b^\star (\Gamma^\star_t) - b^\star \big(\Gamma_t^{0,\star} \big) \big) \drm t }{\sigma^2 \int_0^T b^\star \big(\Gamma_t^{0,\star} \big) \drm t + T \big(\sigma^\circ\big)^2 }
\end{align*}
representing the relative gain respectively in mean and in volatility of the consumption between new and classical contracts. More precisely, if $\Delta \alpha^\star \geq 0$ (resp. $\Delta \beta^\star \geq 0$), the new contracts incentivise the consumers to decrease more the mean (resp. the volatility) of their total consumption at the end of the contract (at time $T$). The results are presented in Figure \ref{fig:diff_effort}. We choose values for $R_P$ of the same order than the estimation in \cite{aid2018optimal}. For $\sigma^\circ$, we set values such that $( \sigma^\circ )^2 = x \%$ of the total variance without effort (nominal variance), i.e. $\Sigma(1) + ( \sigma^\circ)^2 = 0.085^2$. For example, when $\sigma^\circ = 0.085$, then $(\sigma^\circ)^2 = 100 \%$ of the nominal variance: this means that the volatility in the deviation consumption is entirely related to climate hazards ($\sigma = 0$).

\begin{figure}[!ht]
    \begin{center}
        \includegraphics[height=6cm]{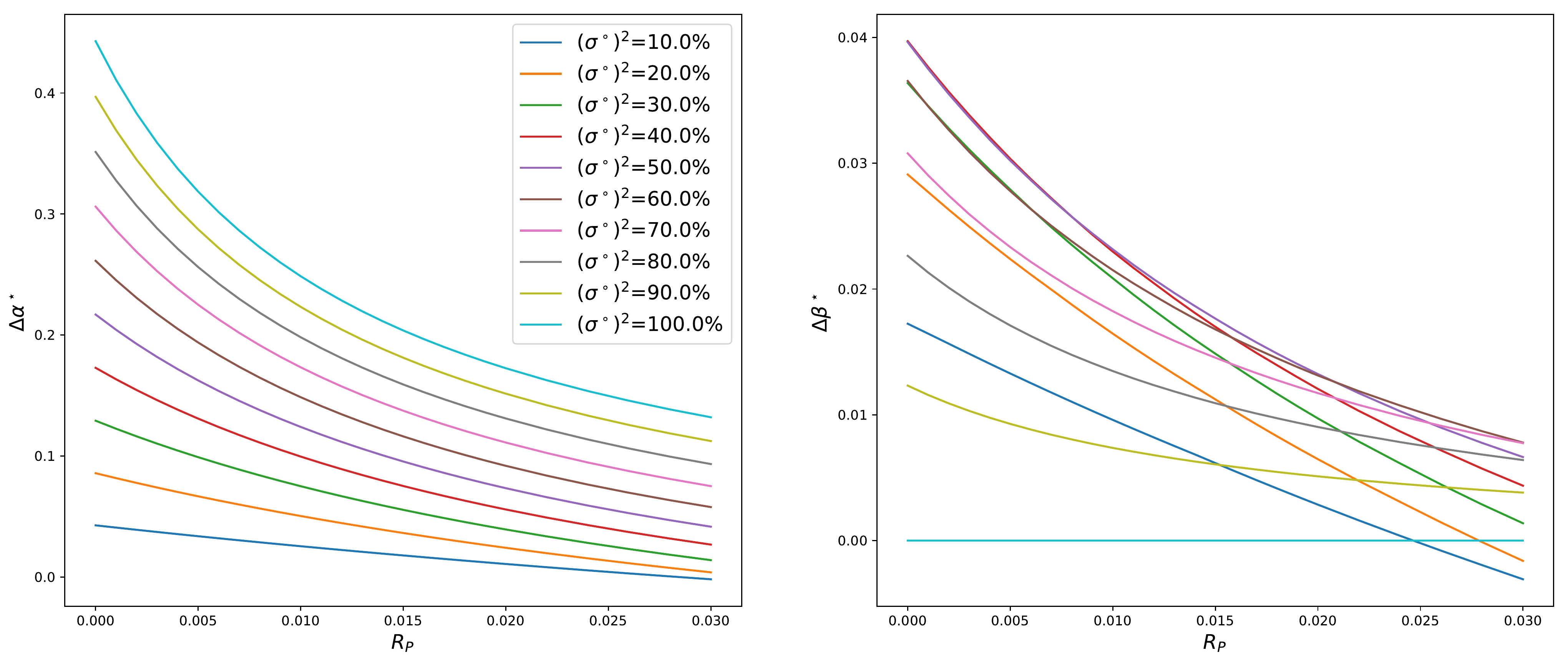}
        \caption{Relative gain in mean (left) and in volatility (right) of the consumption in the linear EVD case.}
        \label{fig:diff_effort}
    \end{center}
\end{figure}  

\medskip

Figure \ref{fig:diff_effort} show that, in most cases, our new contracts lead to a significant decrease in consumption on average and on volatility compared to classical contracts. However, when the risk--aversion of the Principal increases, this gain becomes negative. Nevertheless, it can be stressed that, on the one hand, for the parameters calibrated in \cite{aid2018optimal}, there is a significant gain regardless of the correlation with the common noise. On the other hand, even if the Principal has a relatively high risk--aversion, our new contracts allow a reduction on average and on volatility of the consumption if it is strongly impacted by weather conditions. Therefore, our contracts can help to better manage consumption cut-off during peak demand due to climate hazard.

% Therefore, in the two previous situations, the new contracts induce a beneficial change in the effort of the consumers, from the Principal's point of view, but also in the more global perspective of reducing energy consumption to preserve the climate.

\medskip

$(ii)$ Out of curiosity, we can also investigate the positive $\delta$ case. Using Remark \ref{rk:Zstar}, we can prove, analytically in this case, that $Z^{0,\star} \geq Z^{\star} = 0$. On the one hand, since the payment rate on the drift is positive in both cases and using Theorem \ref{thm:mfe}, the effort of the consumer to reduce his consumption in average is zero. On the other hand, this inequality leads to $\Gamma^{0, \star} \leq \Gamma^{\star} \leq 0$. Therefore, when the Principal offers classical contracts, she incentivises more the consumers to make effort on the volatility. Nevertheless, since $\lambda = 2.8 \times 10^{-2}$ and $\theta=4 \times 10^{-3}$, we can remark that the consumer will in fact make no effort on the volatility, since 
\begin{align*}
    \Gamma_t^{0,\star} = - \max \bigg\{\theta+ R_A (Z_t^{0,\star})^2, \frac{1}{\lambda} \bigg\} = - \frac{1}{\lambda},
\end{align*}
and thus 
\begin{align*}
    \beta_t^{\star}( \Gamma_t^{0,\star}) = 1 \wedge \big( \lambda \big( \Gamma_t^{0,\star})^- \big)^{-1/(\eta^k + 1)} \vee B_{\textnormal{min}} = 1.
\end{align*}
Therefore, for small $\lambda$, both contracts lead to zero effort on the drift and on the volatility. To find a case where our contracts induce less effort on the volatility than classical ones, we need to increase drastically $\lambda$ for example. Figure \ref{fig:delta_pos} illustrates this particular case, with $\delta = 5$, $\lambda = 2.8$ and when $50\%$ of the variance is explained by the common noise ($\sigma^\circ = \sigma = 0.085/\sqrt{2}$).

\begin{figure}[!ht]
    \begin{center}
        \includegraphics[height=8cm]{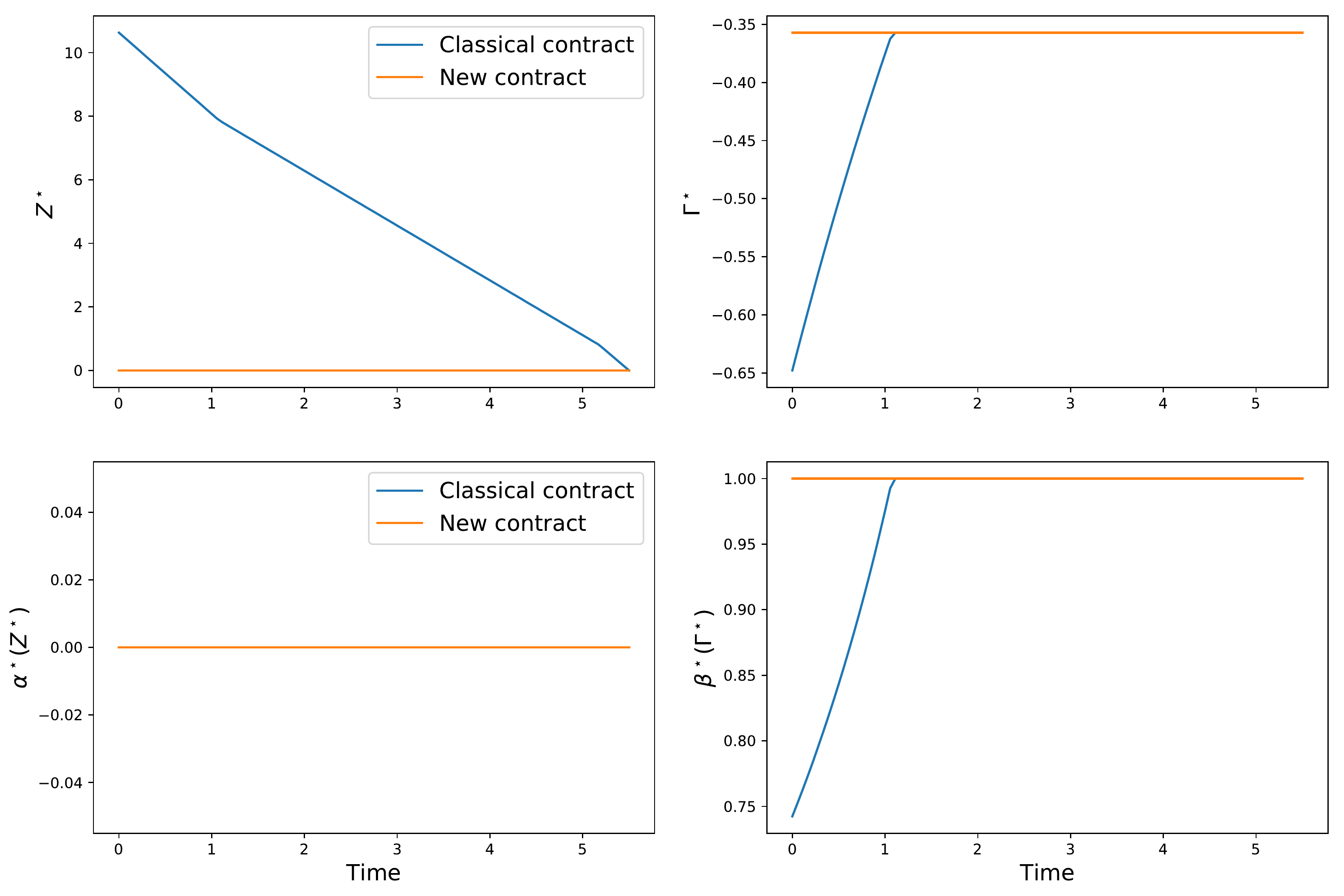}
        \caption{Comparison of payment rates and efforts in the linear EVD case with $\delta \geq 0$. \\
        Parameters: $\delta=5$, $\lambda=2.8$ and $\sigma^\circ = \sigma = 0.085/\sqrt{2}$.}
        \label{fig:delta_pos}
    \end{center}
\end{figure} 

\medskip
    
Although this result does not seem intuitive, the explanation is as follows. When $\delta$ is positive, a reduction of the consumption strongly decreases the consumers' utility in comparison to the marginal gain on the producer's utility. Therefore, it is too costly for the Principal to incentivise the consumers to make efforts, in particular on the drift. Thus, the Principal sets a non negative payment rate $Z$ (upper left graph). However, since she is risk--averse, she wants to share some risk with the consumers through the contract.  In the classical contracts case, the only way to share risk is through the infinitesimal payment $Z^{0,\star}_t \drm X_t$. This explain why $Z^{0,\star}$ is positive (upper left graph, blue curve), while, in the new contracts case, the Principal can share the risk through the indexation on others by $\overline Z^{\mu, \star}$, and can therefore set $Z^\star=0$ (upper left graph, orange curve). Moreover, when the Principal offers a positive payment $Z$, she needs to compensate by giving a negative payment $\Gamma$ (upper right graph, blue curve), in order to minimise the payment indexed on the quadratic variation $(\Gamma_t + R_A Z_t^2) \drm \langle X \rangle_t$. Otherwise, in the new contracts case, the Principal manages the risk through the payment rate $\overline Z^\mu$, and does not need to compensate with a small (negative) $\Gamma$ (upper right graph, orange curve).

\medskip
    
We may think that the previous results are disappointing, because consumers make less effort to reduce the volatility of their consumption (lower right graph). But as mentioned above, we insist on the fact that this case is not supposed to happen, since, according to calibration in \cite{aid2018optimal}, $\lambda=2.8 \times 10^{-2}$ and $\delta=-55.44 \leq 0$.

\medskip

\textbf{Comparison of utility.} We can prove analytically that the utility of the Principal is bigger in the case she can indexed contracts on others' deviation consumption. Indeed, we have
\begin{align*}
    m^{0,P} (t) = &\ \dfrac{\theta}{2} \big( \sigma^{\circ} \big)^2
    - \dfrac{1}{2}  \overline{\rho} \delta^2 (T-t)^2
    + \dfrac{1}{2}  \inf_{z \in \R} \overline h^P (t,z) \\
    \geq &\ \dfrac{\theta}{2} \big( \sigma^{\circ} \big)^2
    - \dfrac{1}{2}  \overline{\rho} \delta^2 (T-t)^2
    + \dfrac{1}{2}  \inf_{z \in \R} \overline h(t,z)
    + \dfrac{1}{2} \big(\sigma^{\circ}\big)^2 \inf_{z \in \R} \Big\{ R_A  z^2 + R_P ( - z + \delta (T-t) )^2 \Big\}.
\end{align*}

The second infimum is attained on $z = R_P \delta (T-t) / (R_A + R_P)$, which leads to
\begin{align*}
    m^{0,P} (t) \geq &\ \dfrac{\theta}{2} \big( \sigma^{\circ} \big)^2
    + \dfrac{1}{2} \Big( \big(\sigma^{\circ}\big)^2 \overline R - \overline{\rho} \Big) \delta^2 (T-t)^2
    + \dfrac{1}{2}  \inf_{z \in \R} \overline h(t,z)
    = m^P (t).
\end{align*}
Hence, for all $t \in [0,T]$ we have $m^{0,P} (t) \geq  m^P (t)$, and thus
% \begin{align*}
%     - \int_t^T m^{0,P} (s) \drm s \leq - \int_t^T m^P (s) \drm s,
% \end{align*}
% and we conclude that 
$V_0^{P} \geq V_0^{0,P}$. This result is very intuitive since in the case of new contracts, the Principal has more controls to maximise her utility (three instead of two).

\medskip

In order to assure convergence when $R_P$ tends to zero, we draw in Figure \ref{fig:utility_difference} the utility difference defined as follow
\begin{align*}
    \Delta V = \dfrac{1+V_0^P}{R_P} - \dfrac{1+V_0^{0,P}}{R_P} =  \dfrac{1}{R_P} \big( V_0^P - V_0^{0,P} \big),
\end{align*}
since 
\begin{align*}
    \lim_{R_P \rightarrow 0} \dfrac{1+V_0^P}{R_P} = \lim_{R_P \rightarrow 0} \dfrac{1- \erm^{ R_P ( \xi_0 - u^P ( 0, \mu^X_0 ) ) }}{R_P} = -  \xi_0 + u^0 ( 0, \mu^X_0 ) = V_0^0.
\end{align*}
We study the effect of the risk--aversion parameter $R_P$ on the producer's utility, and also the impact of the percentage of variance related to common noise. 

\medskip

\begin{figure}[!ht]
    \begin{center}
        \includegraphics[height=6cm]{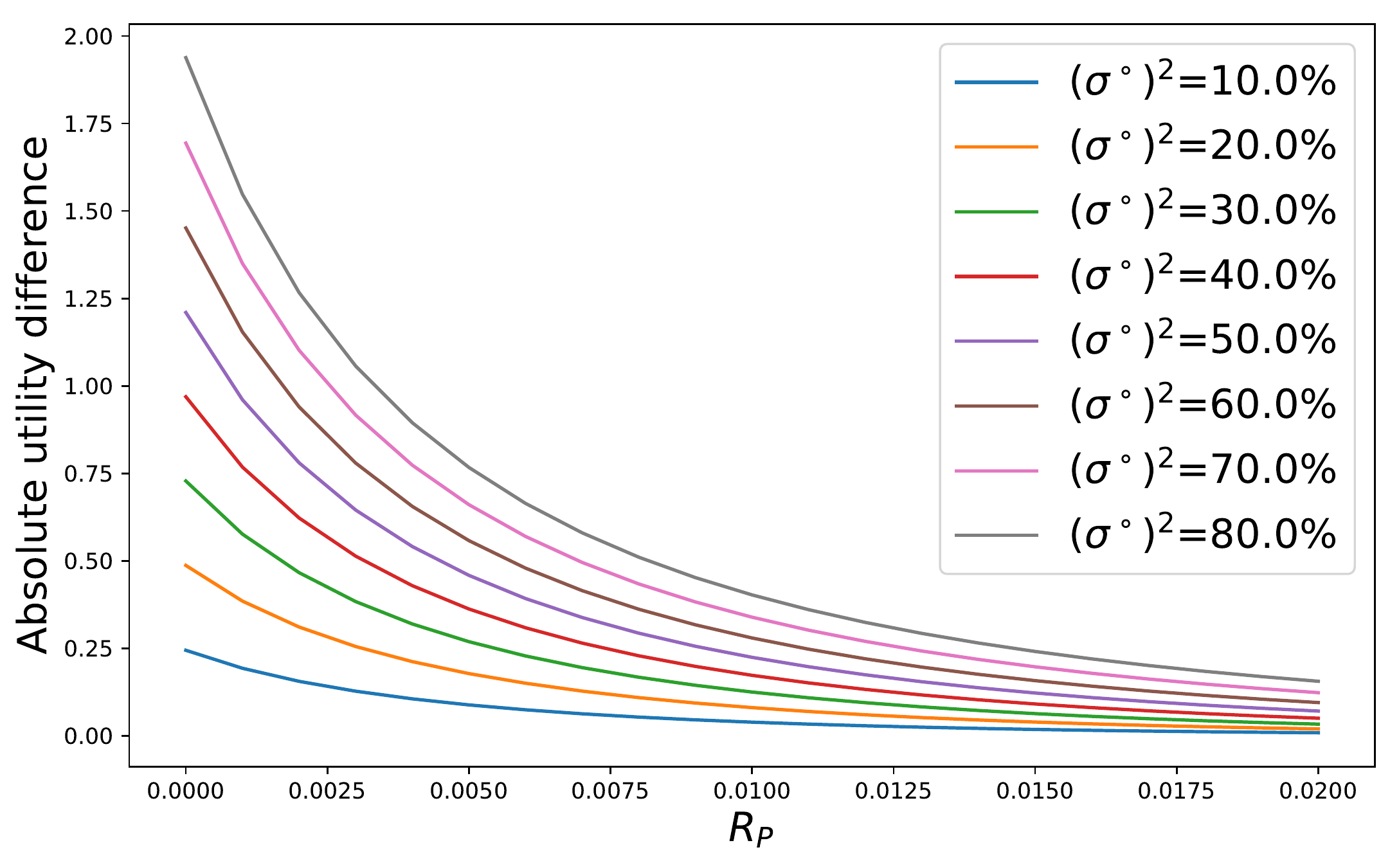}
        \caption{Absolute utility difference in the linear EVD case. \\ Variation with respect to the risk aversion parameter $R_P$ and the correlation with the common noise $\sigma^\circ$.}
        \label{fig:utility_difference}
    \end{center}
\end{figure}

First of all, notice that the utility difference is of order $1$, as the utility itself. Therefore, we have a significant gain by implementing this type of contracts, which is confirmed by Figure \ref{fig:relative_utility_difference}, representing the relative utility difference, computed as:
\begin{align*}
    \overline \Delta V = \dfrac{\dfrac{1+V_0^P}{R_P} - \dfrac{1+V_0^{0,P}}{R_P}}{\dfrac{1+V_0^{0,P}}{R_P}}
    = \dfrac{V_0^P - V_0^{0,P}}{1+V_0^{0,P}}.
\end{align*}

Moreover, it is clear that the more significant the correlation with the common noise is, the more important the utility difference is. This is an expected result, since our type of contracts allows the Principal to better choose the remaining risk she wants to bear, by indexing the contract on others. We have already noticed that, in absence of common noise ($\sigma^\circ = 0$), our contracts are reduced to classical contracts for drift and volatility control. 

\medskip

Figure \ref{fig:utility_difference} shows also that the gain in utility is decreasing with the risk--aversion of the Principal. Indeed, with classical contracts, she is forced to give to the consumers some risk through the linear payment $Z_t \drm X_t$. Since the consumers are risk--averse too, this payment must be compensated by a deterministic one. Our new contracts allows her to better choose the risk she wants to bear. On the one hand, if her risk--aversion is low, she can keep all the risk to herself. Thus, she does not need to compensate the risk with a deterministic payment, and therefore she make a significant gain. On the other hand, if she is risk--averse, she adds a random part to the contract, indexed on the common noise:
\begin{align*}
    \dfrac{R_P}{R_A + R_P} \sigma^\circ \delta \int_0^t (T-s) \drm W^\circ_s.
\end{align*}
But, since the consumers are risk--averse too, this random payment must be compensated by its quadratic variation:
\begin{align*}
    \dfrac{1}{2} \dfrac{R_A R^2_P}{(R_A + R_P)^2} \big( \sigma^{\circ} \big)^2 \delta^2 \int_0^t  \big( T-s \big)^2 \drm s,
\end{align*}
Therefore, the more risk--averse the Principal is, the more costly it is for her to spread the risk.

\begin{figure}[!ht]
    \begin{center}
        \includegraphics[height=6cm]{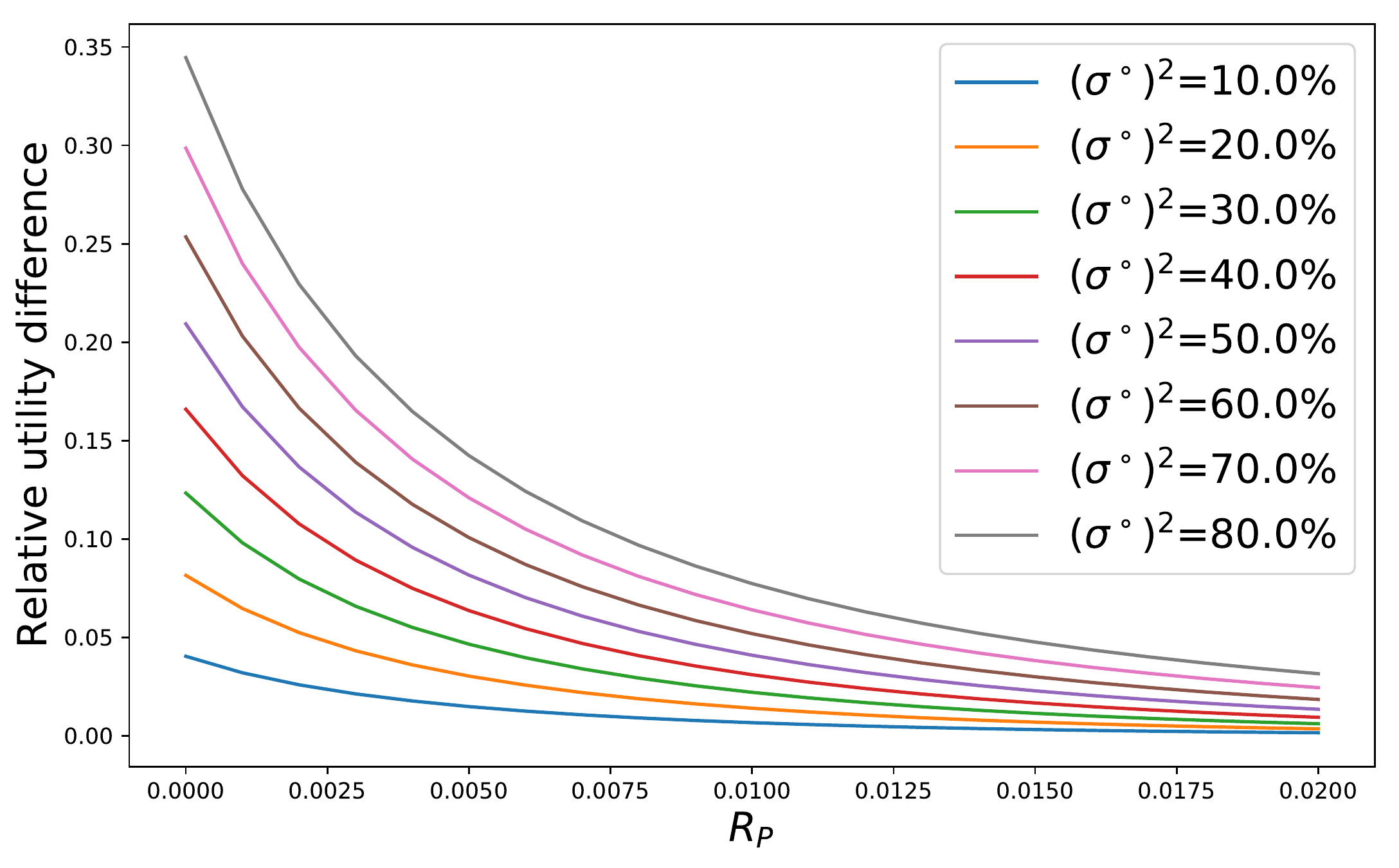}
        \caption{Relative utility difference in the linear EVD case. \\ Variation with respect to the risk aversion parameter $R_P$ and the correlation with the common noise $\sigma^\circ$.}
        \label{fig:relative_utility_difference}
    \end{center}
\end{figure}

\subsection{Risk--neutral Principal}\label{ss:RN_zeta0}

If the Principal is risk--neutral,
we expect a solution of the form 
\begin{align}\label{eq:def_u00}
    v^{0,0} (t, \mu^Y_t) = - \mathbb{E}^{\P_t} \Big[ L_t^{\zeta^0} \Big] + u^{0,0} (t, \mu_t^X).
\end{align}

It can be formally proved that the results of the previous section hold with $R_P =0$. In particular, $u^{0,0}$ is solution to the PDE \eqref{eq:pde_principal_RA_general_zeta0} with $R_P = 0$,
% \begin{align}\label{eq:pde_principal_RN_general_zeta0}
%     0 = &\ - \partial_t u^{0,0} ( t, \mu^X )
%     +  \int (g-f)(x) \mu^X (\drm x) 
%     + \dfrac{\theta}{2} \big( \sigma^{\circ} \big)^2 
%     - \dfrac{1}{2} \big( \sigma^{\circ} \big)^2 \overline{u}^{0,0}_{x, \mu^X}
%     - \dfrac{1}{2}  \overline{\rho} \big( \overline{u}^{0,0}_{\mu^X} \big)^2 \nonumber \\
%     &- \dfrac{1}{2} \big( \sigma^{\circ} \big)^2 \int \int \partial_{\mu^X}^2 u^{0,0} ( t, \mu^X ) ( x, \widetilde{x})  \mu^X (\drm x)  \mu^X (\drm \widetilde{x})
%     + \dfrac{1}{2}  \inf_{z \in \R} \Big\{ 
%     F_0 (q^{0,0} (z) )
%     + R_A \big(\sigma^{\circ}\big)^2 z^2 
%     + \overline{\rho} \big( \big( z^{-} \wedge A_{\textnormal{max}} \big) + \overline{u}^{0,0}_{\mu^X} \big)^2
%     \Big\},
% \end{align}
% with terminal condition $u^{0,P} ( T, \mu^X) = 0$.
and we can also state the equivalent of the proposition \ref{prop:Pb_Principal_RA_general_zeta0} with $R_P=0$ in this case.
% \begin{proposition}[Proposition \ref{prop:Pb_Principal_RA_general_zeta0} with $R_P=0$]\label{prop:Pb_Principal_RN_general_zeta0}
%     If there is a solution $u$ to the PDE \eqref{eq:pde_principal_RN_general_zeta0}, smooth enough in the sense of Definition \ref{def:C12}, and a couple $\zeta^{0,\star} \in \R^2$ satisfying
%     \begin{align*}
%         h^{0}(\overline{u}_{\mu^X}, \overline{u}_{x, \mu^X}, \zeta^{0,\star}) = \inf_{\zeta^{0} \in \R^2} h^{0}(\overline{u}_{\mu^X}, \overline{u}_{x, \mu^X}, \zeta^{0}),
%     \end{align*}
%     such that the following condition is satisfied
%     \begin{align}\label{eq:conditions_martingale_RN_zeta0}
%         \mathbb{E}^\P \bigg[ \int_0^T \big( \overline{u}_{\mu^X} + Z^\star_t \big)^2 \drm t \bigg] < + \infty,
%     \end{align}
%     then
%     \begin{enumerate}[label=(\roman*)]
%         \item $V^{0} = - \xi_0 + u ( 0, \mu^X_0 )$;
%         \item $\gamma^{\star} := - \max \big\{\theta- \overline{u}_{x,\mu^X}  + R_A (Z^{\star})^2, \frac{1}{\overline{\lambda}} \big\}$;
%         \item $z^{\star}$ is the minimiser of the minimisation problem in the HJB equation \eqref{eq:pde_principal_RN_general_zeta0}.
%     \end{enumerate}
% \end{proposition}
But, in order to obtain closed--form solution, we focus on the case of the linear energy value discrepancy.
\begin{proposition}[Proposition \ref{prop:Pb_Principal_RA_linear_simple} with $R_P=0$]\label{prop:Pb_Principal_RN_linear_simple}
    Let the energy value discrepancy be linear, i.e. $(f-g)(x) = \delta x$, $x \in \mathbb{R}$. Then
    \begin{enumerate}
        \item[$(i)$] the producer's value function is given by $V_0^{0,0} = - \xi_0 + u^{0,0} ( 0, \mu^X_0 ) $ where the certainty equivalent function $u^{0,0}$ is characterised by 
        \begin{align*}
            u^{0,0} (t, \mu_t^X) &= \delta (T-t) \int x \mu_t^X (\drm x) - \int_t^T m^{0,0}(s) \drm s, \; \text{ where } \;
            m^{0,0} (t) := \dfrac{\theta}{2} \big( \sigma^{\circ} \big)^2
            - \dfrac{1}{2}  \overline{\rho} \delta^2 (T-t)^2
            + \dfrac{1}{2}  \inf_{z \in \R} \overline h^0 (t,z),
        \end{align*}
        \item[$(ii)$] the optimal payment rate process $\zeta^{0,\star} = (Z^{0,\star}, \Gamma^{0,\star})$ is a deterministic function of time given by
        \begin{align*}
            Z_t^{0,\star} = \mathrm{Arg} \min_{z \in \mathbb{R}} \overline h^0 (t,z) \;
            \text{ and } \; \Gamma_t^{0,\star} = - \max \bigg\{\theta+ R_A (Z_t^{0,\star})^2, \dfrac{1}{\overline{\lambda}} \bigg\}.
        \end{align*}
    \end{enumerate}
\end{proposition}

Let us compare this Proposition with the Proposition~\ref{prop:Pb_Principal_linear}.

\medskip

\textbf{Comparison of efforts.} As previously, we need to distinguish between cases according to positive or negative $\delta$.

\medskip

$(i)$ On the one hand, in the meaningful case, when $\delta \leq 0$, the efforts of the consumers are higher when the Principal can indexed contracts on other's deviation. In fact, let us recall the optimal payment rates $Z$ in the two cases:
\begin{align*}
    Z_t^{0,\star} &= \mathrm{Arg} \min_{z \in \R} \Big\{ F_0 (h + R_A z^2)
    + R_A \big(\sigma^{\circ}\big)^2 z^2 
    + \overline{\rho} \big( (z^{-} \wedge A_{\textnormal{max}}) + \delta (T-t) \big)^2 
    \Big\} \leq 0, \\
    \text{ and } \; Z_t^{\star} &= \mathrm{Arg} \min_{z \in \R} \Big\{ F_0 (h + R_A z^2)
    + \overline{\rho} \big( (z^{-} \wedge A_{\textnormal{max}}) + \delta (T-t) \big)^2 
    \Big\} \leq 0.
\end{align*}

By definition of the minimum, we have:
\begin{align*}
    &\ F_0 \big(\theta+ R_A \big( Z_t^{0,\star} \big)^2 \big)
    + R_A \big(\sigma^{\circ}\big)^2  \big( Z_t^{0,\star} \big)^2
    + \overline{\rho} \big(  \big( \big( Z_t^{0,\star} \big)^{-} \wedge A_{\textnormal{max}} \big) + \delta (T-t) \big)^2 \\
    \leq &\ F_0 (\theta + R_A z^2)
    + R_A \big(\sigma^{\circ}\big)^2 z^2 
    + \overline{\rho} ( (z^{-}\wedge A_{\textnormal{max}}) + \delta (T-t) )^2,
\end{align*}
for all $z$, and in particular for $z=Z_t^{\star}$. In the same way,
\begin{align*}
    F_0 \big(\theta+ R_A \big( Z_t^{\star} \big)^2 \big)
    + \overline{\rho} \big( \big( Z_t^{\star} \big)^{-}\wedge A_{\textnormal{max}}) + \delta (T-t) \big)^2
    &\leq F_0 (h + R_A z^2)
    + \overline{\rho} ( (z^{-}\wedge A_{\textnormal{max}}) + \delta (T-t) )^2,
\end{align*}
for all $z$ and in particular for $z=Z_t^{0,\star}$. Hence, using the first and then the second inequality we have
\begin{align*}
    R_A \big(\sigma^{\circ}\big)^2 \Big( \big( Z_t^{0,\star} \big)^2 - \big( Z_t^{\star} \big)^2  \Big)
    \leq &\ F_0 \Big(\theta+ R_A \big( Z_t^{\star} \big)^2 \Big) - F_0 \Big(\theta+ R_A \big( Z_t^{0,\star} \big)^2 \Big) \\
    &+ \overline{\rho} \Big( \big( Z_t^{\star} \big)^{-}\wedge A_{\textnormal{max}}  + \delta (T-t) \Big)^2 - \overline{\rho} \Big(  \big( Z_t^{0,\star} \big)^{-} \wedge A_{\textnormal{max}} + \delta (T-t) \Big)^2 \leq 0.
\end{align*}

Hence $R_A \big(\sigma^{\circ}\big)^2 \big( Z_t^{0,\star} - Z_t^{\star} \big) \big( Z_t^{0,\star} + Z_t^{\star} \big) \leq 0,$ and since $Z_t^{0,\star}, \; Z_t^{\star} \leq 0$, we obtain $0 \geq Z_t^{0,\star} \geq Z_t^{\star}$. Therefore, in the case where the Principal can index contracts on the deviation consumption of others, the efforts of the consumer to reduce his deviation consumption in average is more important. Moreover, the inequality on the optimal payment rate $Z$ implies that $0 \geq \Gamma_t^{0,\star} \geq \Gamma^{\star}$: the effort on the volatility is also more important. These results are presented in Figure \ref{fig:R_P=0} for $(\sigma^\circ)^2 = 50\%$, \textit{i.e.} $50\%$ of the variance is explained by the common noise.

\begin{figure}[!ht]
    \begin{center}
        \includegraphics[height=5cm]{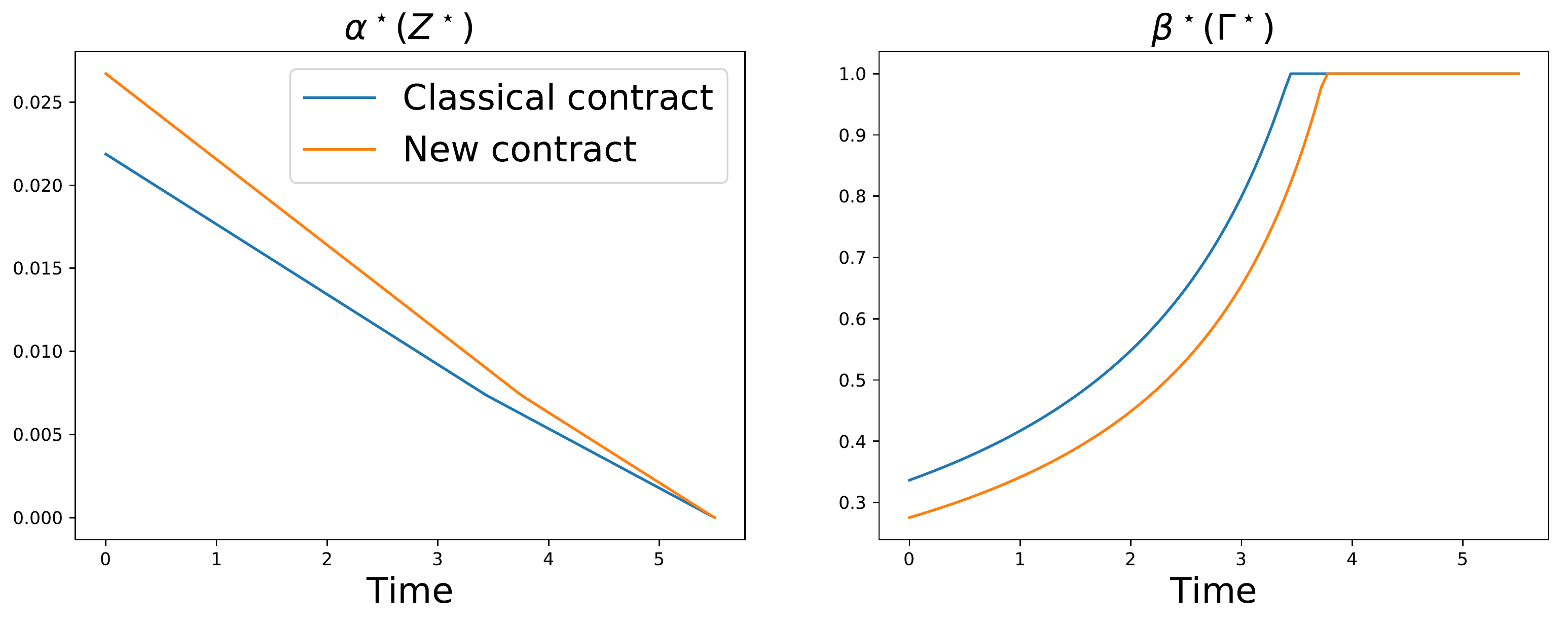}
        \caption{Comparison of efforts in the linear EVD case for a risk--neutral Principal.}
        \label{fig:R_P=0}
    \end{center}
\end{figure}  

\medskip

$(ii)$ On the other hand, if $\delta \geq 0$, we obtain $Z^{\star} = 0$ in both cases, which leads to zero effort from the consumers on their deviation consumption in average. Moreover, the optimal payment rate $\Gamma^{\star}$ is also the same in both cases and leads to the same effort on the volatility. Therefore, in this particular case, the choice of new or classical contracts does not affect the consumers efforts.

\bigskip

\textbf{Comparison of utility.} The utility of a risk--neutral Principal is higher when she can index contracts on others' deviation consumption. Indeed,
\begin{align*}
    m^{0,0} (t) = &\ \dfrac{\theta}{2} \big( \sigma^{\circ} \big)^2
    - \dfrac{1}{2}  \overline{\rho} \delta^2 (T-t)^2
    + \dfrac{1}{2}  \inf_{z \in \R} \Big\{ 
    F_0 (h + R_A z^2)
    + R_A \big(\sigma^{\circ}\big)^2 z^2 
    + \overline{\rho} \big( (z^{-}\wedge A_{\textnormal{max}}) + \delta (T-t) \big)^2 \Big\} \\
    \geq &\ \dfrac{\theta}{2} \big( \sigma^{\circ} \big)^2
    - \dfrac{1}{2}  \overline{\rho} \delta^2 (T-t)^2
    + \dfrac{1}{2}  \inf_{z \in \R} \Big\{ 
    F_0 (h + R_A z^2)
    + \overline{\rho} \big( (z^{-}\wedge A_{\textnormal{max}}) + \delta (T-t) \big)^2
    \Big\} = m^{0} (t).
\end{align*}

Hence, for all $t \in [0,T]$ we have $m^{0,0} (t) \geq  m^{0} (t)$, which leads to
\begin{align*}
    - \int_t^T m^{0,0} (s) \drm s \leq - \int_t^T m^{0} (s) \drm s,
\end{align*}
and we conclude that $V_0^{0} \geq V_0^{0,0}$. This result is shown in Figure  \ref{fig:utility_difference_RN}. The absolute utility difference and the relative utility difference are increasing with respect to the correlation with the common noise, which is consistent with the intuition that the contract allows the Principal to manage the remaining risk. The more important this risk is, the more gain in utility the Principal receives. 

\begin{figure}[!ht]
    \begin{center}
        \includegraphics[height=5cm]{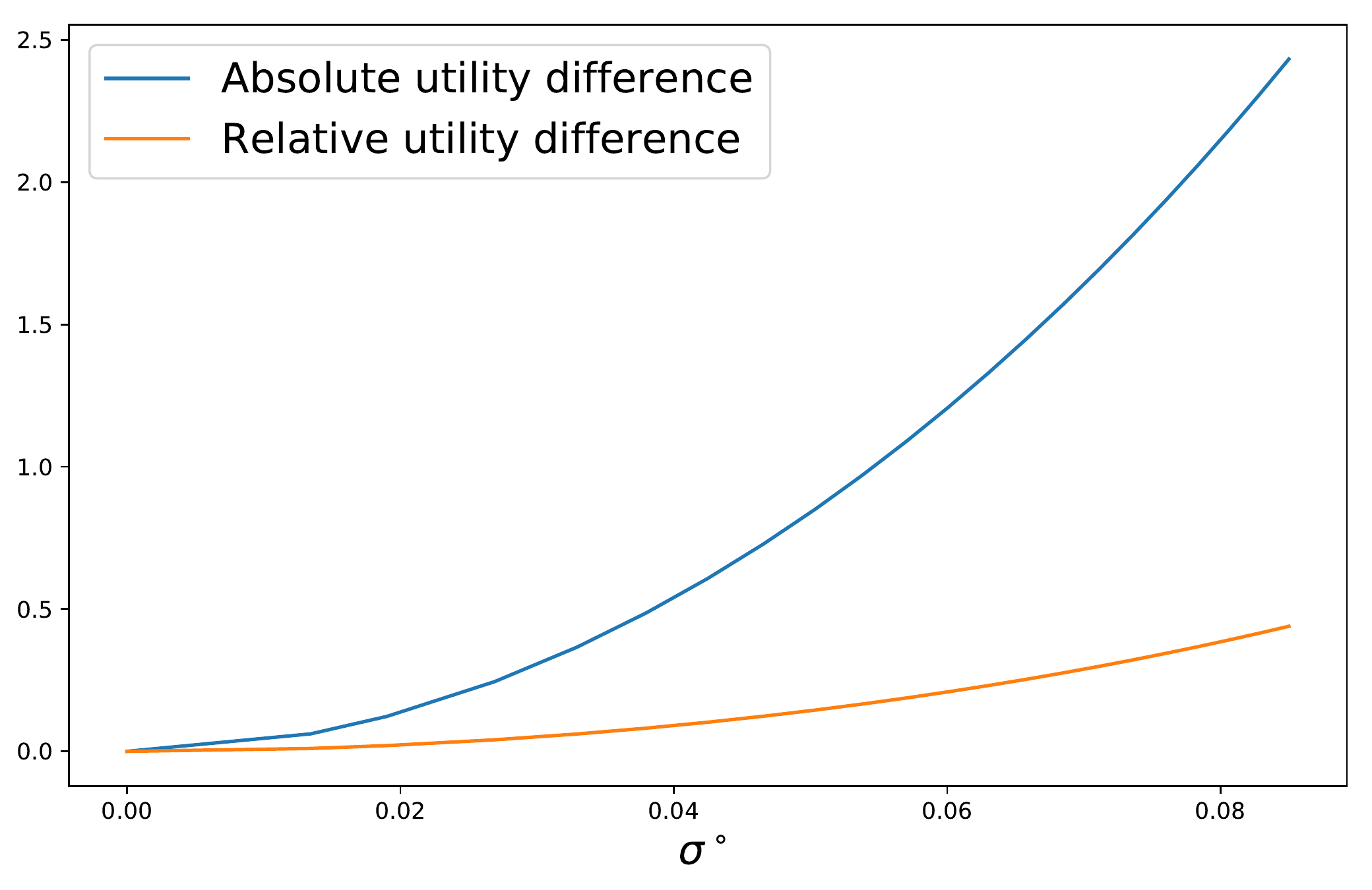}
        \caption{Absolute and relative utility difference in the linear EVD case for a risk--neutral Principal. \\ Variation with respect to the correlation with the common noise $\sigma^\circ$.}
        \label{fig:utility_difference_RN}
    \end{center}
\end{figure}

\medskip

To conclude this section, considering linear EVD, there is a net benefice from implementing contracts with indexation on others' deviation consumption. In addition to the substantial gain in utility for the Principal, this type of contracts induces, in general, more efforts of the consumers to reduce their consumption in average and with less volatility. The best results are obtained for a risk--neutral producer. In this case, if the variance of the deviation is only explained by the common noise, the utility with new contracts is up to $1.5$ the utility with classical contracts (see Figure \ref{fig:utility_difference_RN}). Moreover, Figure \ref{fig:diff_effort} (left) shows that consumers reduce their consumption by $1.5$ times more on average. The best results on volatility is when the volatility is half explained by the common noise, \textit{i.e.} for $\sigma^\circ= 0.085 / \sqrt{2}$. The consumers thus increase their efforts on the volatility by almost $4\%$.

\section{Extension and first--best}
\subsection{Contractible common noise}\label{sec:observable_common_noise}

Throughout this paper, we have studied optimal contracting in the case where the Principal can not offer a contract directly indexed on the common noise. 
%This can either mean that the Principal cannot ovserve it perfectly, or that there are regulatory reasons preventing the producer from using it directly in the contract.
The goal of this section is to define an optimal form of contract in the case of a contractible common noise, and compare the results to those obtained in the previous case. In Section \ref{sec:consumer_problem}, we were looking at the optimal form of contracts for the representative consumer in the case the Principal can only index the contract on his deviation consumption and the conditional law of others. The Remark \ref{rk:contract_w0_observable} leads us to also study the class of contracts one can obtain if the Principal is able to index the contract on it. From \eqref{contract:dependency_w0}, we recall that, in this case, the Principal is offering a contract $\F^{\textnormal{obs},\circ}$--measurable. 

\medskip

Therefore, we expect that the value function of the consumer now depends on three state variables: $X$, $\widehat \mu$ and $W^{\circ}$. We consider the dynamic version of the value function of the representative consumer, $V_t^A$, which may be written as
\begin{align*}
    V^{A}_t = v^A \big(t, X_{t \wedge \cdot}, W^{\circ}_{t \wedge \cdot}, \widehat \mu_{t} \big).
\end{align*} 

If the function $v^A : [0,T] \times \Cc([0,T], \R^2) \times \Pc (\Cc_T)$ is smooth enough in the sense defined in \cite[Section 4.3.4]{carmona2018probabilisticII}, we can apply the Chain Rule with Common Noise under $\Cc^{1,2,2}$--regularity, defined in \cite[Theorem 4.17]{carmona2018probabilisticII} to $v^A$.
Following the same reasoning as the one developed in the Subsection \ref{ss:contract_zmu}, and using the computations in Remark \ref{rk:contract_w0_observable}, we obtain the following form of contract, 
indexed by the triple $\zeta^\circ = (Z, Z^\circ, \Gamma)$,
\begin{align}\label{eq:contractform_w0}
    \xi_t = &\ \xi_0  - \int_0^t \Hc^\circ (X_s, \zeta^\circ_s)  \drm s 
    + \int_0^t Z_s \drm X_s 
    + \dfrac{1}{2} \int_0^t \big( \Gamma_s + R_A Z_s^2 \big) \drm \langle X \rangle_s
    + \sigma^{\circ} \int_0^t Z^{\circ}_s \drm W^{\circ}_s 
    + \dfrac{1}{2} R_A \big( \sigma^{\circ} \big)^2 \int_0^t Z^{\circ}_s \big( 2 Z_s + Z^{\circ}_s \big)^2 \drm s.
\end{align}

Therefore, if the Principal observes the common noise, it is equivalent to index the contract on the law of others or on the common noise. Indeed, the form of contract we obtain here is the same as the one in Remark \ref{rk:contract_w0_observable}. This contracts leads obviously to the same effort of the consumers and to a unique mean--field equilibrium. If we compare both cases where the Principal observes or not the common noise, the only difference is the measurability of the contract's parameters. In fact, the optimal form of contract is the same whether the Principal observes or not the common noise. Of course, the form of the contract is more explicit in this case, and we don't need the indexation on the law of others, but the only change is that the control $\zeta^\circ$ of the Principal here is predictable with respect to a bigger filtration $\F^{\textnormal{obs},\circ}$. 

\medskip

To sum up, if the Principal does not observe the common noise, $\F^{\text{obs}}$ is the natural filtration generated by $X$ and $\widehat \mu$, thus $\zeta$ is $\F^{\text{obs}}$--measurable and the optimal form of contract is given by \eqref{contract_form_mu}. If the Principal observes the common noise, $\F^{\text{obs}, \circ}$ is the natural filtration generated by $X$, $\widehat \mu$ and $W^\circ$, thus $\zeta$ is $\F^{\text{obs}, \circ}$--measurable and the optimal form of contract is given by \eqref{eq:contractform_w0}. In this case, indexing the contract on $\widehat \mu$ or $W^\circ$ is equivalent. In both cases, the contract form is the same, but we are allowed or not to write it in term of the common noise.

\medskip
If the Principal observes the common noise, she might have an interest to offer contracts indexed on it. However, intuitively, in the linear energy discrepancy case, since the optimal compensation rates for the contract are deterministic, the Principal should not gain from indexing the contract on the common noise. The optimisation problem of the Principal is the same as before, except that the supremum is taken over contracts $\xi \in \Xi^\circ$. 
In fact, following the same reasoning as in Section \ref{sec:principal_problem}, Theorem \ref{thm:Pb_Principal_general} and the propositions \ref{prop:Pb_Principal_RA_general}, \ref{prop:Pb_Principal_RN_general} and \ref{prop:Pb_Principal_linear} can be established in this case. Therefore, there is no benefice for indexing the contract on the common noise, if there is already an indexation on the law of others. This means that our contracts indexed on the law of others allow the Principal to index the compensation on the common noise in a hidden way.

\subsection{First--best problem}\label{sec:firstbest}

In this section, we concentrate our attention to the so-called first-best framework, where there is no moral hazard and the Principal can actually choose directly both the contract $\xi$ as well as the actions of the Agents. 

\medskip

Given the reservation utility level of the representative Agent, $R_0$, the problem of the Principal is
\begin{align*}
    V_0^{\textnormal{FB}} := \inf_{\rho > 0} \Big\{ - \rho R_0 + \sup_{(\P, \mu^X) \in \Pc \times \Pc (\Cc_T)} \sup_{\xi \in \Xi^{\textnormal{FB}}} \big\{ J_0^P (\xi, \P) + \rho J_0^A (\xi, \mu^X, \P) \big\} \Big\},
\end{align*}
where $\rho > 0$ is the Lagrange multiplier associated to the participation constraint, and $\Xi^{FB}$ is defined by \eqref{def:Xi_FB}. Recall that the representative Agent is risk--averse, with a risk--aversion parameter $R_A$. We can consider both the cases of a risk--averse or risk--neutral Principal. 
% To lighten the notations, we denote:
% \begin{align*}
%     K^P_T :=  \int_0^T g ( X_s ) \mathrm{d} s + \dfrac{\theta}{2} \int_0^T \drm \langle X \rangle_s, \; K_T^{A,\P} := \int_0^T \big( c \big( \nu^\P_s \big) - f \big( X_s \big) \big) \mathrm{d} s \; \text{ and } \; K_T := K_T^A + K_T^P.
% \end{align*}
We give the results below, then we provide the proofs in Appendix \ref{proof:firstbest}.

\subsubsection{Principal with CARA utility}

If the utility of the Principal is defined as $U^P(x) = - \erm^{-R_P x}$,
% \todo[inline]{D: you can shorten everything after that: we have similar results in the paper before, wo we can skip most of it and get directly to the results.}
using the same tools as in Section \ref{sec:principal_problem}, we obtain the following proposition.

\begin{proposition}\label{prop:Pb_principal_FB_RA}
    In the first--best case, for a risk--averse producer with a CARA utility function
    
    \medskip
    $(i)$ the utility of the producer is given by 
    \begin{align*}
        V_0^{\textnormal{FB}} &= R_0 \bigg( \dfrac{V_0^{\overline R}}{R_0} \bigg)^{1 + \frac{R_P}{R_A}}, \; \text{ where } \; V_0^{\overline R} = - \erm^{ -\overline R \; u^{FB} (0, \mu_0^X)}
    \end{align*}
    and $u^{FB}$ solves the following {\rm HJB} equation
    \begin{align*}
        \partial_t u^{FB}&=
        + \int (g-f) (x) \mu^X (\drm x)
        + \dfrac{\theta}{2} \big( \sigma^\circ \big)^2 
        - \dfrac{1}{2} \big( \sigma^{\circ} \big)^2 \overline u^{FB}_{x, \mu^X}
        + \dfrac{1}{2} \big( \sigma^{\circ} \big)^2 \overline R \; \big( \overline u^{FB}_{\mu^X} \big)^2
        - \dfrac{1}{2} \big( \sigma^{\circ} \big)^2
        \overline u^{FB}_{\mu^X \mu^X} \\
        &+ \overline{\rho} \overline u^{FB}_{\mu^X} \big( \big( \overline u^{FB}_{\mu^X} \big)^{-} \wedge A_{\textnormal{max}} \big) 
        + \dfrac{1}{2} c_{\alpha}^{\star} \big( \overline u^{FB}_{\mu^X} \big)
        - \dfrac{1}{2} \big( \overline u^{FB}_{x, \mu^X} -\theta \big) \Sigma^{\star} \big( \overline u^{FB}_{x, \mu^X} -\theta \big) 
        + \dfrac{1}{2} c_{\beta}^{\star} \big( \overline u^{FB}_{x, \mu^X} -\theta \big),\\
 u(T,\cdot)&=-1;
    \end{align*}
    
    $(ii)$ the optimal effort is the process $\nu^{FB,\star} := (\alpha^{FB,\star}, \beta^{FB,\star})$ defined for all $t \in [0,T]$ by $\alpha^{FB,\star}_t := a^{FB,\star} \big(t, \mu^X_t \big)$ and $\beta^{FB,\star}_t := b^{FB,\star} \big(t, \mu^X_t \big)$ where, for $k=1, \dots, d$,
    \begin{align*}
        a^{k, FB,\star} (t,\mu) := \rho^k \big( \big( \overline{u}^{FB}_{\mu^X} (t,\mu) \big)^{-} \wedge A_{\textnormal{max}} \big) \; \text{ and } \; b^{k,FB, \star} := 1 \wedge \big( \lambda^k \big( \overline{u}^{FB}_{x, \mu^X} (t, \mu) -\theta \big)^{-} \big)^{\frac{-1}{\eta^k + 1}} \vee B_{\textnormal{min}};
    \end{align*}
    
    $(iii)$ the optimal contract is given by
    \begin{align*}
        \xi^\star = 
        - \dfrac{1}{R_A} \ln (-R_0)
        + \int_0^T \big( c \big( \nu_s^{FB,\star} \big) - f(X_s) \big) \drm s;
    \end{align*}
    
    $(iv)$ in the linear energy discrepancy case, $u^{FB}$ is of the form
    \begin{align*}
        u^{FB} \big(t, \mu_t^X \big) &= \delta (T-t) \int x \mu_t^X (\drm x) - \int_t^T \overline m^P (s) \drm s,
    \end{align*}
    where 
    \begin{align*}
        \overline m^P (t) = &\ 
        \dfrac{\theta}{2} \big( \sigma^\circ \big)^2 
        + \dfrac12 \Big( \big( \sigma^{\circ} \big)^2 \overline R - \overline \rho \Big) \delta^2 (T-t)^2
        + \dfrac{\overline \rho}{2} \big( \delta^- (T-t) \wedge A_{\textnormal{max}} + \delta (T-t) \big)^2
        + \dfrac{\theta}{2} \Sigma^{\star} \big( -\theta \big) 
        + \dfrac12 c_{\beta}^{\star} \big( -\theta \big),
    \end{align*}
    and the functions defining optimal efforts are deterministic functions of time:
    \begin{align*}
        a^{k,FB, \star} (t) := \rho^k \big( \big( \delta (T-t) \big)^{-} \wedge A_{\textnormal{max}} \big) \; \text{ and } \; b^{k,FB, \star} (t) := 1 \wedge \big( \lambda^k\theta\big)^{\frac{-1}{\eta^k + 1}} \vee B_{\textnormal{min}}, \; \text{ for } k=1, \dots, d.
    \end{align*}
\end{proposition}

\subsubsection{Risk--neutral Principal}

The risk--neutral Principal problem is very similar to the risk--averse one, informally the following Proposition is obtained by setting $R_P = 0$ (and thus $\overline R = 0$) in Proposition \ref{prop:Pb_principal_FB_RA}
\begin{proposition}\label{prop:Pb_principal_FB_RN}
    In the first--best case, for a risk--neutral producer,
    
    \medskip
    $(i)$ her utility is given by 
    \begin{align*}
        V_0^{\textnormal{FB}} &= \dfrac{1}{R_A} \ln ( - R_0 ) + V^0 (0, \mu_0^X),
    \end{align*}
    where $V^0$ is solution to the following {\rm HJB} equation
    \begin{align*}
        0 = &\ \partial_t V^0
        + \int (f-g) (x) \mu^X (\drm x)
        - \dfrac{\theta}{2} \big( \sigma^\circ \big)^2 
        + \dfrac{1}{2} \big( \sigma^{\circ} \big)^2 \overline{V}^0_{x, \mu^X}
        + \dfrac{1}{2} \big( \sigma^{\circ} \big)^2 \overline{V}^0_{\mu^X, \mu^X} \\
        &- \overline{\rho} \overline{V}^0_{\mu^X} \big( \big( \overline{V}^0_{\mu^X} \big)^{-} \wedge A_{\textnormal{max}} \big) - \dfrac{1}{2} c_{\alpha}^{\star} \big( \overline{V}^0_{\mu^X} \big)
        + \dfrac{1}{2} \big( \overline{V}^0_{x, \mu^X} -\theta \big) \Sigma^{\star} \big( \overline{V}^0_{x, \mu^X} -\theta \big) - \dfrac{1}{2} c_{\beta}^{\star} \big( \overline{V}^0_{x, \mu^X} -\theta \big);
    \end{align*}
    
    $(ii)$ in the linear energy discrepancy case, $V^0$ is of the form
    \begin{align*}
        V^0 (t, \mu_t^X) &= \delta (T-t) \int x \mu_t^X (\drm x) - \int_t^T \overline m^0 (s) \drm s,
    \end{align*}
    where $\overline m^0$ is equal to $\overline m^P$ for $\overline R = 0$,
        % $\overline m^0 (t) =  \dfrac{\theta}{2} \big( \sigma^\circ \big)^2 
        % + \overline{\rho} \delta (T-t) \big( \big ( \delta (T-t) \big)^{-} \wedge A_{\textnormal{max}} \big) 
        % + \dfrac{1}{2} \overline \rho \big( \big( \delta (T-t) \big)^{-} \wedge A_{\textnormal{max}} \big)^2
        % + \dfrac{\theta}{2} \Sigma^{\star} \big( -\theta \big) 
        % + \dfrac{1}{2} c_{\beta}^{\star} \big( -\theta \big), $
    and the functions defining optimal efforts are the same deterministic functions of time as in the risk--averse case.
\end{proposition}

\section{Conclusion}\label{sec:conclusion}

We extend in this paper the problem of demand response contracts in electricity markets set in \cite{aid2018optimal} by considering a continuum of consumers with mean--field interaction, whose consumption is impacted by a common noise. We proved that the producer can benefit from considering the continuum of consumers with mean--field interaction by indexing contracts on the consumption of one Agent and the law of others. This new type of contracts allows the Principal to reward a consumer who makes more effort than the others, or to penalise him if he makes less effort. At least in the linear energy value discrepancy case, the producer's utility is increased by the use of new contracts, and in most cases these contracts induce more efforts from the consumers to reduce the average level of their consumption and with less volatility. 

\medskip

The optimal contract is indexed in a hidden way on the consumption adjusted for climate hazards. This allows the Principal to offer a compensation indexed on the process which is really controlled by the consumer, to encourage him for making effort on the drift and the volatility of his deviation. Moreover, if the Principal is risk-averse, she can add to this contract a part indexed on others, which is in fact an indexation on the common noise, to better choose the remaining risk she wants to bear. 

\medskip

In the case where the principal is authorised to index the contract directly to the common noise, we obtain the same form of contracts. Therefore, contracting on the conditional law of others or on common noise is strictly equivalent. Nevertheless, if the Principal could not observe the common noise, or if there exist some regulatory rules preventing her from using the common noise directly in the contract, indexing the contract on others is a way to overcome this.

\medskip

Our approach provides better management of the risk associated with common noise. Thus, the greater the variance explained by the common noise, the more significant the results are. Therefore, these new contracts could improve demand response during periods where consumption is strongly affected by weather conditions, for example in winter, when the risk of electricity blackouts is high, and thus demand response is more than needed. Naturally, in the absence of common noise, our contracts are reduced to classical contracts for drift and volatility control.

\begin{appendices}

\section{Technical proofs}\label{sec:technical_proofs}

\subsection{Proof of Lemma \ref{lemma:rep}}\label{ss:proof_lemma}

Let $\P \in \Pc$. First of all, by definition of $\Pc$ we have $\P [ \Lambda \in \U_0 ] =1$, thus $\Lambda(\mathrm{d}s,\mathrm{d} v) = \delta_{\nu^\P_s}(\mathrm{d} v) \mathrm{d}s$ $\P$--a.s. for some $\F$--predictable control process $\nu^\P := \big( \alpha^\P, \beta^\P \big)$. Therefore, $(X, W, W^\circ)$ is an It\=o process with drift $A(\nu^\P)$ and quadratic variation $B\big(\nu_t^\P \big) B^\top \big( \nu_t^\P \big)$ under $\P$, where
\begin{align*}
    B\big(\nu_t^\P \big) B^\top \big( \nu_t^\P \big) =
    \begin{pmatrix}
        \Sigma(\beta_t^\P) + \big(\sigma^\circ\big)^2 & \sigma^\top(\beta_t^\P) & \sigma^{\circ} \\
        \sigma (\beta_t^\P) & \mathrm{I}_d & \mathbf{0}_d \\
        \sigma^\circ & \mathbf{0}_d^\top & 1
    \end{pmatrix} 
    = 
    \begin{pmatrix}
        0 & \sigma^\top (\beta_t^\P) & \sigma^{\circ} \\
        \mathbf{0}_d & \mathrm{I}_d & \mathbf{0}_d \\
        0 & \mathbf{0}_d^\top & 1
    \end{pmatrix}
    \times 
    \begin{pmatrix}
        0 & \mathbf{0}_d^\top & 0 \\
        \sigma(\beta_t^\P) & \mathrm{I}_d & \mathbf{0}^\top_d \\
        \sigma^{\circ} & \mathbf{0}_d^\top & 1,
    \end{pmatrix},\; \drm t \otimes \P (\drm \omega)-\mathrm{a.e.}
\end{align*}

Furthermore, following the line of \citet*{lin2018second}, we consider the extended space $\Omega^e = \Omega \times \Omega'$ where $\Omega' = \Cc ( [0,T], \R^{d+2})$. $\Omega'$ is equipped with the filtration $(\Fc'_t)_{t \geq 0}$, generated by the canonical process, and $\P_0'$ is the Wiener measure on $\Omega'$. We define $\Fc^e_t := \Fc_t \otimes \Fc_t'$, $\F^e := \F \otimes \F'$ and $\P^e := \P \otimes \P_0'$. We denote $X^e$, $W^e$ and $W^{\circ,e}$ the natural extensions of $W$ and $W^\circ$ from $\Omega$ to $\Omega^e$. By \citet*[Theorem 4.5.2]{stroock1997multidimensional}, there is a $d+2$--dimensional Brownian motion $B^e$ on $(\Omega^e, \F^e, \P^e)$, such that
\begin{align*}
    \drm
    \begin{pmatrix}
         X^e_t \\
         W^e_t \\
         W^{\circ,e}_t
    \end{pmatrix} 
    = 
    \begin{pmatrix}
        - \alpha_t^\P \cdot \mathbf{1}_d \\
        \mathbf{0}_d \\
        0
    \end{pmatrix} \drm t
    + 
    \begin{pmatrix}
        0 & \sigma^\top( \beta_t^{\P} ) & \sigma^{\circ} \\
        \mathbf{0}_d & \mathrm{I}_d & \mathbf{0}_d \\
        0 & \mathbf{0}_d^\top & 1
    \end{pmatrix}
    \drm B^e_t.
\end{align*}

Therefore, we have $\drm W^e_t = \drm \big( B_t^{e,2}, \dots, B_t^{e,d+1} \big)^\top$ and $\drm W^{\circ,e}_t = \drm B_t^{e,d+2}$. Then, 
\begin{align*}
    X^e_t = - \int_0^t \alpha_s^\P \cdot \mathbf{1}_d \drm s + \int_0^t \big( 0, \sigma^\top( \beta_s^{\P} ), \sigma^{\circ} \big) \drm B^e_s = - \int_0^t \alpha_s^\P \cdot \mathbf{1}_d \drm s + \int_0^t \sigma ( \beta_s^{\P} ) \cdot \drm W^e_s + \int_0^t \sigma^{\circ} \drm W^{\circ,e}_s,
\end{align*}
for $t\geq 0$, $\P^e$--a.s., which implies the desired result.

\subsection{Another representation for the set of measures}\label{appendix_anotherrepres}
The general approach to moral hazard problems in \citet*{cvitanic2018dynamic} requires to distinguish between the efforts of the agent which give rise to absolutely continuous probability measures in $\Pc$, namely the ones for which only the drift changes, or for which the volatility control changes while keeping fixed the quadratic variation of $X$. The goal of this subsection is to provide the appropriate formulation in our setting. Let us start by introducing some notations.

\vspace{0.5em}
We let $\overline \Pc$ be the set of probability measures $\overline \P$ on $(\Omega, \Fc_T)$ such that
\begin{enumerate}[label=(\roman*)]
    \item  the canonical vector process $(X,W,W^\circ)^\top$ is an $(\F,\P)$--local martingale for which there exists an $\F$--predictable and $B$--valued process $\beta^{\overline \P}$ such that the $\overline \P$--quadratic variation of $(X,W,W^\circ)^\top$ is $\overline \P$--a.s. equal to
     \[
    \begin{pmatrix}
            \Sigma\big(\beta_s^{\overline \P}\big) + \big(\sigma^\circ\big)^2 & \sigma^\top\big( \beta_s^{\overline \P}\big) & \sigma^{\circ} \\
        \sigma \big( \beta_s^{\overline \P} \big) & \mathrm{I}_d & \mathbf{0}_d \\
        \sigma^\circ & \mathbf{0}_d^\top & 1
    \end{pmatrix},\; s\in[0,T];
    \]
 %   \item $\overline\P \circ (X_0)^{-1} = \varrho$, and there exists a measure $\iota$ on $\mathbb R^d\times\mathbb R$ such that $\overline \P \circ \big((W_0,W_0^{\circ})\big)^{-1} = \iota$ ;
    \item $\overline\P\big[\Lambda \in \U_0]=1$.
%    \item for $\overline\P-$a.e. $\omega \in \Omega$ and for every $t \in[0,T]$, we have
%    \begin{align*}
%        \mu_t(\omega)=\overline\P^{\omega}_t \circ (X_{t\wedge \cdot})^{-1}.
%    \end{align*}
%        \item $(W^{\circ},\mu)$ is $\overline \P-$independent of $W$.
\end{enumerate}
Arguing as in Lemma \ref{lemma:rep}, we know that for all $\overline\P\in\overline\Pc$
\begin{align*}
    X_t &= X_0  + \int_0^t  \sigma \big(\beta^{\overline \P}_s) \cdot \drm W_s + \int_0^t \sigma^{\circ} \drm W^{\circ}_s, ~t\in[0,T], \; \overline\P-\textnormal{a.s.}
\end{align*}

{\color{black} Notice that using classical results of \citet*{bichteler1981stochastic} (see \citet*[Proposition 6.6]{neufeld2014measurability} for a modern presentation), we can define a pathwise version of the quadratic variation of $\langle X\rangle$,
%and $\langle X,W\rangle$, both 
being $\F$--predictable, allowing us to define the following $\R_+$--valued process
\[
S_t:=\underset{n\rightarrow+\infty}{\mathrm{limsup}}\; n\big(\langle X\rangle_t-\langle X\rangle_{t-1/n}\big).
%,\; \Sigma_t:=\underset{n\rightarrow+\infty}{\mathrm{limsup}}\; n\big(\langle X,W\rangle_t-\langle X,W\rangle_{t-1/n}\big),\; t\in[0,T].
\]}

\begin{definition}\label{def:Pnu}
For any $A$--valued and $\F$--predictable process $\alpha$, any $\overline\P\in\overline\Pc$, and $\F$--predictable and $B$--valued process $\beta$\footnote{Strictly speaking, the process $\beta$ should be indexed by the measure $\overline \P$, but we chose to not do so in order to alleviate notations.} such that, $\overline\P$--a.s., $\Sigma\big(\beta^{\overline\P}\big)=\Sigma\big(\beta\big)$,
we define the equivalent measures $\overline \P^{\nu}$, for $\nu := (\alpha, \beta)$, by their Radon--Nykodym density on $\Fc_T$:
\[
\frac{\mathrm{d}\overline\P^{\nu}}{\mathrm{d}\overline\P}:=\exp\bigg(-\int_0^T\frac{\alpha_s\cdot\mathbf{1}_d}{\Sigma(\beta_s)}\sigma\big(\beta_s\big)\cdot\mathrm{d}W_s-\frac12\int_0^T(\alpha_s\cdot\mathbf{1}_d)^2\mathrm{d}s\bigg).
\]
\end{definition}
Notice that such a measure is well--defined since $A$ is a compact set, and the $B$--valued processes are automatically bounded and bounded away from $0$. It is then immediate to check that the $\Pc$ coincides exactly with the set of all probability measures of the form $\overline \P^{\nu}$, which satisfy in addition
\begin{enumerate}[label=(\roman*)]
     \item $\overline\P^{\nu} \circ (X_0)^{-1} = \varrho$, and there exists a measure $\iota$ on $\mathbb R^d\times\mathbb R$ such that $\overline \P^{\nu} \circ \big((W_0,W_0^{\circ})\big)^{-1} = \iota$;
    \item for $\overline\P$--a.e. $\omega \in \Omega$ and for every $t \in[0,T]$, we have $\mu_t(\omega)= \big( \overline\P^{\nu} \big)^{\omega}_t \circ (X_{t\wedge \cdot})^{-1}$;
    \item $(W^{\circ},\mu)$ is $\overline \P^{\nu}$--independent of $W$.
\end{enumerate}
For any $\overline \P\in\overline\Pc$, we denote by $\overline \Uc(\overline\P)$ the set of controls $\nu \in \Uc$ such that $\overline \P^{\nu}\in\Pc$.

\subsection{Best--reaction functions of the consumers and equilibrium}

In this Subsection, we wish to relate the best--reaction function of a consumer to a given contract $\xi\in\Xi$, and a given measure $\widehat \mu$ played by the mean--field of the other consumers, in other words $V_0^A(\xi,\widehat\mu)$, to an appropriate second--order BSDE.

\medskip
With this in mind, we consider the process $S^0$, taking values in the set of symmetric positive $2\times2$ matrices, defined by
\[
S^0_t:=\begin{pmatrix}
S_t + \big(\sigma^\circ \big)^2 & \sigma^\circ\\
\sigma^\circ & 1
\end{pmatrix},\; t\in[0,T],
\]
as well as the following norms, for any $p\geq 1$, 
\[
 \|Z\|^p_{\mathbb H^p}
 :=
 \sup_{\overline\P\in\overline\Pc} \E^{\overline\P}\bigg[\bigg(\int_0^T Z_t^\top S^0_tZ_t\mathrm{d}t\bigg)^{\frac p2}\bigg], \ \mbox{and}\ 
 \|Y\|_{\mathbb S^p}^p
 :=
 \sup_{\P\in\overline\Pc} \E^\P\bigg[\sup_{0\leq t\leq T}|Y_t|^p\bigg],
\]
defined for any $\mathbb F^{\rm obs}$--predictable, $\R^2$--valued process $Z$ and any $\mathbb F^{\rm obs}$--optional, $\R$--valued process $Y$ with c\`adl\`ag paths.  
 
\medskip
We consider the following so--called second--order BSDE (2BSDE for short)
\begin{equation}\label{eq:2bsde2}
 Y_t
 =
-\mathrm{e}^{-R_A \xi}+\int_t^TF(X_s,Y_s,Z^1_s,S_s)\mathrm{d}s
 -\int_t^T Z^1_s\mathrm{d}X_s-\sigma^\circ\int_t^T Z^2_s\mathrm{d}W^\circ_s+\int_t^T\mathrm{d} K_s,
 \end{equation}
The following definition recalls the notion of 2BSDE, and uses the additional notation 
\[
 \overline{\mathcal{P}}_t(\overline{\mathbb{P}},(\F^{\rm obs})^{+})
 :=
 \left\{\overline{\P}^\prime\in\overline{\mathcal{P}}: \overline\P[E]=\overline{\P}^\prime[E]~\mbox{for all}~E\in(\Fc^{\rm obs}_t)^{+}\right\}, \ \text{for any $(\overline\P,t)\in\overline\Pc\times[0,T]$}. 
\]

\begin{definition} \label{def:2BSDE}
We say that $(Y,Z, K)$ is a solution to the {\rm 2BSDE} \eqref{eq:2bsde2} if for some $k>1$
\begin{enumerate}[label=$(\roman*)$] 
    \item $Y$ is a c\`adl\`ag and an $ (\F^{\rm obs})^{\overline\Pc}_+$--optional process, and $\| Y \|_{\S^k}<+\infty$;
    \item $Z =(Z^1,Z^2)^\top$ is an $ (\F^{\rm obs})^{\overline\Pc}$--predictable $\R^2$--valued process with $\big\|Z\big\|_{\H^k}<+\infty$;
    \item $K$ is an $ (\F^{\rm obs})^{\overline\Pc}$--optional, c\`adl\`ag, non--decreasing, $K_0=0$, $\underset{\overline\P\in\overline\Pc}{\sup}\ \E^{\overline\P}[K_T^k]<+\infty$, and satisfies the minimality condition
\begin{equation}\label{minimality2}
 K_t
 =
 \underset{ \overline{\mathbb{P}}^\prime \in \overline{\mathcal{P}}_t(\overline{\mathbb{P}},\F^{+}) }{ {\rm essinf}^{\overline{\mathbb P}}}\mathbb{E}^{\overline{\mathbb P}^\prime}\Big[ K_T\Big|(\Fc^{\rm obs}_t)^{\overline\P+}\Big],
 ~0\leq t\leq T,
 \ \overline{\P}-a.s.\mbox{ for all}\ 
 \overline\P \in \overline{\mathcal P}.
 \end{equation}
\end{enumerate}
\end{definition}
The main result of this section relates the solution to the above 2BSDE to the best--reaction function of the Agent. 
\begin{proposition}\label{prop:genbestreac}
Fix $(\xi,\widehat\mu) \in \Xi \times \Pc(\Cc_T)$. Let $(Y,(Z^1,Z^2)^\top, K)$ be a solution to the {\rm 2BSDE} \eqref{eq:2bsde2}.
%with $k=p$
We have
\[
V_0^A(\xi,\widehat\mu)=\sup_{\overline\P\in\overline\Pc}\E^{\overline\P}[Y_0].
\]
Conversely, the $($dynamic$)$ value function $V_t^A(\xi,\widehat\mu)$ always provides the first component of a solution to \eqref{eq:2bsde2}. Moreover, any optimal effort $\nu^{\star}:=(\alpha^{\P^\star},\beta^{\P^\star})$, {\color{black}and the optimal measure $\overline \P^\star \in \overline \Pc$ must be such that 
 \[
 K=0,\; \big(\overline{\P}^\star\big)^{\nu^{\star}}-\mathrm{a.s.}, \; \big(\alpha^{\star}_t,\beta^{\star}_t\big)\in\underset{(a,b)\in A\times\Sigma^{-1}(S_t)}{\mathrm{argmax}}\; F(X_{t\wedge\cdot}, Y_t, Z^1_t,S_t),\; \big(\overline{\P}^\star\big)^{\nu^{\star}}-\mathrm{a.s.},
 \]
where $\big(\overline \P^\star \big)^{\nu^\star}$ is defined from $\overline \P^\star$ by {\rm Definition \ref{def:Pnu}}.}
\end{proposition}

\begin{proof}
The proof is classical and follows the lines of \citet*[Proposition 4.6]{cvitanic2018dynamic}. We thus only mention here why the required assumptions are satisfied. First of all, the map $F$ can readily be rewritten as a function $\widetilde F(x,y,(S^0)^{1/2}(z^1,z^2)^\top, S^0)$ as done by \citet*{possamai2018stochastic}, since the matrix $S^0$ is always invertible in our setting. Moreover, we can easily get rid off the linear term $-R_Af(x)y$ in $F$ by considering the 2BSDE satisfied by $\mathrm{e}^{-\int_0^tf(X_{s\wedge\cdot})\mathrm{d}s}Y_t$ instead. We therefore assume without loss of generality that $f$ is $0$ here. Then, $\widetilde F$ and the terminal condition $-\mathrm{e}^{-R_A \xi}$ satisfy the Lipschitz and integrability properties in \cite[Assumption 1.1 $(i)$ and $(ii)$]{possamai2018stochastic}, because controls are bounded and by definition of the set of contracts, recall \eqref{eq:integxi} (and that the densities from probabilities in $\Pc$ to probabilities in $\overline\Pc$ have moments of any order, uniformly on the measures).  \cite[Assumption 4.1]{possamai2018stochastic} is also automatically satisfied as $\widetilde F$ is $0$ for $y=z^1=z^2=0$.

\medskip
Next, \cite[Assumption 1.1 $(iii)$--$(v)$]{possamai2018stochastic} are also satisfied by the set of measures $\overline \Pc$, see for instance \citet*{nutz2013constructing}. Finally, the set $\overline \Pc$ is saturated in the sense of \cite[Definition 5.1]{possamai2018stochastic}, see 
\cite[Remark 5.1]{possamai2018stochastic}.
\end{proof}
With Proposition \ref{prop:genbestreac} in hand, we can now characterise mean--field equilibria thanks to a 2BSDE of mean--field type, reminiscent of the mean--field BSDE obtained in the setting of \cite{elie2019tale} where only the drift of $X$ was controlled.

\begin{theorem}\label{th:genbestreac}
The pair $(\P^\star,\mu^\star)$ belongs to $\Mc^\star(\xi)$ if and only if {\color{black} there exists $\overline \P^\star \in \overline \Pc$ and $\nu^{\star}:=(\alpha^{\star},\beta^{\star})$ such that $\P^\star= \big( \overline \P^\star \big)^{\nu^\star}$, where $\big( \overline \P^\star \big)^{\nu^\star}$ is defined from $\overline \P^\star$ in the sense of \textnormal{Definition \ref{def:Pnu}}, and
\begin{align*}
    K^\star=0,\; \big( \overline \P^\star \big)^{\nu^\star}-\mathrm{a.s.}, \; \big(\alpha^{\star}_t,\beta^{\star}_t\big)\in\underset{(a,b)\in A\times\Sigma^{-1}(S_t)}{\mathrm{arg} \max}\; F(X_{t\wedge\cdot}, Y_t^\star, Z^{1\star}_t,S_t),\; \big( \overline \P^\star \big)^{\nu^\star}-\text{a.s.},
\end{align*}
%  \[
%  K^\star=0,\; \big(\overline{\P}^\star\big)^{\nu^{\star}}-\mathrm{a.s.}, \; \big(\alpha^{\star}_t,\beta^{\star}_t\big)\in\underset{(a,b)\in A\times\Sigma^{-1}(S_t)}{\mathrm{argmax}}\; F(X_{t\wedge\cdot}, Y_t^\star, Z^{1\star}_t,S_t),\; \big(\overline{\P}^\star\big)^{\nu^{\star}}-\mathrm{a.s.},
%  \]
where $(Y^\star,(Z^{1\star},Z^{2\star})^\top, K^\star)$ is a solution to the \textnormal{mean--field 2BSDE}}
\begin{equation}\label{eq:2bsde2mf}
    Y_t^\star =
    -\mathrm{e}^{-R_A \xi}+\int_t^TF(X_s,Y_s^\star,Z^{1\star}_s,S_s)\mathrm{d}s
    -\int_t^T Z^{1\star}_s\mathrm{d}X_s
    -\sigma^\circ\int_t^T Z^{2\star}_s\mathrm{d}W^\circ_s
    +\int_t^T\mathrm{d} K^\star_s,
\end{equation}
which satisfies \textnormal{Definition \ref{def:2BSDE}} and the fixed--point constraint
\begin{align*}
    \mu_t(\omega)=( \P^\star)^{\omega}_t \circ (X_{t\wedge \cdot})^{-1}.
\end{align*}
\end{theorem} 
\begin{proof}
By Proposition \ref{prop:genbestreac}, we have a characterisation of the best--reaction function of the Agent to an arbitrary $\widehat\mu$. An equilibrium then necessitates only that $\widehat\mu$ coincides with the conditional distribution of $X$ under $\P^\star$, which is exactly what is given by the mean--field 2BSDE \eqref{eq:2bsde2mf}.
\end{proof}

We end this section with the following lemma which provides us with explicit integrability properties for the processes $Z$ and $Z^\mu$ associated to a contract in $\Xi_{\rm S}$, {\color{black} in the sense of Definition \ref{def:simple_contracts}}. This will prove useful for us when analysing the problem of the Principal.
\begin{lemma}\label{lemma:integZ}
For any $(Z,\overline Z^\mu,\Gamma)\in\overline\Vc$, there exists some $p^\prime\in(1,p)$ such that 
\[
\sup_{\P\in\Pc}\E^\P\bigg[\bigg(\int_0^T|Z_s|^2\mathrm{d}s\bigg)^{p^\prime/2}\bigg]+ \sup_{\P\in\Pc}\E^\P\bigg[\bigg(\int_0^T\big|\overline Z^\mu_s\big|^2\mathrm{d}s\bigg)^{ p^\prime/2}\bigg]<+\infty.
\]
\end{lemma}
\begin{proof}
First, we know by Theorem \ref{thm:mfe} and its proof that if we define
\[
Y_t:=-\mathrm{e}^{-R_A\xi_t^{\xi_0,\zeta}},\;  Z^1_t:=-R_AY_tZ_t,\;   Z^2_t:=-R_AY_t \overline Z^\mu_t,\;  \underline \Gamma_t:= -R_AY_t\Gamma_t,\; t\in[0,T],
\]
\[
 K_t:=\int_0^t\bigg(-R_AY_s \Hc^\circ(X_s,\zeta_s)-\frac12\underline\Gamma_s \Big( S_s + \big( \sigma^\circ \big)^2 \Big) - F(X_s, Y_s,  Z^1_s, S_s ) \bigg) \mathrm{d}s,\; t\in[0,T],
\]
then $(Y,(Z^1,Z^2)^\top,K)$ solves 2BSDE \ref{eq:2bsde2} and in particular that there is some $\bar p\in (1,p)$ such that $(Z^1,Z^2)^\top\in\H^{\bar p}$. Furthermore, we also have that 
\[
\bigg|\frac1{Y_t}\bigg|=\mathrm{e}^{R_A\xi_t^{\xi_0,\zeta}},
\]
so that we deduce using H\"older's inequality, the definition of $\Xi_{\rm S}$, and the fact that densities between measures in $\Pc$ and $\overline\Pc$ have moments of any order
\[
\sup_{\P\in\Pc}\E^\P\bigg[\sup_{0\leq t\leq T}\bigg|\frac1{Y_t}\bigg|^p\bigg]<+\infty.
\]
Then, we have for any $\tilde p\in(1,\bar p)$
\begin{align*}
\sup_{\P\in\Pc}\E^\P\bigg[\bigg(\int_0^T|Z_s|^2\mathrm{d}s\bigg)^{\tilde p/2}\bigg]&=\frac1{R_A^{\tilde p}}\sup_{\P\in\Pc}\E^\P\bigg[\bigg(\int_0^T\bigg|\frac{Z^1_s}{Y_s}\bigg|^2\mathrm{d}s\bigg)^{\tilde p/2}\bigg]\\
&\leq \frac1{R_A^{\tilde p}}\sup_{\P\in\Pc}\E^\P\bigg[\sup_{0\leq t\leq T}\bigg|\frac1{Y_t}\bigg|^{\tilde p}\bigg(\int_0^T\big|Z^1_s\big|^2\mathrm{d}s\bigg)^{\tilde p/2}\bigg]\\
&\leq  \frac1{R_A^{\tilde p}}\sup_{\P\in\Pc}\bigg(\E^\P\bigg[\sup_{0\leq t\leq T}\bigg|\frac1{Y_t}\bigg|^p\bigg]\bigg)^{\frac{\tilde p}{p}}\sup_{\P\in\Pc}\bigg(\E^\P\bigg[\bigg(\int_0^T\big|Z^1_s\big|^2\mathrm{d}s\bigg)^{\frac{p\tilde p}{2(p-\tilde p)}}\bigg]\bigg)^{1-\tilde p/p}
\end{align*}
To conclude, we want to make sure that we can choose $\tilde p\in (1,\bar p)$ such that $p\tilde p/(p-\tilde p)\in(1,\bar p]$. This is equivalent to having
\[
\begin{cases}
\displaystyle \tilde p>\frac{p}{1+p}\\
\displaystyle \tilde p\leq \frac{p\bar p}{p+\bar p},
\end{cases}
\]
which is always possible since $p\bar p/(p+\bar p)>p/(1+p)\Longleftrightarrow \bar p(p-1)>0$. The same reasoning gives us the required result for $\overline Z^\mu$.
\end{proof}

\subsection{Proof of Theorem \ref{thm:main}}\label{proof:main}
Before explaining how to prove the aforementioned result, notice that the second equality in Theorem \ref{thm:main} is trivial. Indeed, in absence of limited liability, the value of the Principal is a non--increasing function of the utility obtained by the Agent. Mathematically, this translates into the fact that the dynamics of both state variables in the Principal's problem actually do not depend on $L$, so that the dependence on the associated initial value is straightforward.

\medskip
The proof of the first equality relies on arguments similar to the ones developed in \cite{cvitanic2018dynamic}, \cite{elie2019tale} and \cite{aid2018optimal}. Using Proposition \ref{prop:genbestreac} and Theorem \ref{th:genbestreac}, we know that for $\xi\in\Xi$, we have that there exists an equilibrium $(\P^\star, \mu^\star) \in \Mc^\star(\xi)$, where $\P^\star$ is such that for $\P^{\star}$--a.e. $\omega \in \Omega$ and for every $t \in [0,T]$, we have
    \begin{align*}
        \mu^{\star}_t(\omega)=(\P^\star)^{\omega}_t \circ (X_{t \wedge \cdot} )^{-1},
    \end{align*}
and
 \[
 K=0,\; \P^{\star}-\mathrm{a.s.}, \; \big(\alpha^{\star}_t,\beta^{\star}_t\big)\in\underset{(a,b)\in A\times\Sigma^{-1}(S_t)}{\mathrm{argmax}}\; F(X_{t\wedge\cdot}, Y_t, Z^1_t,S_t),\; \P^{\star}-\mathrm{a.s.},
 \]
 where $K$ is the last component of the solution $(Y,(Z^1,Z^2)^\top,K)$ of the 2BSDE
 \[
  Y_t
 =
-\mathrm{e}^{-R_A \xi}+\int_t^TF(X_s, Y_s, Z^1_s,S_s)\mathrm{d}s
 -\int_t^T Z^1_s\mathrm{d}X_s-\sigma^\circ\int_t^T Z^2_s\mathrm{d}W^\circ_s+\int_t^T\mathrm{d} K_s.
 \]
 
The main difference between contracts in $\Xi$ and $\Xi_{\rm S}$ comes from whether the process $K$ above is absolutely continuous with respect to Lebesgue measure or not. Since it is not in general, we will approximate it by a sequence of absolutely continuous ones. Fix thus some $\eps>0$, and define the absolutely continuous approximation of $K$
\[
 K^{\eps}_t
 :=
 \frac1\eps\int_{(t-\eps)^+}^t K_s\mathrm{d}s,\; t\in[0,T].
\]
Clearly, $ K^{\eps}$ is $\F^{\overline\Pc}$--predictable, non--decreasing $\overline\Pc$--q.s. and
 \begin{equation}\label{eq:kk}
 K^\eps=0,
 \; \P^\star-\mbox{a.s. for all}\
 (\P^\star,\mu^\star)\in \Mc^\star(\xi).
 \end{equation}
We next define for any $t\in[0,T]$ the process
 \begin{equation}\label{Yeps}
 Y^\eps_t
 :=
 Y_0
 -\int_0^t F(X_s, Y^\eps_s,Z^1_s,S_s)\mathrm{d}s
 +\int_0^t Z^1_s \mathrm{d}X_s+\sigma^\circ\int_0^t Z^2_s\mathrm{d}W^\circ_s
 -\int_0^t \mathrm{d}K^\eps_s,
 \end{equation}
and verify that $(Y^\eps, (Z^1,Z^2), K^{\eps})$ solves the 2BSDE with terminal condition $-\mathrm{e}^{-R_A\xi^\eps}:=Y^\eps_T$ and generator $F$. Indeed, since $K^\eps \leq K$,  $ K^{\eps}$ does satisfy the required minimality condition, which is obvious by \eqref{eq:kk}.  We also verify that $\sup_{\overline\P\in\overline\Pc}\E^{\overline\P}\big[|\mathrm{e}^{-R_A\xi^\eps}|^{p}\big]<\infty$. Thus, by \cite[Theorem 4.4]{possamai2018stochastic}, we have the estimates
 \begin{equation}\label{estimates-eps}
 \|Y^\eps\|_{\S^{\bar p}}+\big\|(Z^1,Z^2)^\top\big\|_{\H^{\bar p}}<\infty,\;
 \mbox{for}\;
 \bar p\in(1,p).
 \end{equation}
We finally observe that a probability measure $\P$ satisfies $K=0,$ $\P$--a.s. if and only if it satisfies $K^\eps=0$, $\P$--a.s. Notice then that for any $(t,\omega,x,y,z^1,z^2)\in[0,T]\times\Omega \times\R^4$, the map
\begin{equation}\label{surjective}
    \gamma\longmapsto -R_A y \Hc^\circ \bigg( x, - \dfrac{1}{R_A y} (z^1,z^2,\gamma) \bigg) - \frac12 \gamma \Big(S_s(\omega) + \big( \sigma^\circ \big)^2 \Big) -F(x,y,z^1,S_s(\omega))
    ~\mbox{is surjective on}~(0,\infty).
\end{equation}
Indeed, it is non--negative, by definition of $\Hc^\circ$ and $F$, convex, continuous on the interior of its domain, and is coercive by the boundedness of the controls.

\medskip
Let $\dot{K}^\eps$ denote the density of the absolutely continuous process $K^\eps$ with respect to the Lebesgue measure. Applying a classical measurable selection argument (the maps appearing here are continuous, and we can use the results from \citet*{benes1970existence,benes1971existence}), we may deduce the existence of an $\F^{\rm obs}$--predictable process $\Gamma^\eps$ such that
% {\color{red}
% \[
%  \dot{K}_s^{\eps}
%  =
% -R_AY_s\Hc(X_s,\widehat\mu_s,Z_s,Z^\mu_s,\Gamma_s^\eps,\widehat \alpha^\star_s)-\frac12\overline\Gamma_s^\eps S_s-F(X_s,Y^\eps_s, Z_s,S_s).
% \]}
\[
 \dot{K}_s^{\eps}
 =
-R_A Y^\eps_s \Hc^\circ \bigg( X_s,  - \dfrac{1}{R_A Y^\eps_s} (Z^1_s, Z^2_s, \Gamma_s^\eps) \bigg) - \frac12 \Gamma_s^\eps \Big( S_s + \big( \sigma^\circ \big)^2 \Big) - F(X_s,Y^\eps_s, Z^1_s,S_s).
\]

For $\dot{K}_s^{\eps}>0$, this is clear from \eqref{surjective}, and if $\dot{K}_s^{\eps}=0$, $\Gamma^\eps_s$ can be chosen arbitrarily. Substituting in \eqref{Yeps}, it follows that the following representation for $Y^\eps$ holds
% {\color{red}
% \[
%  Y^\eps_t
%  =
%  Y_0
%  +\int_0^t R_AY_s\Hc(X_s,\widehat\mu_s,Z_s,Z^\mu_s,\Gamma_s^\eps,\widehat\alpha^\star_s)\mathrm{d}s
%  +\int_0^t  Z_r\cdot \mathrm{d}X_r+\sigma^\circ\int_0^t  \overline Z^\mu_s \mathrm{d}X_s
%  +\frac 12\int_0^t \Gamma^\eps_s\mathrm{d}\langle X\rangle_s.
% \]}
\[
 Y^\eps_t
 =
 Y_0
 +\int_0^t R_A Y^\eps_s\Hc^\circ\bigg( X_s, - \dfrac{1}{R_A Y^\eps_s} (Z^1_s,Z^2_s,\Gamma_s^\eps) \bigg) \mathrm{d}s
 +\int_0^t Z^1_s \mathrm{d}X_s
 +\sigma^\circ \int_0^t Z^2_s \mathrm{d}W^\circ_s
 +\frac 12\int_0^t \Gamma^\eps_s\mathrm{d}\langle X\rangle_s.
\]

Applying It\=o's formula to $-1/R_A\log(-Y_t^\eps)$, 
\begin{align*}
    \xi^\eps_t = &- \frac1{R_A} \log(-Y_0^\eps) - \int_0^t \Hc^\circ (X_s, \overline \zeta_s)  \drm s 
    + \int_0^t Z_s \drm X_s 
    + \sigma^{\circ} \int_0^t \overline{Z}_s^{\mu} \drm W^{\circ}_s 
    + \dfrac{1}{2} \int_0^t \big( \Gamma_s + R_A |Z_s|^2 \big) \drm \langle X \rangle_s \\
    &+ \dfrac{1}{2} R_A  \big( \sigma^{\circ} \big)^2 \int_0^t  \overline Z_s^{\mu} \big(\overline Z_s^{\mu}+ 2Z_s\big)  \drm s.
\end{align*}
where
\[ 
    \overline \zeta := (Z, \overline Z^\mu, \Gamma) = - \dfrac1{R_A Y^\eps} (Z^1, Z^2, \Gamma^\eps).
\]

Define $Z^\mu$ to be any $\F^{\rm obs}$--predictable process, taking values in $\Lc$, such that $\widehat\E^{\P^\star_s}[ Z_s^\mu(\widehat X_{s\wedge\cdot})]= \overline Z^\mu_s$. Using the dynamic of $\widehat X$, we deduce that
\begin{align*}
    \xi_t^{\eps} := &\ -\frac1{R_A}\log(-Y^\eps_0)  - \int_0^t \Hc (X_s, \widehat \mu_s, \zeta_s, \alpha_s^\star)  \drm s 
    + \int_0^t Z_s \drm X_s 
    + \dfrac{1}{2} \int_0^t \big( \Gamma_s + R_A Z_s^2 \big) \drm \langle X \rangle_s
    + \int_0^t \widehat{\mathbb{E}}^{\P^\star_s} \big[ Z_s^{\mu} \big(\widehat X_{s\wedge\cdot}\big) \drm \widehat X_s \big]
    \nonumber \\
    &+ \dfrac{1}{2} R_A \int_0^t  \widehat{\mathbb{E}}^{\P^\star_s} \widecheck{\mathbb{E}}^{\P^\star_s} \big[ Z_s^{\mu} \big(\widehat X_{s\wedge\cdot}\big) Z_s^{\mu} \big(\widecheck{X}_{s\wedge\cdot}\big) \drm \big\langle \widehat X, \widecheck{X} \big\rangle_s \big]
    + R_A \int_0^t Z_s \widehat{\mathbb{E}}^{\P^\star_s}  \big[Z_s^{\mu} \big(\widehat X_{s\wedge\cdot}\big) \drm \big\langle X, \widehat X \big\rangle_s \big],\; t\in[0,T],
\end{align*}
where $\zeta:= (Z, Z^\mu, \Gamma).$ This shows that the contract $\xi^\eps$ has the required dynamics \eqref{contract_form_mu}, since, at equilibrium, $\mu^\star=\widehat\mu$ and the effort $\alpha^\star = \widehat \alpha^\star$ and is unique. The fact that it belongs to $\Xi_{\rm S}$ then stems from \eqref{estimates-eps} and arguments similar to those in Lemma \ref{lemma:integZ}. We can then conclude, in view of Appendix \ref{appendix_anotherrepres} and as in the proof of \cite[Theorem 3.6]{cvitanic2018dynamic} by noting that $\xi^\varepsilon=\xi,\; \mathbb P^\star$--a.s.

\subsection{Proof of Theorem \ref{thm:Pb_Principal_general}}\label{proof:Pb_Principal_general}

Let $v$ be a solution to the PDE \eqref{eq:pde_vP}, smooth enough in the sense of Definition \ref{def:C12}, such that the condition \eqref{eq:conditions_martingale_general} is satisfied. Moreover, assume that the supremum in the PDE \eqref{eq:pde_vP} is attained for a function $\overline v^\star$, from $[0,T] \times \Pc(\R^2)$ to  $\R^3$. We need to prove that $V_0^{P} = v(0,\mu^Y_0)$.

\medskip

Applying the Chain Rule under $\Cc^{1,2}$--regularity on $v$, as a function depending on time and on the conditional law $\mu^Y$ of $Y^{\overline \zeta} = (X, L^{\overline \zeta})^\top$, we obtain:
\begin{align*}
    \drm v ( t, \mu^Y_t ) = \bigg( & \partial_t v ( t, \mu^Y_t )
    + \int \partial_{\mu^L} v ( t, \mu^Y_t ) (x, \ell) (g-f)(x) \mu^Y_t (\drm x, \drm \ell) 
    + \dfrac{1}{2} \big( \sigma^{\circ} \big)^2 \int \int \partial_{\mu^X}^2  v ( t, \mu^Y_t ) (y, \widetilde y) \mu^Y_t (\drm y) \mu^Y_t (\drm \widetilde y) \\
    &+ \dfrac{\theta}{2} \big(\sigma^{\circ}\big)^2 \int \partial_{\mu^L} v ( t, \mu^Y_t ) (x, \ell) \mu^Y_t (\drm x, \drm \ell) 
    + \dfrac{1}{2} \big( \sigma^{\circ} \big)^2 \int \partial_{x} \partial_{\mu^X} v ( t, \mu^Y_t ) (x, \ell) \mu^Y_t (\drm x, \drm \ell) \\
    &+ \dfrac{1}{2} h(\mu^Y_t, \partial_{\mu} v (t, \mu^Y_t), \partial_{y} \partial_\mu v(t,\mu^Y_t), \partial^2_{\mu} v(t,\mu^Y_t), \overline \zeta) \bigg) \drm t \\
    &+ \sigma^{\circ} \int \big( - \partial_{\mu^X} v ( t, \mu^Y_t ) (x, \ell)  + \partial_{\mu^L} v ( t, \mu^Y_t ) (x, \ell) \big( Z_t + \overline Z_t^{\mu} \big) \big) \mu^Y_t (\drm x, \drm \ell) \drm W_t^{\circ}.
\end{align*}

By assumption, $v$ is solution to the HJB equation \eqref{eq:pde_vP}:
\begin{align*}
    0 = &\ \partial_t v ( t, \mu^Y_t )
    + \dfrac{1}{2} \big( \sigma^{\circ} \big)^2 \int \int \partial_{\mu^X}^2  v ( t, \mu^Y_t ) \big(x, \ell, \widetilde x, \widetilde \ell \big) \mu^Y_t (\drm x, \drm \ell) \mu^Y_t (\drm \widetilde x, \drm \widetilde \ell)
    + \int \partial_{\mu^L} v ( t, \mu^Y_t ) (x, \ell) (g-f)(x) \mu^Y_t (\drm x, \drm \ell) \nonumber \\
    &+ \dfrac{\theta}{2} \big(\sigma^{\circ}\big)^2 \int \partial_{\mu^L} v ( t, \mu^Y_t ) (x, \ell) \mu^Y_t (\drm x, \drm \ell)
    + \dfrac{1}{2} \big( \sigma^{\circ} \big)^2 \int \partial_{x} \partial_{\mu^X} v ( t, \mu^Y_t ) (x, \ell) \mu^Y_t (\drm x, \drm \ell) \nonumber \\
    &+ \dfrac{1}{2} \sup_{\overline v \in \R^3} h(\mu^Y_t, \partial_{\mu} v (t, \mu^Y_t), \partial_{y} \partial_\mu v(t,\mu^Y_t), \partial^2_{\mu} v(t,\mu^Y_t), \overline v),
\end{align*}
therefore we obtain
\begin{align*}
    v (T, \mu^Y_T ) = &\ v(0,\mu^Y_0) 
    + \dfrac{1}{2} \int_0^T \Big( 
    h(\mu^Y_t, \partial_{\mu} v, \partial_{y} \partial_\mu v, \partial^2_{\mu} v, \overline \zeta) 
    -  \sup_{\overline{\zeta} \in \R^3} h(\mu^Y_t, \partial_{\mu} v, \partial_{y} \partial_\mu v, \partial^2_{\mu} v, \overline v)
    \Big) \drm t \\
    &+ \sigma^{\circ} \int_0^T \int \big( - \partial_{\mu^X} v ( t, \mu^Y_t ) (x, \ell)  + \partial_{\mu^L} v ( t, \mu^Y_t ) (x, \ell) \big( Z_t + \overline Z_t^{\mu} \big) \big) \mu^Y_t (\drm x, \drm \ell) \drm W_t^{\circ}
\end{align*}

Under Assumption \eqref{eq:conditions_martingale_general} on the partial derivatives of $v$, the process
\begin{align*}
    \sigma^{\circ} \int_0^\cdot \int \big( - \partial_{\mu^X} v ( t, \mu^Y_t ) (x, \ell)  + \partial_{\mu^L} v ( t, \mu^Y_t ) (x, \ell) \big( Z_t + \overline Z_t^{\mu} \big) \big) \mu^Y_t (\drm x, \drm \ell) \drm W_t^{\circ}.
\end{align*}
is a $\P$--martingale, since the following quantities
\begin{align*}
    \underbrace{\E^\P \bigg[ \sup_{0\leq t\leq T} \bigg| \int_0^t \int \partial_{\mu^X} v ( s, \mu^Y_s ) (x, \ell)   \mu^Y_s (\drm x, \drm \ell) \drm W_s^\circ \bigg| \bigg] }_{=:A},\; \text{and}\; \underbrace{\E^\P \bigg[ \sup_{0\leq t\leq T} \bigg| \int_0^t \int \partial_{\mu^L} v ( s, \mu^Y_s ) (x, \ell) \big( Z_s + \overline Z_s^{\mu} \big) \mu^Y_s (\drm x, \drm \ell) \drm W_s^\circ \bigg| \bigg] }_{=:B},
\end{align*}
are finite. Indeed, on the one hand, using Burkholder--Davis--Gundy inequality, we have for some constant $C>0$, independent of $\P$,
\begin{align*}
    A&\leq C \E^\P \bigg[ \bigg(\int_0^T \bigg( \int \partial_{\mu^X} v ( t, \mu^Y_t ) (x, \ell) \mu^Y_t (\drm x, \drm \ell) \bigg)^2 \drm s\bigg)^{1/2} \bigg] <+\infty,
    % \leq &\ 2 \E^\P \bigg[ \int_0^T \bigg( \int \partial_{\mu^X} v ( t, \mu ) (x, \ell) \mu (\drm x, \drm \ell) \bigg)^2 \drm t \bigg] + 2 \E^\P \bigg[ \int_0^T \big( Z_t + \overline Z_t^{\mu} \big)^2 \bigg( \int \partial_{\mu^L} v ( t, \mu ) (x, \ell) \mu (\drm x, \drm \ell) \bigg)^2 \drm t \bigg] \\
%    \leq &\ 2 \E^\P \bigg[ \int_0^T \bigg( \int \partial_{\mu^X} v ( t, \mu ) (x, \ell) \mu (\drm x, \drm \ell) \bigg)^2 \drm t \bigg] + 2 \E^\P \bigg[  \int_0^T \bigg( \int \partial_{\mu^L} v ( t, \mu ) (x, \ell) \mu (\drm x, \drm \ell) \bigg)^4 \drm t \bigg]^{\frac12} \E^\P \bigg[ \int_0^T \big( Z_t + \overline Z_t^{\mu} \big)^4 \drm t \bigg]^{\frac12},
\end{align*}
by the first part of Assumption \eqref{eq:conditions_martingale_general}. On the other hand, using in addition H\"older's inequality, we have, for some constant $\widetilde C > 0$ independent of $\P$, and for $p^\prime>1$ given by Lemma \ref{lemma:integZ}
\begin{align*}
    B&\leq \widetilde C  \E^\P \bigg[ \bigg(\int_0^T \bigg( \int\partial_{\mu^L} v ( s, \mu^Y_s ) (x, \ell) \big( Z_s + \overline Z_s^{\mu} \big)\mu^Y_s (\drm x, \drm \ell) \bigg)^2 \drm s\bigg)^{1/2} \bigg] \\
    &\leq  2 \widetilde C  \E^\P \bigg[ \sup_{0\leq t\leq T} \bigg| \int\partial_{\mu^L} v ( t, \mu^Y_t ) (x, \ell)\mu^Y_t (\drm x, \drm \ell)\bigg|\bigg(\int_0^T  \big( Z_s^2 + (\overline Z_s^{\mu})^2 \big) \drm s\bigg)^{1/2} \bigg] \\
 &\leq  2 \widetilde C  \bigg(\E^\P \bigg[\sup_{0\leq t\leq T} \bigg| \int\partial_{\mu^L} v ( t, \mu^Y_t ) (x, \ell)\mu^Y_t (\drm x, \drm \ell)\bigg|^{\frac{p^\prime}{p^\prime-1}} \bigg]\bigg)^{1-\frac1{p^\prime}} \bigg(\E^\P \bigg[ \bigg(\int_0^T  \big( Z_s^2 + (\overline Z_s^{\mu})^2 \big) \drm s\bigg)^{p^\prime/2} \bigg]\bigg)^{\frac1{p^\prime}}<+\infty.
%    \leq &\ 2 \E^\P \bigg[ \int_0^T \bigg( \int \partial_{\mu^X} v ( t, \mu ) (x, \ell) \mu (\drm x, \drm \ell) \bigg)^2 \drm t \bigg] + 2 \E^\P \bigg[  \int_0^T \bigg( \int \partial_{\mu^L} v ( t, \mu ) (x, \ell) \mu (\drm x, \drm \ell) \bigg)^4 \drm t \bigg]^{\frac12} \E^\P \bigg[ \int_0^T \big( Z_t + \overline Z_t^{\mu} \big)^4 \drm t \bigg]^{\frac12},
\end{align*}
Indeed, the first term is finite by the second part of Assumption \eqref{eq:conditions_martingale_general}, and the second term is also finite since $\overline \zeta \in \overline \Vc$ implies that $Z$ and $\overline Z^\mu$ are in $\H^{p^\prime}$ by Lemma \ref{lemma:integZ}. Therefore we obtain 
\begin{align*}
    \E^\P \big[ v (T, \mu^Y_T ) \big] = v(0,\mu^Y_0) 
    + \int_0^T \E^\P \Big[ 
    &\ h(\mu^Y_t,\partial_{\mu} v (t,\mu^Y_t), \partial_{y} \partial_{\mu} v (t,\mu^Y_t), \partial^2_{\mu} v (t,\mu^Y_t), \overline \zeta_t) \nonumber \\
    &- \sup_{\overline{v} \in \R^3}\theta(\mu^Y_t, \partial_{\mu} v (t,\mu^Y_t), \partial_{y} \partial_{\mu} v (t,\mu^Y_t), \partial^2_{\mu} v (t,\mu^Y_t), \overline v)
    \Big] \drm t.
\end{align*}

Using the terminal condition $v^P ( T, \mu^Y_T) = U^P \Big( - \mathbb{E}^{\P_T} \Big[ L^{\overline \zeta}_T \Big] \Big)$ and noticing that for all $t \in [0,T]$,
\begin{align*}
    h(\mu^Y_t,\partial_{\mu} v (t,\mu^Y_t), \partial_{y} \partial_{\mu} v (t,\mu^Y_t), \partial^2_{\mu} v (t,\mu^Y_t), \overline \zeta_t)
    - \sup_{\overline{v} \in \R^3}\theta(\mu^Y_t, \partial_{\mu} v (t,\mu^Y_t), \partial_{y} \partial_{\mu} v (t,\mu^Y_t), \partial^2_{\mu} v (t,\mu^Y_t), \overline v) \leq 0,
\end{align*}
with equality for $\overline{\zeta}^{\star}_t := \overline{v}^\star (t, \mu_t^Y)$ by assumption, which leads to
\begin{align*}
    \E^\P \Big[ U^P \Big( - \mathbb{E}^{\P_T} \Big[ L^{\overline \zeta}_T \Big] \Big) \Big] \leq v(0, \mu^Y_0),
\end{align*}
with equality for $\overline{\zeta}^{\star}$. Therefore, $v(0, \mu^Y_0) = V^{P}$.

\subsection{Proof of Proposition \ref{prop:Pb_Principal_RA_general}}\label{proof:Pb_Principal_RA_general}

Let $u$ be a solution to the PDE \eqref{eq:pde_principal_RA_general}, smooth enough in the sense of Definition \ref{def:C12}, satisfying Condition \eqref{eq:conditions_martingale_RP}. Let $\overline \zeta^\star = (Z^{\star}, \overline{Z}^{\mu,\star}, \Gamma^{\star})$ be a process in $\overline \Vc$ such that for all $t \in [0,T]$, $\overline \zeta_t^\star := \overline v^\star \big(t, \mu^X_t \big)$ is the maximiser of $h^P$ defined by \eqref{eq:def_hP}. We define the function $v$ as:
\begin{align*}
    v(t, \mu^Y_t) = - \erm^{ R_P \big( \mathbb{E}^{\P_t} \big[L^{\overline \zeta^\star}_t \big] - u ( t, \mu^X_t ) \big) }.
\end{align*}

To prove the first point of the proposition, it is sufficient to show that the function $v$ satisfies the assumptions of Theorem \ref{thm:Pb_Principal_general}. Indeed, we will have, by the first point of the theorem, 
\begin{align*}
    - \erm^{ R_P ( \xi_0- u (0, \mu_0^X) ) } = v(0,\mu^Y_0) = V_0^P.
\end{align*}

First of all, $v$ has the same regularity as $u$, and is therefore smooth enough in the sense of Definition \ref{def:C12}. 
% Moreover, using the terminal condition on $u$, we obtain $v^P ( T, \mu) = - \erm^{ R_P \mathbb{E}^{\P_t^L} [L_T] }$ and therefore $v^P$ satisfy the terminal condition of the PDE \eqref{eq:pde_vP}. 
Moreover, we can prove that $v$ satisfies the condition \eqref{eq:conditions_martingale_general}. Indeed, since 
\begin{align*}
    \int \partial_{\mu^X} v(t, \mu^Y_t)(x, \ell) \mu^Y_t (\drm x, \drm \ell) &= - R_P v(t,\mu^Y_t) \overline u_{\mu^X} \big( t,\mu^X_t \big) \; \text{ and } \; \int \partial_{\mu^L} v(t, \mu^Y_t )(x, \ell)  \mu^Y_t (\drm x, \drm \ell) = R_P v(t,\mu^Y_t),
\end{align*}
Condition \eqref{eq:conditions_martingale_general} is equivalent here to
\begin{align*}
  \E^\P \bigg[\bigg( \int_0^T \big| \overline u_{\mu^X} \big( t,\mu^X_t \big)  v(t,\mu^Y_t) \big|^2 \drm t\bigg)^{1/2} \bigg] +  \E^\P \bigg[ \sup_{0\leq t\leq T}  \big|  v(t,\mu^Y_t) \big|^{p^\prime/(p^\prime-1)} \bigg] < + \infty.
\end{align*}

By applying H\"older's inequality twice, we have the following upper bound for the first expectation
\begin{align*}
    % \E^\P \bigg[ \bigg( \int_0^T \big| v(t,\mu) \big|^4 \drm t \bigg)^{\frac12} \bigg( \int_0^T \big| \overline u_{\mu^X} \big|^4 \drm t \bigg)^{\frac12} \bigg]
    % \leq 
    \bigg( \E^\P \bigg[  \bigg(\int_0^T \big| \overline u_{\mu^X} \big( t,\mu^X_t \big) \big|^2 \drm t\bigg)^{p^\prime/2} \bigg] \bigg)^{1/p^\prime}
    \bigg(\E^\P \bigg[ \sup_{0\leq t\leq T}  \big|  v(t,\mu^Y_t) \big|^{p^\prime /(p^\prime-1)} \bigg]  \bigg)^{1-1/p^\prime },
\end{align*}
where the first expectation is finite since $u$ satisfies Condition \eqref{eq:conditions_martingale_RP}. It remains to prove that the second one is finite. By applying H\"older's inequality,
\begin{align*}
    \E^\P \bigg[ \sup_{0\leq t\leq T}  \big|  v(t,\mu^Y_t) \big|^{p^\prime /(p^\prime -1)} \bigg]
    &= \E^\P \bigg[ \sup_{0\leq t\leq T} \erm^{ \frac{p^\prime R_P}{p^\prime -1} \big( \mathbb{E}^{\P_t} \big[ L^{\overline \zeta^\star}_t \big] - u ( t, \mu^X_t ) \big) } \bigg] \\
    &\leq  \bigg( \E^\P \bigg[ \sup_{0\leq t\leq T} \erm^{\frac{\varepsilon p^\prime R_P}{p^\prime-1} \mathbb{E}^{\P_t} \big[ L^{\overline \zeta^\star}_t \big]}  \bigg] \bigg)^{1/\varepsilon} \bigg( \E^\P \bigg[ \sup_{0\leq t\leq T} \erm^{- \frac{q^\prime p^\prime R_P}{p^\prime -1} u ( t, \mu^X_t ) } \drm t \bigg] \bigg)^{1/q^\prime},
\end{align*}
for $q^\prime = \varepsilon/(\varepsilon-1)$ and recalling that $\varepsilon = \sqrt{\overline p (p^\prime-1)/p^\prime} $. Thus, the second expectation if finite since $u$ satisfies Condition \eqref{eq:conditions_martingale_RP}. Moreover, since $\langle X\rangle$ is bounded and applying again Holder's inequality with $\varepsilon$ and $q'$, there exists some positive constant $C$ such that the first expectation has the following upper bound
\begin{align*}
    % &\ C\bigg( \E^\P \bigg[ \sup_{0\leq t\leq T} \erm^{\frac{p' p R_P}{p-1} \mathbb{E}^{\P_t^L} [\xi_t +\int_0^t g(X_s)\drm s]}  \bigg] \bigg)^{\frac{1}{p'}} \\
    % \leq &\ 
    C \bigg( \E^\P \bigg[ \sup_{0\leq t\leq T} \erm^{\frac{\varepsilon^2 p^\prime R_P}{p^\prime-1} \mathbb{E}^\P \big[ \xi^{\overline \zeta^\star}_t \big| \Fc_t^\circ \big]}  \bigg] \bigg)^{\frac{1}{\varepsilon^2}} 
    \bigg(  \E^\P \bigg[ \sup_{0\leq t\leq T} \erm^{\frac{\overline p R_P}{\varepsilon-1} \mathbb{E}^\P \big[ \int_0^t g(X_s)\drm s \big| \Fc_t^\circ \big]}  \bigg] \bigg)^{\frac{1}{\varepsilon-1}}.
\end{align*}
By noting that $\frac{\varepsilon^2 p^\prime}{p^\prime -1}=\overline p$, we deduce from \eqref{integrability:rp:xi} that the first term is finite. Since $g$ has linear growth and $X$ has bounded drift and volatility, the second term is also finite.

\medskip

To apply Theorem \ref{thm:Pb_Principal_general}, it remains to prove that $v$ is a solution to the PDE \eqref{eq:pde_vP} and that $\overline \zeta^\star$ satisfies the optimality condition on $h$. By computing the partial derivatives of $v$ in term of $u$,
% \begin{align*}
%     \partial_t v ( t, \mu ) &= - R_P v ( t, \mu ) \partial_t u ( t, \mu^X ), \\
%     \partial_{\mu^X} v ( t, \mu ) (x) &= - R_P v ( t, \mu ) \partial_{\mu^X} u ( t, \mu^X ) (x), \\ 
%     \partial_{\mu^L} v ( t, \mu ) (\ell) &= R_P v ( t, \mu ), \\ 
%     \partial_x \partial_{\mu^X} v ( t, \mu ) (x) &= - R_P v ( t, \mu ) \partial_x \partial_{\mu^X} u ( t, \mu^X ) (x), \\
%     \partial_{\ell} \partial_{\mu^L} v ( t, \mu ) (\ell) &=
%     \partial_{x} \partial_{\mu^L} v ( t, \mu ) (\ell) =
%     \partial_{\ell} \partial_{\mu^{X}} v ( t, \mu ) (\ell)  = 0,  \\
%     \partial_{\mu^X}^2  v ( t, \mu )  ( x, \widetilde{x}) &= - R_P v ( t, \mu ) \Big( \partial_{\mu^X}^2 u ( t, \mu^X ) ( x, \widetilde{x}) - R_P \partial_{\mu^X} u ( t, \mu ) ( x) \partial_{\mu^X} u ( t, \mu ) (\widetilde{x}) \Big), \\
%     \partial_{\mu^L}^2  v ( t, \mu )  \big( \ell, \widetilde{\ell} \big) &= R_P^2 v ( t, \mu ), \\
%     \partial_{\mu^X \mu^L}^2  v ( t, \mu )  \big( x, \ell, \widetilde{x}, \widetilde{\ell} \big) &= - R^2_P v ( t, \mu ) \partial_{\mu^X} u ( t, \mu^X ) ( x) ,
% \end{align*}
the function $h$ in the PDE \eqref{eq:pde_vP} can be rewritten as:
\begin{align*}
    h(\mu_t, \partial_{\mu} v, \partial_{y} \partial_{\mu} v, \partial^2_{\mu} v, \overline v) = 
    R_P v \Big(
    & \Sigma^{\star} (\gamma) \big(\theta+ R_A z^2 - \overline{u}_{x,\mu^X}  \big) + c_{\beta}^\star (\gamma)
    + c_{\alpha}^\star (z)
    + 2 \overline{\rho} \big( z^{-} \wedge A_{\textnormal{max}} \big)  \overline{u}_{\mu^X}  \\
    &+ \big(\sigma^{\circ}\big)^2 (R_A + R_P) \big(z + \overline z^{\mu} \big)^2
    - 2 R_P \big( \sigma^{\circ} \big)^2 \big( z + \overline z^{\mu} \big) \overline{u}_{\mu^X} 
    \Big).
\end{align*}

Noticing that $v < 0$, 
\begin{align*}
    \sup_{\overline v \in \R^3} h(\mu^Y_t, \partial_{\mu} v(t, \mu^Y_t), \partial_{y} \partial_{\mu} v(t, \mu^Y_t), \partial^2_{\mu} v(t, \mu^Y_t), \overline v) =  R_P v(t, \mu_t^Y)  \inf_{\overline v \in \R^3} h^P(\mu_t^X, \overline{u}_{\mu^X} \big(t, \mu^X_t \big), \overline{u}_{x, \mu^X} \big(t, \mu^X_t \big), \overline v),
\end{align*}
with $h^P$ defined by \eqref{eq:def_hP}. Since the infimum of $h^P$ on $\overline{z}^{\mu}$ is attained in 
\begin{align*}
    \overline{z}^{\mu,\star} \big(t, \mu^X_t \big) = - z + \dfrac{R_P}{R_A + R_P} \overline{u}_{\mu^X} \big(t, \mu^X_t \big),
\end{align*}
point $(iii)$ of the proposition has been proven and point $(v)$ is an easy computation for simple contracts in Definition \ref{def:simple_contracts} with the optimal payment rate $\overline{Z}_t^{\mu,\star} = \overline{z}^{\mu,\star} \big(t, \mu^X_t \big)$. Moreover, we obtain
\begin{align*}
    \inf_{\overline v \in \R^3} h^P(\mu_t^X, \overline{u}_{\mu^X}, \overline{u}_{x, \mu^X}, \overline v) = 
    &- \big(\sigma^{\circ}\big)^2 \dfrac{R_P^2}{R_A + R_P} \big( \overline{u}_{\mu^X} \big)^2 
    - \overline{\rho} \big( \overline{u}_{\mu^X} \big)^2
    + \inf_{z \in \R} \Big\{ 
    F_0 (q(z,\overline{u}_{x,\mu^X}))
    + \overline{\rho} \big( (z^{-} \wedge A_{\textnormal{max}}) + \overline{u}_{\mu^X} \big)^2
    \Big\}.
\end{align*}

Therefore, the function $v$ is solution to the PDE \eqref{eq:pde_vP} if
\begin{align*}
    0 = R_P v \bigg( 
    &- \partial_t u \big( t, \mu_t^X \big) 
    + \int (g-f)(x) \mu_t^X (\drm x) 
    + \dfrac{\theta}{2} \big(\sigma^{\circ}\big)^2 
    + \dfrac{1}{2} \bigg( \big( \sigma^{\circ} \big)^2 \overline R - \overline{\rho} \bigg)  \big( \overline{u}_{\mu^X} \big( t, \mu_t^X \big) \big)^2 \\
    &- \dfrac{1}{2} \big( \sigma^{\circ} \big)^2 \big( \overline{u}_{x,\mu^X} + \overline{u}_{\mu^X, \mu^X} \big) \big( t, \mu_t^X \big)
    + \dfrac{1}{2} \inf_{z \in \R} \Big\{ 
    F_0 \big( q \big(z,\overline{u}_{x,\mu^X} \big( t, \mu_t^X \big) \big) \big)
    + \overline{\rho} \big( (z^{-} \wedge A_{\textnormal{max}}) + \overline{u}_{\mu^X} \big( t, \mu_t^X \big) \big)^2
    \Big\} \bigg),
\end{align*}
and this equality is true since $u$ is solution to the PDE \eqref{eq:pde_principal_RA_general}.

\medskip

Consider now the following minimisation problem:
\begin{align}\label{eq:min_pb_general}
    \inf_{z \in \R} \Big\{ 
    F_0 \big( q \big(z,\overline{u}_{x,\mu^X} \big( t, \mu_t^X \big) \big) \big)
    + \overline{\rho} \big( (z^{-} \wedge A_{\textnormal{max}}) + \overline{u}_{\mu^X} \big( t, \mu_t^X \big) \big)^2
    \Big\}.
\end{align}

As noticed in \cite[Lemma 4.1]{aid2018optimal}, the function $F_0$ is non--decreasing, and the infimum is reached for $\gamma^{\star} = - q$, which proves the point $(iv)$ of the Proposition, and we have
\begin{align*}
    F_0 (q) = q \Sigma^{\star} ( -q ) + c_{\beta}^{\star} (-q) = \sum_{k=1}^d \dfrac{\big( \sigma^k \big)^2}{\lambda^k} \bigg[ \lambda^k \mathds{1}_{\lambda^k q \leq 1} + \dfrac{1}{\eta^k} \bigg( \big( 1 + \eta^k \big) \big( \lambda^k q \big)^{\frac{\eta^k}{1 + \eta^k}} -1 \bigg)  \mathds{1}_{\lambda^k q > 1} \bigg].
\end{align*}

For $z \geq 0$, the minimisation problem~(\ref{eq:min_pb_general}) is equal to
$\overline{\rho} \big( \overline{u}_{\mu^X} \big( t, \mu_t^X \big) \big)^2 + \inf_{z \geq 0} \big\{ F_0 \big( q \big( z,\overline{u}_{x,\mu^X} \big( t, \mu_t^X \big) \big) \big) \big\}$, and since $F_0$ is non-decreasing, its minimum is attained on the minimum of $q \big( z, \overline{u}_{x,\mu^X} \big( t, \mu_t^X \big) \big)$, for $z=0$. Therefore, \eqref{eq:min_pb_general} is equal to $\overline{\rho} \big( \overline{u}^P_{\mu^X} \big( t, \mu_t^X \big) \big)^2 + F_0 \big(\theta- \overline{u}_{x, \mu^X} \big( t, \mu_t^X \big) \big)$.

\medskip

On the other hand, for $z \leq 0$, the minimisation problem~(\ref{eq:min_pb_general}) is equal to
\begin{align*}
    \inf_{z \leq 0} \Big\{ 
    F_0 \big( q \big( z,\overline{u}_{x,\mu^X} \big( t, \mu_t^X \big) \big) \big)
    + \overline{\rho} \big( (-z \wedge A_{\textnormal{max}}) + \overline{u}_{\mu^X} \big( t, \mu_t^X \big) \big)^2
    \Big\},
\end{align*}
then, if $\overline{u}_{\mu^X} \big( t, \mu_t^X \big) \geq 0$, the infimum is reached on $z=0$. Otherwise, the infimum lie between $\overline{u}_{\mu^X} \big( t, \mu_t^X \big)  \vee - A_{\textnormal{max}}$ and $0$.

\medskip

To sum up, the optimal process $Z^\star$ is defined for $t\in[0,T]$ by $Z_t^\star = z^\star \big(t, \mu^X_t \big)$ which is the minimiser $z^{\star}$ of~(\ref{eq:min_pb_general}) and satisfies
\begin{align*}
    z^{\star} \big(t, \mu^X_t \big) =0, \text{ when } \; \overline{u}_{\mu^X}\big(t, \mu^X_t \big) \geq 0,
    \text{ and } \; z^{\star} \big(t, \mu^X_t \big) \in \big[\overline{u}_{\mu^X} \big(t, \mu^X_t \big) \vee - A_{\textnormal{max}}, 0 \big] \; \text{ when } \; \overline{u}_{\mu^X} \big(t, \mu^X_t \big) \leq 0,
\end{align*}
thus the point $(ii)$ has been proven. 

\medskip

To conclude the proof, it is sufficient to notice that $\overline \zeta^\star$ defined by the triple $(Z^\star, Z^{\mu, \star}, \Gamma^\star)$ satisfies the optimality condition and apply the Theorem \ref{thm:Pb_Principal_general} to the function $v$.

\section{Reservation utility of the consumer}\label{sec:reservation}

The contract $\xi$ offered by the Principal has to satisfy the participation constraint $V^A (\xi) \geq R_0$, where $R_0$ is defined as the expected utility of the consumer without contract:
\begin{align*}
    R_0 := \sup_{\P \in \Pc} \E^\P \bigg[- \exp \bigg( -R_A \int_0^T \big( f(X_s)-c \big( \nu^\P_s \big)  \big) \drm s \bigg) \bigg]. 
\end{align*}
The underlying idea is that consumers will refuse the contract if it provides them with a utility level which is lower than the one they could attain by themselves without any contract. The corresponding value can then be obtained through a standard HJB equation PDE, which we can solve explicitly when the function $f$ is linear. Apart from the terms depending on the common noise, the same results as in \cite{aid2018optimal} are obtained. {\color{black} Nevertheless, we choose to detail here the proofs of the various results of the section, to highlight the impact of the common noise.}

\begin{proposition}[Consumer's reservation utility]\label{prop:RU}
Assume that $f$ has linear growth. Then the following holds.
    \begin{enumerate}[label=$(\roman*)$] 
    \item The consumer's reservation utility is given by $R_0 = - \erm^{- R_A \psi(0,X_0)}$, where the corresponding certainty equivalent $\psi$ is a solution $($in the viscosity sense$)$ of the {\rm HJB} equation
    \begin{align}\label{eq:pdeCE}
    \begin{cases}
        \displaystyle  0 = \partial_t \psi(t,x)+ f(x) - \dfrac{1}{2} \Big(  c^{\star}_{\beta}(\gamma^{0}(t,x)) - \gamma^{0}(t,x) \Sigma^{\star}(\gamma^{0}(t,x)) - \gamma^{0}(t,x) \big(\sigma^{\circ} \big)^2 \Big), \; (t,x) \in [0,T) \times \R, \\
        \psi(T,x) = 0, \; x \in \R,
    \end{cases}
    \end{align}
    where $\gamma^{0}:= \partial_{xx}^2 \psi - R_A (\partial_x \psi )^2$.
            \item Assume that {\rm PDE} \eqref{eq:pdeCE} has a $\Cc^{1,2}$ solution $\psi$ such that for any $\P\in\Pc$ the following condition is satisfied
    \begin{align}\label{eq:condition_martingale_CE}
        \E^\P \bigg[ \int_0^T  \erm^{-2 R_A \psi (t,X_t) } \big( \partial_x \psi (t,X_t) \big)^2  \drm t\bigg]<+\infty.
    \end{align}
    Define the feedback control
    \begin{align}\label{eq:feedback_CE}
        \alpha_t^{0} := 0 \; \text{ and } \beta_t^{0,k} := b^{k,\star} \big( \gamma^{0}(t,X_t) \big), \; k=1,\dots,d, \; t \in [0,T].
    \end{align} 
    Then, an optimal effort of the consumer is defined by the feedback control \eqref{eq:feedback_CE}.    \end{enumerate}
\end{proposition}

\begin{proof}
$(i)$ Since the function $f$ is non--decreasing, the consumer has no reason to make an effort on the drift of his consumption deviation, as no compensation is offered for this costly effort. More rigorously, the comparison theorem for SDEs with the same diffusion coefficient, see for instance \citet*[Corollary 3.1]{peng2006necessary}, as well as the fact that $f$ is non--decreasing imply immediately that the supremum in the definition of $R_0$ can only be attained for efforts of the form $\nu^\P=(0.\beta^\P)\in\U$. Notice however that, due to his risk aversion, he might have interest to make effort on the volatility. Denoting by $R(t, X_t)$ the dynamic version of the reservation utility, satisfying
    \begin{align*}
        R(0, X_0) = R_0 \; \text{ and } R(T, X_T) = -1; 
    \end{align*}
    and using standard stochastic control theory, we obtain the following HJB equation  for the function $R$:
    \begin{align*}
        0 = \partial_t R(t,x) - R_A R(t,x) f(x) + \dfrac12 \big( \sigma^{\circ} \big)^2 \partial_{xx}^2 R(t,x) + \dfrac{1}{2} \sup_{b\in B} \big\{ R(t,x) R_A  c_{\beta}(b) + \partial_{xx}^2 R(t,x) \Sigma(b) \big\}.
    \end{align*}
    
    The optimal effort on the volatility without contract is thus $b^{k,\star} (\gamma^{0})$ where
    \begin{align*}
        \gamma^{0}(t,x) := - \dfrac{\partial_{xx}^2 R(t,x)}{R(t,x) R_A}.
    \end{align*}
    
    Using the same notations as in the previous subsection, we obtain
    \begin{align*}
        0 = \partial_t R(t,x) - R_A R(t,x) f(x) + \dfrac{1}{2}  R(t,x) R_A \Big(  c^{\star}_{\beta}(\gamma^{0}(t,x)) - \gamma^{0}(t,x) \Sigma^{\star}(\gamma^{0}(t,x)) - \gamma^{0}(t,x) \big(\sigma^{\circ} \big)^2 \Big).
    \end{align*}
    
    The solution to the previous PDE is non--positive and can thus be written under the form $R(t,x) = - \erm^{- R_A \psi(t,x)}$ where the function $\psi$ is the certainty equivalent function, satisfying the PDE \eqref{eq:pdeCE}. 
    
    \medskip
$(ii)$ Let $\psi$ be a $\Cc^{1,2}$--solution to the PDE \eqref{eq:pdeCE}. We then can apply It\=o's formula to the function $R_0(t,x) := - \erm^{-R_A \psi (t,x)}$ under an arbitrary $\P\in\Pc$
    \begin{align*}
        \drm R_0 (s,X_s) = &\ \partial_t R_0 (s,X_s) \drm s 
        - \alpha^\P_s \cdot \mathbf{1}_d \partial_x R_0 (s,X_s) \drm s 
        + \partial_x R_0 (s,X_s) \big( \sigma(\beta_s^\P) \cdot \drm W_s + \sigma^\circ \drm W^\circ_s \big) \\
        &+ \dfrac12 \partial_{xx}^2 R_0 (s,X_s) \big( \Sigma(\beta_s^\P) + \big( \sigma^\circ \big)^2 \big) \drm s.
    \end{align*}
    
Denoting by $\big( M_t^\P \big)_{t \geq0}$ the process $M_t^\P = \erm^{ R_A \int_0^t ( \frac12 c_{\beta} (\beta_s^\P) - f(X_s) ) \drm s } R_0(t,X_t),$ $t\in[0,T]$, we obtain, again by It\=o's formula
    \begin{align*}
        M_t^\P = & M_0^\P + \int_0^t \erm^{ R_A \int_0^s ( \frac12 c_{\beta} (\beta^\P_u) - f(X_u) ) \drm u } \bigg( \partial_t R_0 (s,X_s) 
        - \alpha^\P_s \cdot \mathbf{1}_d \partial_x R_0 (s,X_s)
        + \dfrac12 \partial_{xx}^2 R_0 (s,X_s) \big( \Sigma(\beta^\P_s) + \big( \sigma^\circ \big)^2 \big) \bigg) \drm s \\
        &+ \int_0^t M_s^\P R_A (c(\nu_s^\P) - f(X_s))  \drm s
        + \int_0^t \erm^{ R_A \int_0^T ( \frac12 c_{\beta} (\beta^\P_s) - f(X_s) ) \drm s } \partial_x R_0 (s,X_s) \big( \sigma(\beta^\P_s) \cdot \drm W_s + \sigma^\circ \drm W^\circ_s \big).
    \end{align*}
    
Replacing by the derivatives of $\psi$, we obtain
    \begin{align*}
        &\ M_t^\P = M_0^\P + \int_0^t R_A M_s^\P h^\psi(X_s, \partial_x \psi, \partial_{xx}^2 \psi, \alpha^\P_s, \beta^\P_s) \drm s
        - \int_0^t R_A M_s^\P \partial_x \psi \big( \sigma(\beta_s) \cdot \drm W_s + \sigma^\circ \drm W^\circ_s \big), \\
        \text{where } & h^\psi(x, \partial_x \psi, \partial_{xx}^2 \psi, a, b) =  - \partial_t \psi
        + c(\nu_s) - f(x)
        + a \cdot \mathbf{1}_d \partial_x \psi
        + \dfrac12 \big( R_A \big( \partial_{x} \psi \big)^2 - \partial_{xx}^2 \psi \big)  \big( \Sigma(b) + \big( \sigma^\circ \big)^2 \big).
    \end{align*}
    Under Condition \eqref{eq:condition_martingale_CE}, the term
   $
        \int_0^t R_A M_s^\P \partial_x \psi \big( \sigma(\beta^\P_s) \cdot \drm W_s + \sigma^\circ \drm W^\circ_s \big),
$
    is a $\P$--martingale. Indeed, recall that under any $\P\in \Pc$, the drift and the volatility of $X$ are bounded, $c_\beta$ and $\Sigma$ are continuous functions on the compact set $B$, and $f$ has linear growth. Hence, using in particular Cauchy--Schwarz inequality, \eqref{eq:condition_martingale_CE} ensures that the above stochastic integral is in $\mathbb H^1(\P)$ and is thus a $\P$--martingale. We deduce
    \begin{align*}
        &\ \E^\P [M_T^\P] = M_0^\P + \E^\P \bigg[ \int_0^t R_A M^\P_s h^\psi (X_s, \partial_x \psi, \partial^2_{xx} \psi, \alpha_s^\P, \beta^\P_s) \drm s \bigg].
    \end{align*}
    
Using the boundary condition for $\psi$, and replacing by their respective values of $M_0^\P$ and $M_T^\P$, we obtain
    \begin{align*}
        R_0(0,X_0) = \E^\P \Big[- \erm^{ R_A \int_0^T ( \frac12 c_{\beta} (\beta_s^\P) - f(X_s) ) \drm s } \Big] - \E^\P \bigg[ \int_0^t R_A M_s^\P h^\psi (X_s, \partial_x \psi, \partial_{xx}^2 \psi, \alpha_s^\P, \beta_s^\P) \drm s \bigg].
    \end{align*}
    
Using the HJB equation satisfied by $\psi$, we have
% is solution to PDE \eqref{eq:pdeCE}, we have
% \begin{align}
%     - \partial_t \psi  = f - \dfrac{1}{2} \Big(  c^{\star}_{\beta}(\gamma^{0}) - \gamma^{0} \Sigma^{\star}(\gamma^{0}) - \gamma^{0} \big(\sigma^{\circ} \big)^2 \Big).
% \end{align}
% In addition we have
\begin{align*}
    h^\psi (x, \partial_x \psi, \partial_{xx}^2 \psi, a, b) = &\ 
    \dfrac12 \big( c_\alpha(a) + 2 a \cdot \mathbf{1}_d \partial_x \psi \big)
    + \dfrac{1}{2} \big( c_\beta (b) - \gamma^{0} \Sigma(b) \big)
    - \dfrac{1}{2} c^{\star}_{\beta}(\gamma^{0}) 
    + \dfrac{1}{2} \gamma^{0} \Sigma^{\star}(\gamma^{0}),
\end{align*}
and by simple computations
\begin{align*}
    c_\beta (b) - \gamma_t^{0} \Sigma(b) \geq &\ \inf_{b' \in B} \big\{ c_\beta (b') - \gamma^{0} \Sigma(b') \big\} = - H_v \big( \gamma^{0} \big) = - \gamma^{0} \Sigma^\star \big( \gamma^{0} \big) + c^{\star}_{\beta}(\gamma^{0}), \\
    c_\alpha(a) + 2 a \cdot \mathbf{1}_d \partial_x \psi \geq &\ \inf_{a' \in A} \big\{ c_\alpha(a') + 2 a' \cdot \mathbf{1}_d \partial_x \psi \big\} = - H_d (\partial_x \psi ) = 0,
\end{align*}
since the function $\partial_x \psi $ is non--negative. Therefore, $h^\psi (x, \partial_x \psi, \partial_{xx}^2 \psi, a, b) \geq 0$, with equality for the optimal controls $(\alpha^0, \beta^0)$ defined by \eqref{eq:feedback_CE}, which leads to
$R_0(0,X_0) \geq \E^\P \big[- \exp \big(R_A \int_0^T ( \frac12 c_{\beta} (\beta_s^\P) - f(X_s) ) \drm s \big) \big],$
with equality for the optimal controls.
\end{proof}

%\begin{remark}
%    \textcolor{magenta}{In \cite{aid2018optimal}, the condition \eqref{eq:condition_martingale_CE} is satisfied since $R_A \partial_x \erm^{- R_A \psi (t,x)} $ is bounded on $[0,T_n]$ for $T_n = \inf \big\{ t > 0 : |X_t - X_0 | \geq n \big\}$ and $\sigma (\beta)$ is bounded. }
%    \todo[inline]{E: Do we have the same result ? Probably...}
%\end{remark}

\medskip

The previous result shows that, even without contracting, the consumer’s optimal behaviour exhibits a positive effort to reduce the volatility of the consumption deviation process. Of course, this is naturally due to the fact that consumers in our model are assumed to be risk--averse, and are therefore negatively impacted by the variance of their deviation. 

\medskip

We conclude this section by providing a closed--form expression for the reservation utility when $f$ is linear.

\begin{proposition}[$f$ linear]
    Let $f(x) = \kappa x$, $x \in \R$, with $\kappa \geq 0$. Then, the reservation utility of the consumer is
    \begin{align*}
        R_0 &= - \erm^{- R_A ( \kappa T X_0 + \psi_0(T) )}, \; \text{where } \psi_0 (T) := - \int_0^T H_{0} (\gamma^{0}(t)) \drm t, \\ 
        \gamma^{0}(t) &= - R_A \kappa^2 (T-t)^2, \; \text{ and } \; H_{0} (\gamma) := \dfrac{1}{2} \Big( c^{\star}_{\beta}(\gamma) - \gamma \Sigma^{\star}(\gamma) - \gamma \big(\sigma^{\circ} \big)^2 \Big).
    \end{align*}
    
    The consumer’s optimal effort on the drift and on each volatility usage are respectively:
    \begin{align*}
        \alpha^{0} := 0 \; \text{ and } \beta_t^{0,k} := 1 \wedge \left(  \lambda^{k} R_A \kappa^2 (T-t)^2 \right)^{\frac{-1}{\eta^{k} + 1}} \vee B_{\textnormal{min}}, \; k=1,\dots,d.
    \end{align*}
    thus inducing an optimal distribution $\P^{0}$ under which the deviation process follows the dynamics 
    \begin{align*}
        \drm X_t = \sigma^{\star} \big( \gamma^{0}(t) \big) \cdot \drm W_t + \sigma^{\circ} \drm W^{\circ}_t.
    \end{align*}
\end{proposition}

\begin{proof}
    By directly plugging the guess $\psi(t, x) = A(t)x + \psi_0(t)$ in the PDE \eqref{eq:pdeCE}, we obtain
    \begin{align*}
        &\ 0 = A'(t) x + \psi_0'(t) + \kappa x - \dfrac{1}{2} \Big( c^{\star}_{\beta}(- R_A A^2(t)) + R_A A^2(t) \Sigma^{\star}(- R_A A^2(t)) + R_A A^2(t) \big(\sigma^{\circ} \big)^2 \Big), \;  A(T) = \psi_0(T) = 0.
    \end{align*}
    
    This provides $A(t) = \kappa (T-t)$ and $\psi_0(t) = - \int_t^T H_{0} \big( - R_A A^2(s) \big) \drm s $ with
$
        H_{0} (\gamma) := \dfrac{1}{2} \Big( c^{\star}_{\beta}(\gamma) - \gamma \Sigma^{\star}(\gamma) - \gamma \big(\sigma^{\circ} \big)^2 \Big).
$
    
    Finally the expression of the maximiser $\beta^{0,k}$ follows from Proposition \ref{prop:RU}. Moreover, this smooth solution to the PDE satisfies the condition \eqref{eq:condition_martingale_CE}. Indeed, this condition is equivalent to having, for any $\P\in\Pc$
    \begin{align*}
 \E^\P \bigg[ \int_0^T  \erm^{-2 R_A ( \kappa (T-t) X_t + \psi_0(t) ) }  (T-t)^2 \drm t\bigg]<+\infty,
    \end{align*}
    which is true since $X$ is an It\=o process with bounded drift and volatility. We thus conclude with Proposition \ref{prop:RU} $(ii)$ that it is indeed the value function inducing the reservation utility.
\end{proof}

\begin{remark}
    One can notice that the certainty equivalent $\psi$ is a decreasing function of the correlation $\sigma^\circ$ with the common noise, since
    \begin{align*}
        \psi(t, x) 
        %&= A(t)x + \psi_0(t) = \kappa x (T-t) - \int_t^T H^{0} \big( - R_A A^2(s) \big) \drm s \\
        &= \kappa x (T-t) - \dfrac{1}{2} \int_t^T \Big( c^{\star}(- R_A A^2(s)) + R_A A^2(s) \Sigma^{\star}(- R_A A^2(s)) \Big) \drm s
        - \dfrac{1}{2} R_A \big(\sigma^{\circ} \big)^2 \int_t^T A^2(s) \drm s.
    %    &\leq \kappa x (T-t) - \dfrac{1}{2} \int_t^T \Big( c^{\star}(- R_A A^2(s)) + R_A A^2(s) \Sigma^{\star}(- R_A A^2(s)) \Big) \drm s.
    \end{align*}
    Therefore, the reservation utility of the consumer is negatively impacted by the presence of common noise.
\end{remark}

\section{Details and proof for the first--best case}\label{proof:firstbest}

Adapting the reasoning in \cite{elie2019contracting}, we are led to introduce a slight modification of the so--called Morse--Transue space on the initial canonical space $\Omega$, defined here for any $\P \in \Pc$ by
\begin{align*}
    M^{\P} (\R) := \big\{ \vartheta : \Omega \longrightarrow \R, \; \text{measurable,} \; \E^\P [ \phi (a \vartheta ) ] < + \infty, \; \text{for all} \; a \geq 0 \big\},
\end{align*}
where $\phi : \R \longrightarrow \R$ is the following Young function $\phi : x \longmapsto \exp ( |x|) - 1$. Then, $M^\P (\R)$ endowed with the norm $\| \vartheta \|_{\phi} := \inf \big\{ k > 0: \E^\P [ \phi(\vartheta/k)] \leq 1 \big\} $ is a Banach space. In this case, the set of admissible contracts $\Xi^{\textnormal{FB}}$ is defined as
\begin{align}\label{def:Xi_FB}
    \Xi^{\textnormal{FB}} := \big\{ \xi\in M^\P (\R): \Fc_T-\text{measurable, such that} \; \E^\P [\xi | \Fc_T^\circ] \in M^\P (\R), \; \forall \; \P \in \Pc \big\}.
\end{align}
Thus, for any $(\xi, \mu^X, \P) \in \Xi^{\textnormal{FB}} \times \Pc (\Cc_T) \times \Pc$, the quantities $J_0^A(\xi, \mu^X, \P)$ and $J_0^P (\xi, \P)$ are well--defined. Given the reservation utility level of the representative Agent, $R_0$, the problem of the Principal is
\begin{align*}
    V_0^{\textnormal{FB}} := \inf_{\rho > 0} \bigg\{ - \rho R_0 + \sup_{(\P, \mu^X) \in \Pc \times \Pc (\Cc_T)} \sup_{\xi \in \Xi^{\textnormal{FB}}} \big\{ J_0^P (\xi, \P) + \rho J_0^A (\xi, \mu^X, \P) \big\} \bigg\},
\end{align*}
where $\rho > 0$ is the Lagrange multiplier associated to the participation constraint. We first maximise this utility with respect to $\xi$. Let us consider, for any probability $(\P, \mu^X) \in \Pc \times \Pc (\Cc_T)$, the following map $\Xi^{\P} : \Xi^{\textnormal{FB}} \longrightarrow \R$ defined by
\begin{align*}
    \Xi^{\P} (\xi) := \E^{\P} \bigg[ U^P \bigg( - \mathbb{E}^{\P} \bigg[\xi + \int_0^T g ( X_s ) \mathrm{d} s + \dfrac{\theta}{2} \int_0^T \drm \langle X \rangle_s \bigg| \Fc^\circ_T \bigg] \bigg) 
    + \rho U^A \bigg( \xi - \int_0^T \big( c \big( \nu^\P_s \big) - f \big( X_s \big) \big) \mathrm{d} s \bigg) \bigg]. 
\end{align*}

Recall that the representative Agent is risk--averse, with a risk--aversion parameter $R_A$. We can consider both the cases of a risk--averse or risk--neutral Principal. To simplify the notations, we define
\begin{align}\label{def:K_T}
    K^P_T :=  \int_0^T g ( X_s ) \mathrm{d} s + \dfrac{\theta}{2} \int_0^T \drm \langle X \rangle_s, \; K^{A,\P}_T := \int_0^T \big( c \big( \nu^\P_s \big) - f \big( X_s \big) \big) \mathrm{d} s \; \text{ and } \; K_T := K_T^{A,\P} + K_T^P.
\end{align}

We now turn to the proof of Proposition \ref{prop:Pb_principal_FB_RA} by considering a CARA risk averse principal since the risk-neutral principal case is deduced by taking $R_P=0$.\vspace{0.5em}

If the Principal is risk averse, the utility of the Principal is defined as $U^P(x) = - \erm^{-R_P x}$. Thus we obtain
\begin{align*}
    \Xi^{\P} (\xi) := \E^{\P} \big[ - \exp  \big(R_P \mathbb{E}^\P \big[\xi + K^P_T \big| \Fc_T^\circ \big] \big)
    - \rho \exp \big( - R_A \xi + R_A K^{A,\P}_T \big) \big]. 
\end{align*}

For any $\vartheta \in \Xi^{\textnormal{FB}}$ and $\varepsilon \geq 0$,
\begin{align*}
    \dfrac{1}{\varepsilon} \big( \Xi^{\P} (\xi + \varepsilon  \vartheta) - \Xi^\P (\xi) \big) = 
    &\ \dfrac{1}{\varepsilon} \E^{\P} \big[ - \exp  \big(R_P \mathbb{E}^{\P} \big[\xi + \varepsilon \vartheta+ K^P_T \big| \Fc^\circ_T \big] \big)
    + \exp  \big(R_P \mathbb{E}^{\P} \big[\xi + K^P_T \big| \Fc^\circ_T \big] \big) \big] \\
    &+ \dfrac{1}{\varepsilon} \E^{\P} \big[ - \rho \exp \big( - R_A \big( \xi +\varepsilon  \vartheta\big) + R_A K_T^{A,\P} \big)
    + \rho \exp \big( - R_A \xi + R_A K_T^{A,\P} \big) \big] \\
    = &\ \dfrac{1}{\varepsilon} \E^{\P} \Big[ \erm^{R_P \mathbb{E}^{\P} [\xi + K^P_T | \Fc^\circ_T ]} \big( 1- \erm^{R_P \varepsilon \mathbb{E}^{\P} [\vartheta| \Fc^\circ_T]} \big) \Big]
    + \dfrac{1}{\varepsilon} \E^{\P} \Big[ \rho \erm^{- R_A \xi + R_A K_T^{A,\P}} \big(1- \erm^{ - R_A \varepsilon  \vartheta} \big)  \Big].
\end{align*}

Therefore, letting $\varepsilon \rightarrow 0$, the Gâteaux derivative is given by
\begin{align*}
    D \Xi^\P (\xi) [\vartheta] = &\ \E^{\P} \Big[ - R_P \mathbb{E}^{\P} [\vartheta| \Fc^\circ_T] \erm^{R_P \mathbb{E}^{\P} [\xi + K^P_T | \Fc^\circ_T ]} 
    + \rho \vartheta R_A \erm^{- R_A \xi + R_A K_T^{A,\P}} \Big] \\
    = &\ \E^{\P} \Big[  \mathbb{E}^{\P} \big[- R_P\vartheta\erm^{R_P \mathbb{E}^{\P} [\xi + K^P_T | \Fc^\circ_T ]} \big| \Fc^\circ_T \big]  
    + \rho \vartheta R_A \erm^{- R_A \xi + R_A K_T^{A,\P}} \Big]
\end{align*}

Conditioning by $\Fc^\circ_T$, we obtain:
\begin{align*}
    D \Xi^\P (\xi) [\vartheta] = &\ \E^{\P} \Big[ \mathbb{E}^{\P} \big[- R_P \vartheta \erm^{R_P \mathbb{E}^{\P} [\xi + K^P_T | \Fc^\circ_T ]} \big| \Fc^\circ_T \big]  
    + \E^{\P} \big[ \rho h R_A \erm^{- R_A \xi + R_A K_T^{A,\P}} \big| \Fc^\circ_T \big] \Big] \\
    =&\ \E^{\P} \Big[ - R_P \vartheta  \erm^{R_P \mathbb{E}^{\P} [\xi + K^P_T | \Fc^\circ_T ]} 
    + \rho \vartheta R_A \erm^{- R_A \xi + R_A K_T^{A,\P}} \Big].
\end{align*}

For any $\P \in \Pc$, let us introduce $\xi^\star (\P)$ defined by
\begin{align}\label{eq:xi_P}
    \xi^\star (\P) = 
    - \dfrac{1}{R_A + R_P} \ln \bigg( \dfrac{R_P}{\rho R_A} \bigg) 
    + K_T^A
    - \dfrac{R_P}{R_A + R_P} \mathbb{E}^{\P} [ K_T | \Fc^\circ_T ].
\end{align}
so that
\begin{align*}
    \mathbb{E}^{\P} [\xi^\star (\P) | \Fc^\circ_T ] := 
    - \dfrac{1}{R_A + R_P} \ln \bigg( \dfrac{R_P}{\rho R_A} \bigg) 
    + \dfrac{R_A}{R_A + R_P} \mathbb{E}^{\P} [ K_T^A | \Fc^\circ_T ]
    - \dfrac{R_P}{R_A + R_P}  \mathbb{E}^{\P} [K^P_T | \Fc^\circ_T ],
\end{align*}
and
\begin{align*}
    \xi^\star (\P) := - \dfrac{1}{R_A} \ln \bigg( \dfrac{R_P}{\rho R_A} \bigg) + K_T^A - \dfrac{R_P}{R_A} \big( \mathbb{E}^{\P} [\xi ^\star (\P)| \Fc^\circ_T ] + \mathbb{E}^{\P} [K^P_T | \Fc^\circ_T ] \big).
\end{align*}

Then, for any $\vartheta \in \Xi^{\textnormal{FB}}$, we have $D \Xi^\P \big( \xi^\star(\P) \big) [\vartheta] = 0$ and $\Xi^\P$ is strictly concave function, so that this $\xi^\star (\P)$ attains the minimum of $\Xi^\P$ and is therefore optimal. Plugging these expressions back into the Principal  and recalling that $\overline{R}$ is defined by $1/\overline R := 1/R_A + 1/R_P$, the value function of the Principal in the first best case rewrites:
\begin{align*}
    V_0^{\textnormal{FB}} &= \inf_{\rho > 0} \bigg\{ \rho \bigg( - R_0 + \dfrac{R_A + R_P}{R_P}
     \exp  \bigg( \dfrac{R_A}{R_A + R_P} \ln \bigg( \dfrac{R_P}{\rho R_A} \bigg) \bigg) V_0^{\overline R} \bigg)  \bigg\},\;   \text{where} \; V_0^{\overline R}:= 
    \sup_{\P \in \Pc} \E^{\P} \bigg[ 
    - \exp  \big( \overline R \; \mathbb{E}^{\P} \big[ K_T \big| \Fc^\circ_T \big] \big) \bigg].
\end{align*}
Notice that $V_0^{\overline R}$ does not depend on $\rho$. Then direct calculations lead to the optimal Lagrange multiplier and first--best value function:
\begin{align*}
    \rho^\star = \dfrac{R_P}{R_A} \bigg( \dfrac{V_0^{\overline R}}{R_0} \bigg)^{1 + \frac{R_P}{R_A}} \; \text{and}\; 
    V_0^{\textnormal{FB}} &= R_0 \bigg( \dfrac{V_0^{\overline R}}{R_0} \bigg)^{1 + \frac{R_P}{R_A}}.
\end{align*}
% \todo[inline]{D: you can shorten everything after that: we have similar results in the paper before, wo we can skip most of it and get directly to the results.}

Using the same tools as in Section \ref{sec:principal_problem}, we can easily prove the points $(i)$, $(ii)$ and $(iv)$ of Proposition \ref{prop:Pb_principal_FB_RA}. The point $(iii)$ is a straightforward computation of the contract defined by \eqref{eq:xi_P}.

\end{appendices}

\small
\bibliography{bibliographyDylan.bib}

\end{document}